\documentclass{memoir}
\usepackage{mathptmx}
\usepackage{helvet}
\usepackage{courier}
\usepackage{type1cm}

\usepackage{makeidx}         
\usepackage{graphicx}        
\usepackage{multicol}        
\usepackage[bottom]{footmisc}

\usepackage{amssymb,amsmath,amsfonts}
\usepackage{enumerate}

\usepackage{bm}

\usepackage[frenchb,english]{babel}
\usepackage{dsfont}
\usepackage{changepage}
\usepackage[T1]{fontenc}
\usepackage[utf8]{inputenc}

\usepackage{amsmath}

\usepackage{amsthm}

\usepackage{bookmark}

\usepackage{numprint}

\usepackage[left]{lineno}

\usepackage{rotating}

\newcommand{\N} {{{\mathbb N}}}
\newcommand{\M} {{ \mathcal M }}

\newcommand{\bpf}{\begin{proof}}
\newcommand{\epf}{\end{proof}}

\newcommand{\ee}{{\rm e}}

\newcommand{\E}{\mathbb{E}}
\newcommand{\ind}{\mathbf{1}}
\newcommand{\Co}{\mathcal{C}}
\newcommand{\D}{\mathbb{D}}
\newcommand{\Z}{\mathbb{Z}}
\newcommand{\R}{\mathbb{R}}

\newcommand{\I}{\textsc{I}}

\newcommand{\Card}{\mathrm{Card}}

\renewcommand{\P}{\mathbb{P}}

\newcommand{\Id}{\mathrm{Id}}
\newcommand{\supp}{\mathrm{supp}}
\newcommand{\ER}{\mathrm{ER}}

\newcommand\cro[1]{\langle #1 \rangle}

\def\cS{\text{\textsc{s}}}
\def\cI{\text{\textsc{i}}}
\def\cRr{\text{\textsc{r}}}
\def\cIS{\text{\textsc{is}}}
\def\cRSr{\text{\textsc{rs}}}

\def\bpI{\bar{p}^{\text{\textsc{i}}}}
\def\bpS{\bar{p}^{\text{\textsc{s}}}}
\def\bpR{\bar{p}^{\text{\textsc{r}}}}
\def\d{\text{ d}}
\def\muS{\mu^\cS}
\def\muSI{\mu^\cIS}
\def\muSR{\mu^\cRSr}
\def\muSn{\mu^{\textsc{n},\cS}}
\def\muSIn{\mu^{\textsc{n},\cIS}}
\def\muSRn{\mu^{\textsc{n},\cRSr}}
\def\muISn{\mu^{\textsc{n},\cIS}}
\def\muRSn{\mu^{\textsc{n},\cRSr}}
\def\muSnn{\mu^{\textsc{n},\cS}}
\def\muSInn{\mu^{\textsc{n},\cIS}}
\def\muSRnn{\mu^{\textsc{n},\cRSr}}
\def\muISnn{\mu^{\textsc{n},\cIS}}
\def\muRSnn{\mu^{\textsc{n},\cRSr}}
\def\bmuS{\bar{\mu}^\cS}
\def\bmuSI{\bar{\mu}^\cIS}
\def\bmuSR{\bar{\mu}^\cRSr}
\def\bmuIS{\bar{\mu}^\cIS}
\def\bmuRS{\bar{\mu}^\cRSr}
\def\NIS{N^{\cI\cS}}
\def\NS{N^{\cS}}
\def\NRS{N^{\cRr\cS}}
\def\bNSI{\bar N^{\cI\cS}}
\def\bNIS{\bar N^{\cI\cS}}
\def\bNS{\bar N^{\cS}}
\def\bNRS{\bar N^{\cRr\cS}}

\def\NISn{N^{\textsc{n},\cI\cS}}

\def\NISnn{N^{\textsc{n},\cI\cS}}
\def\NSnn{N^{\textsc{n},\cS}}
\def\NRSnn{N^{\textsc{n},\cRr\cS}}
\def\In{I^{\textsc{n}}}
\def\Sn{S^{\textsc{n}}}
\def\Rn{R^{\textsc{n}}}

\def\th{^{\mbox{th}}}

\def\interior#1{\smash{\mathop{#1}\limits^{\lower1pt\hbox{$\scriptscriptstyle\circ$}}}}

\theoremstyle{plain}
\newtheorem{theorem}{Theorem}[section]
\newtheorem{lemma}[theorem]{Lemma}
\newtheorem{corollary}[theorem]{Corollary}
\newtheorem{example}[theorem]{Example}
\newtheorem{notation}[theorem]{Notation}
\newtheorem{definition}[theorem]{Definition}
\newtheorem{proposition}[theorem]{Proposition}
\newtheorem{remark}[theorem]{Remark}
\newtheorem{assumption}[theorem]{Assumption}

\allowdisplaybreaks

\numberwithin{equation}{section}
\numberwithin{theorem}{section}
\numberwithin{figure}{section}
\numberwithin{section}{chapter}
\numberwithin{table}{section}

\def\be{\begin{eqnarray}}
\def\ee{\end{eqnarray}}
\def\ben{\begin{eqnarray*}}
\def\een{\end{eqnarray*}}

\title{Stochastic epidemics in a heterogeneous community}

\author{Viet Chi Tran\footnote{LAMA, Univ Gustave Eiffel, UPEM, Univ Paris Est Creteil, CNRS, F-77447, Marne-la-Vallée, France; E-mail:
    \texttt{chi.tran@u-pem.fr}}}

\begin{document}

\maketitle

This document is the Part III of the book \textit{Stochastic Epidemic Models with Inference} edited by Tom Britton and Etienne Pardoux \cite{brittonpardoux}.

 \tableofcontents

\emph{Acknowledgements:} This research has been supported by the ``Chaire Mod\'elisation Math\'ematique et Biodiversit\'e" of Veolia Environnement-Ecole Polytechnique-Museum National d'Histoire Naturelle-Fondation X. V.C.T. also acknowledges support from Labex CEMPI (ANR-11-LABX-0007-01), GdR GeoSto 3477, ANR Project Cadence (ANR-16-CE32-0007) and ANR Project Econet (ANR-18-CE02-0010). \\


\setcounter{chapter}{0}
\setcounter{section}{0}
\setcounter{theorem}{0}

\chapter*{Introduction}
\addcontentsline{toc}{chapter}{Introduction}

Recently, network concepts have received much attention in infectious disease modelling, essentially for modeling purposes, and the reader is also referred to earlier references of Durrett \cite{durrettIII}, Newman \cite{newman_SIAMIII}, House \cite{houseIII} or Kiss et al.\ \cite{kissmillersimonIII}.
In the compartmental models presented in Part I of this volume, any infected individual can contaminate any susceptible individuals. In many public health problems, heterogeneity issues have to be taken into account, in particular some diseases such as AIDS or HCV (Hepatitis C Virus) may spread only along a social network: the network of people having sexual intercourse or of injecting drug partners. The need to take into account the network along which an epidemic spreads has been underlined by numerous papers, starting for example from \cite{danonfordhousejewelkeelingrobertsrossvernonIII,eameskeelingIII}, and more recently \cite{bansalgrenfellmeyersIII,houseIII}.\\

After introducing random networks and describing how the spread of disease can be modelled on such structures, we explain how to approximate the dynamics by deterministic differential equations when the graphs are large.
Mathematical models for epidemics on large networks are obtained by  mean-field approximation (e.g.\ \cite{durrettIII,KG99III,pastorsatorrasvespignaniIII}) or through large population approximations (e.g.\ \cite{ballnealIII,decreusefonddhersinmoyaltranIII,grahamhouseIII,barbourreinertIII,jansonluczakwindridgeIII}). They generally stipulate simple structures for the network: small worlds (e.g.\ \cite{KG99III,moorenewmanIII}), configuration models (e.g.\ \cite{kissgreenkaoIII,maylloydIII,pellisspencerhouseIII, volzIII, volzancelmeyersIII}), random intersection graphs and graphs with overlapping communities (e.g.\ \cite{brittondeijfenlageraslindholmIII,ballsirltrapmanIII,coupechouxlelargeIII})...\\

In the last section, real data from the AIDS epidemic in Cuba is studied (data from \cite{clemenconarazozarossitranIII} and that can be found in supplementary materials of this book). We show how to conduct descriptive statistical procedures. By performing clustering and simplification of the graph, we decompose it into smaller clusters where the probabilistic models of the previous sections can be used.\\



\begin{notation}\label{notation:part3} In this part, we denote by $\N$ the set of strictly positive integers and by $\Z_+$ the set $\N\cup \{0\}$.\\
For any real bounded function $f$ on
$\Z_+$, let $\| f \|_{\infty}$ denote the supremum of
$f$ on $\Z_+$. For all such $f$ and $y\in\Z_+$, we denote by $\tau_yf$ the function $x \mapsto f(x-y)$.
For all $n\in \Z_+$, $\chi^n$ is the function $x \mapsto x^n$, and in particular, $\chi \equiv \chi^1$
is the identity function, and
$\mathbf 1 \equiv \chi^0$ is the function constant equal to 1.\\
We denote by $\M_F(\Z_+)$ the set of finite measures on $\Z_+$, equipped with the topology of weak convergence. For all
 $\mu \in \M_F(\Z_+)$ and real bounded function $f$ on $\Z_+$, we write
\begin{equation}\label{notation:integrale}
\cro{\mu,f}=\sum_{k\in \Z_+}f(k)\,\mu(k),
\end{equation}where we use the notation $\mu(k)=\mu(\{k\}).$\\
For $k\in \Z_+$, we write $\delta_k$ for the Dirac measure at $k$. In particular, for any test function $f$ from $\Z_+$ to $\R$, $\langle \delta_k,f\rangle =f(k)$. \\
For a sequence $D_1,\dots D_n\in \Z_+$, if $\mu=\sum_{k=1}^n \delta_{D_k}$, then
\[\langle \mu,f\rangle=\sum_{k=1}^n f(D_k),\]implying in particular that $\langle \mu,1\rangle= n$ and $\langle \mu,\chi\rangle=\sum_{k=1}^n D_k$.
\end{notation}

\chapter{Random Graphs}

\section{Definitions}
Usually, social networks on which disease spread are very complex. It is thus convenient to model them by \textit{random} networks. We start with some definitions, and then present some common families of random networks. There is a growing literature on random networks to which we refer the reader for further developments (e.g.\ \cite{bollobas2001III,vanderhofstadIII}).

\begin{definition}A random graph $\mathcal{G}=(V,E)$ is a set of vertices $V$ and a set of edges $E\subset V\times V$. If $u$, $v\in V$ are connected in the random graph, then $(u,v)\in E$.
\end{definition}

The set of vertices of $\mathcal{G}$ is $V$, but when we will need to make precise that it is the set of vertices of $\mathcal{G}$, we will use the notation $V(\mathcal{G})$. The population size is $|V|=N$. In the sequel, we will label the vertices with integers, so that $V=\{1,\dots N\}$.

\begin{definition}\label{def:adjmatrix}The adjacency matrix of the graph $\mathcal{G}$ is a matrix $G\in \mathcal{M}_{V\times V}(\R)$ such that $\forall u,v\in V,$
\begin{align*}
 G_{uv}=1  & \mbox{ if } (u,v)\in E,\\
 G_{uv}=0 & \mbox{ if }(u,v)\notin E.
 \end{align*}
\end{definition}
If the matrix is symmetric, the graph in undirected: to any edge from $u$ to $v$ corresponds an edge from $v$ to $u$. Else, if $(u,v)\in E$ and $(v,u)\notin E$, the graph is oriented with only the directed edge from $u$ to $v$ belonging to $E$. We say that $u$ is the \textit{ego} and $v$ the \textit{alter} of the edge.\\
If we consider weighted graphs, we can generalize the entries of $G$ to real non-negative numbers.\\

In this chapter, we will focus on undirected non-weighted graphs.

\begin{definition}
The degree of a vertex $u\in V$ in the graph $\mathcal{G}$ is
\[D_u=\sum_{v\in V}G_{uv}.\]
$D_u$ hence corresponds to the number of neighbours of the vertex $u$, i.e.\  the number of the vertices of $\mathcal{G}$ that can be reached in one step starting from $u$.\\
If the graph is oriented, the above notion corresponds to the out-degree, and similarly we can define as in-degree the number of vertices of $\mathcal{G}$ that lead to $u$ in one step: \[D^{\mbox{in}}_u=\sum_{v\in V} G_{vu}.\]For undirected graphs, the out and in-degrees coincide.
\end{definition}

\begin{definition}The degree distribution of a finite graph $\mathcal{G}$ is:
\[\frac{1}{N}\sum_{u\in V}\delta_{D_u}=\sum_{d\in \Z_+} \frac{\Card\{u\in V  : D_u=d\}}{N} \delta_{d}.\]
For $d\in \Z_+$, $\Card\{u\in V  : D_u=d\}/N$ is the proportion of vertices with degree $d$.
\end{definition}

We see that the notion of degree distribution can be generalized to graphs with infinitely many vertices: the degree distribution is a probability measure on $\Z_+$, $\sum_{d\in \Z_+} p_d \delta_d$, where the weight $p_d$ of the atom $d\in \Z_+$ is the proportion of vertices with degree $d$.\\

Let us consider the product of the matrix $G$ with itself: $G^2=G \times G$. Notice that
\[G^2_{uv}=\sum_{w\in V}G_{uw}G_{wv},\]
and thus, $G^2_{uv}>0$ if there is a path consisting of two edges of $G$ that links $u$ and $v$. More precisely, $G^2_{uv}$ counts the number of paths of length exactly 2 that link $u$ and $v$. Generalizing this definition, and with the convention that $G^0=\Id$ the identity matrix of $\R^N$, we obtain that:

\begin{definition}Two vertices $u$ and $v$ of the graph $\mathcal{G}$ are connected if there is a path in $\mathcal{G}$ going from $u$ to $v$, i.e.\  if there exists some integer $n\geq 1$ such that $G^n_{uv}>0$. We can then define the graph distance between $u$ and $v$ by:
\begin{equation}
d_G(u,v)=\inf \{n\geq 0,\ G^n_{uv}>0\}.
\end{equation}By convention, $\inf \emptyset=+\infty$.\\
For $r\geq 0$, we define by $B_G(u,r)$ the ball of $\mathcal{G}$ with center $u$ and radius $r$ for the graph distance:
\[B_G(u,r)=\big\{v\in V : d_G(u,v)\leq r\big\}.\]
\end{definition}

Several important descriptors of the graph depend on this graph distance. We remark for instance that $D_u=\Card (B_G(u,1))-1$. Also, we can define a shortest path (for the graph distance) between two vertices $u$ and $v$. The diameter of the graph is:
\[\mbox{diam}(\mathcal{G})=\sup\{d_G(u,v) : u,v\in V\}.\]

\begin{definition}
For a vertex $u$ in a graph $\mathcal{G}$, we denote by $\mathcal{C}(u)$ the connected component of $u$, i.e.\  the set of vertices $v\in V$ that are connected to $u$:
\[\mathcal{C}(u)=\big\{v\in V : d_G(u,v)<+\infty\big\}.\]
\end{definition}

\section{Classical examples of random graphs}\label{CHI1.2}

Random graphs, especially those arising from applications, can have very complex distributions and topologies. There are some simple families of random graphs. We now present the complete graph, the Erd\"os--R\'enyi graphs, the stochastic block model, the configuration model and the household model.

\begin{definition}[Complete graph]
The complete graph $K_N$ is the graph where all the pairs of vertices are linked by an edge, i.e.\  $E=V\times V$.
\end{definition}

The complete graph is in fact a deterministic graph, and $\forall u,v\in V(K_N)$, $d_G(u,v)=1$ if $u\not= v$.

\begin{definition}[Erd\"os--R\'enyi random graph (ER)]\label{ER random graph definition} Erd\"os--R\'enyi random graphs are undirected graphs where each pair of vertices $(u,v)\in V^2$ is linked by an edge with probability $p\in [0,1]$ independently from the other pairs. \\
The distribution $\ER(N,p)$ of Erd\"os--R\'enyi random graphs is completely defined by the family $(G_{uv} ; \ u,v\in V,\ u<v)$ of i.i.d.\ random variables with Bernoulli distribution $\mbox{Ber}(p)$, $p\in [0,1]$.
\end{definition}

Notice that for $p=1$, the Erd\"os--R\'enyi graph corresponds to the complete graph $K_N$.\\

\bigskip

These graphs can be generalized if we introduce a partition of the population according to a discrete type, taking $K$ values, say $\{1,\dots ,K\}$: to each vertex $u\in V$ is associated a type $k_u\in \{1,\dots K\}$.
This corresponds to cases where a community contains different types of individuals that display specific roles in contact behaviour. Types might be related to age-groups, social behaviour or occupation.

\begin{definition}[Stochastic block model graph (SBM)]\label{def:SBM}
A stochastic block model graph is a undirected graph, where each vertex is given a type independently from the others, all with the same probability, and where each pair of vertices is linked independently of the other pairs with a probability depending on the types of the vertices. If there are $K$ types, say $\{1,\dots K\}$, we will denote by $(\rho_i)_{i\in \{1,\dots K\}}$ the probability distribution of the types, and by $\pi_{ij}$ the probability of linking a vertex of type $i$ with a vertex of type $j$.
\end{definition}

If there is just one type of vertices ($K=1$), the SBM resumes to ER graphs. For $K=2$ where vertices of the same type cannot be connected ($\pi_{11}=\pi_{22}=0$), we obtain \textit{bipartite} graphs. For instance, sexual networks in heterosexual populations are bipartite networks. The interested reader is referred to the review of Abbe \cite{abbeIII}.

\begin{proposition}\label{prop:deg-poisson}The degree distribution of a vertex $u$ in an $\ER(N,p)$ random graph with $N$ vertices and connection probability $p$ is a binomial distribution $\mathrm{Bin}(N,p)$. When the connection probability is $\lambda/N$, with $\lambda>0$, then for any integer $d\geq 0$,
\[\lim_{N\rightarrow +\infty} \P_N(D_u=d)= \frac{\lambda^d}{d!} e^{-\lambda},\]
showing that the probability distribution converges to a Poisson distribution with expectation $\lambda$.
\end{proposition}

The proof of this result is easy and let to the reader.\\

A detailed presentation and study of Erd\"os--R\'enyi graphs and their limits when $N\rightarrow +\infty$ can be found in \cite{vanderhofstadIII} for example. In particular, the case where the connection probability is $\lambda/N$, is carefully discussed. The case $\lambda>1$ is termed the supercritical case, while the case $\lambda<1$ is the subcritical case.\\

\bigskip

Proposition \ref{prop:deg-poisson} emphasizes the importance of graphs defined from their degree distributions. The next class of graphs has been introduced by Bollobas \cite{bollobas2001III} and Molloy and Reed \cite{molloyreedIII}. The reader is referred to Durrett \cite{durrettIII} and van der Hofstad \cite{vanderhofstadIII} for more details.

\begin{definition}[Configuration model graph (CM)]\label{def:CM}Let $\mathbf{p}=(p_k, \ k\in \Z_+)$ be a probability distribution on $\Z_+$. The Bollob\'as--Molloy--Reed or Configuration model random graph with vertices $V$ is constructed as follows.
We associate with each vertex $u\in V$ an independent random variable $X_u$ drawn from the distribution $\mathbf{p}$, that corresponds to the number of half edges attached to $u$. Conditionally on $\{\sum_{u\in V} X_u\mbox{ even}\}$, the Configuration model random graph is a multigraph (a graph with possibly self-loops and multiple edges) obtained by pairing the half-edges uniformly at random.
\end{definition}

A possible algorithm for pairing the half edges (also called stubs) is the following:
\begin{itemize}
\item Associate with each half edge an independent uniform random variable on $[0,1]$ and sort the half-edges by decreasing values.
\item Pair each odd stub with the following even stub. Note that if the number of stubs $\sum_{u\in V}X_u$ is odd, it is possible to add or remove one stub arbitrarily.
\end{itemize}
Note that this linkage procedure does not exclude self-loops or
multiple edges. When the size of the graph $N\rightarrow +\infty$ with a fixed degree
distribution, self-loops and multiple edges
become less and less apparent in the global picture (see e.g. \cite[Theorem 3.1.2]{durrettIII}).\\

In \cite{vanderhofstadIII}, it is carefully studied how one can turn a multigraph into a simple graph
(without self-loop nor multi-edge), either by erasing self-loops and merging multi-edges, or by conditioning on
obtaining a simple graph.
Note that in this respect, a Configuration model with a Binomial distribution $\mathcal{B}(N,p/N)$ looks like an
Erd\"os--R\'enyi graph with multiple-edges and self-loops.\\

Because of this construction, we see that in such a network, given an edge of ego $u$, the alter $v$ is chosen proportionally to his/her number of half-edges (i.e.\  his/her degree). Thus, the following degree distribution $\mathbf{q}=(q_k, k\in \Z_+)$ defined as the size-biased degree distribution of $\mathbf{p}$ will play a major role in the understanding of disease dynamics on CM graphs:
\begin{equation}q_k=\frac{k p_k}{\sum_{\ell \in \Z_+} \ell p_\ell}.\label{def:size-biased-distr}\end{equation}

\begin{example}
Particular graphs of this family include the regular graphs, where all the vertices have the same degree $d$ (that is $p_d=1$ and $\forall k\not= d, p_k=0$) and the graphs whose degree distribution is a power law: for some $\alpha>1$,
\[p_k \stackrel{k\rightarrow +\infty}{\sim} k^{-\alpha}.\]
\end{example}

A key quantity when dealing with configuration models is the generating function of its degree distribution, defined as:
\begin{equation}\label{CM:fgeneratrice}
g(z)=\sum_{k\geq 0} z^k p_k = \E_{\mathbf{p}}\big(z^D\big),
\end{equation}where the notation in the right-hand side recalls that the random variable $D$ has distribution $\mathbf{p}$.\\
In case it exists, the moment of order $q$ of the degree distribution can be written by means of the generating function:
\[\forall q\geq 0,\ \E_{\mathbf{p}}\big(D^q\big)=g^{(q)}(1).\]

\begin{example}Let us recall the probability generating function of some usual parametric distributions:
\begin{enumerate}
\item[(i)] For a Poisson distribution with parameter $\alpha$: $g(z)=e^{\alpha(z-1)}$.
\item[(ii)] For a Geometric distribution with parameter $\rho$: $g(z)=\frac{\rho z}{1-z(1-\rho)}$.
\item[(iii)] For a Binomial with parameters $(n,\rho)$: $g(z)=(z\rho +1-\rho)^n$.
\end{enumerate}
\end{example}

\begin{assumption}\label{hyp:moments}
Let us assume that $\mathbf{p}=(p_k, k\in \Z_+)$ admits a second order moment:
$$m=g'(1)=\sum_{k\in \Z_+} k p_k,\qquad \sigma^2=g''(1)+g'(1)-(g'(1))^2=\sum_{k\in \Z_+}(k-m)^2 p_k.$$
\end{assumption}

Notice that under Assumptions \ref{hyp:moments}, the size-biased degree distribution $\mathbf{q}$ defined in \eqref{def:size-biased-distr} admits a moment of order 1, which is referred to as the mean excess degree:
\begin{equation}\label{def:kappa}
\kappa = \sum_{k\geq 0} \frac{k(k-1)p_k}{m}=\frac{\sigma^2}{m} +m-1=\frac{g''(1)}{g'(1)}.
\end{equation}

\bigskip

The household models (see Part II of the present volume) can be built on the previous graph models. They were first analysed in detail in \cite{ballmollisonsaliatombaIII} and we also refer to Chapter \ref{sec:households} in Part II of this volume.
They account for several levels of mixing, for instance local and global in case of 2 levels. In the latter case, the population is partitioned into clusters or households. A first possible approach is to consider a graph model on the entire population (for example a CM in \cite{ballmollisonsaliatombaIII,ballnealIII,ballsirltrapmanIII}) on which the household structure is superposed independently. The links are considered stronger between individuals of the same household (for example they can transmit diseases at higher rates). Another possibility is to define the graph between individuals by taking into account the household structure, which results into clustering effects.
\begin{definition}[Household models]
A graph belong to the family of Household model if it is an SBM where the types are the households.
\end{definition}
Each household can be viewed as a vertex in a graph describing the global connections, while the intra-group connections between individuals of the same group are described by a local graph model.\\
How clustering affects epidemics using household models has for example been studied by \cite{ballbrittonsirlIII,coupechouxlelarge2014III}.\\

\bigskip

Let us also mention other families of random graphs: for example, the exponential random graphs, which are defined by their Radon--Nikodym densities. We refer to \cite{chatterjeeIII} for developments.

\begin{definition}[Exponential random graph model (ERGM)]A random graph belongs to the family of exponential random graphs if its distribution is of the following form. For a positive integer $K$, for a vector of parameters $\theta=(\theta_1,\dots \theta_K)\in \R^K$ and for a vector of statistics $(T_1,\dots T_K)$ of the graph, we have for any deterministic graph $g$:
\[\P_\theta\big(G=g\big)=\exp\Big(\sum_{k=1}^K \theta_k T_k(g)-c(\theta)\Big).\]
The renormalizing constant $c(\theta)$ is also called partition function in statistical mechanics.
\end{definition}

Examples of statistics $T_k$ are the number of edges, the degrees of vertices, the number of triangles or other patterns. In Rolls et al.\ \cite{rolls2013III}, ERGMs are for example used to estimate parameters describing the social networks of people who inject drugs in Australia. This has inspired a similar study for the French case, see \cite{cousien2III}.


\section{Sequences of graphs}

Let us consider a sequence of graphs $(\mathcal{G}_N)_{N\geq 1}$, such that for all $N\geq 1$, $\Card(V(\mathcal{G}_N))=N$.

For a given graph $\mathcal{G}$ and for an integer $j\geq 1$, let us denote by $\mathcal{C}_{(j)}(\mathcal{G})$ the $j$th largest connected component of $\mathcal{G}$.

\begin{definition}[Giant component]\label{def:giant}
Consider a sequence of graphs $(\mathcal{G}_N)_{N\geq 1}$ such that for all $N\geq 1$, $\Card(V(\mathcal{G}_N))=N$.
If
\[\liminf_{N\rightarrow +\infty} \frac{\Card\big(V(\mathcal{C}_{(1)}(\mathcal{G}_N))\big)}{N}>0,\]
then we say that the sequence $(\mathcal{G}_N)_{N\geq 1}$ is highly connected and that the graph $\mathcal{G}_N$ admits a giant component, $\mathcal{C}_{(1)}(\mathcal{G}_N)$.
\end{definition}

For $\ER(N,p)$ in the supercritical regime (with $Np>1$), there exists a giant component \cite[Theorem 4.8]{vanderhofstadIII}. So does it for the CM, as shown by Molloy and Reed \cite{molloyreedIII,Moll98III}. The condition for the existence with positive probability of a giant component in CM graphs
is that the expectation of the size biased distribution minus 1, $\kappa$, is larger than 1:
$$\kappa :=\sum_{k\in \Z_+}(k-1)\frac{kp_k}{\sum_{\ell\in \Z_+} \ell p_\ell} = \E_{\mathbf{q}}(D-1) >1.$$
This is connected with results on the super-criticality of Galton--Watson trees (see \cite[Section 3.2 p.~75]{durrettIII} for example). Heuristically, a CM graph looks like a tree locally, and a vertex of degree $k$ of the graph corresponds in the tree to a node with 1 parent and $k-1$ offspring. From the construction of the CM graphs given after Definition \ref{def:CM}, the degrees of the vertices encountered along the CM graph are given by the size-biased distribution.\\

If $\Card\big(V(\mathcal{C}_{(2)}(\mathcal{G}_N))\big)=o(N)$, then the giant component $\mathcal{C}_{(1)}(\mathcal{G}_N)$ is said to be unique. In many models such as ER, it is shown that the second largest component is of order $\log N$ (see \cite[Corollary 4.13]{vanderhofstadIII}).\\

The notion of being `highly connected', as introduced in Definition \ref{def:giant}, can also be extended.
\begin{definition}[Sequence of dense graphs]
We say that the graph sequence $(\mathcal{G}_N)_{N\geq 1}$ is a sequence of dense graphs if:
\[\liminf_{N\rightarrow +\infty} \frac{\Card\big(E(\mathcal{G}_N)\big)}{N^2}>0.\]
\end{definition}

\bigskip

Of course, the next important notion is the notion of convergence of a sequence of graphs $(\mathcal{G}_N)_{N\geq 1}$. The topologies and notions of convergence depend on the order of the edge numbers. For graphs that are not dense, such as tree-like graphs, a large literature around the Hausdorff-Gromov topology has developed and we refer for instance to Addario-Berry et al.\ \cite{addarioberrybroutingoldschmidtIII,addarioberrybroutingoldschmidt2012III}. When the graph is dense, the topology is inspired by ideas coming from the topologies of measure spaces (see Borgs et al.\ \cite{borgschayeslovaszsosvesztergombiIII} or Lovasz and Szegedy \cite{lovaszszegedyIII}). 


\section{Definition of the SIR model on a random graph}\label{sec:graphes}

We now describe the spread of infectious diseases on graphs. We consider a population of size $N$ whose individuals are the
 vertices of a random graph $\mathcal{G}_N$. As in compartmental models, the population is partitioned into three classes that can change in time: susceptible individuals who can contract the disease (individuals of type $\cS$), infectious individuals who transmit the disease (type $\cI$) and
 removed individuals who were previously infectious and can not transmit the disease any more (type $\cRr$). The corresponding sets of vertices,
  at time $t$, are respectively denoted by $\cS_t$, $\cI_t$ and $\cRr_t$, and the corresponding sizes by $S_t$, $I_t$ and $R_t$. \\

On the graph $\mathcal{G}_N$, the dynamics is as follows. To each $\cI$ individual is associated an exponential random clock with rate $\gamma$ to determine its removal. To each edge with an infectious ego and a susceptible alter, we associate a random exponential clock with
rate $\lambda$. When it rings, the edge transmits the disease and the susceptible alter becomes infectious.

\begin{example}[Compartmental models]
When the graph $\mathcal{G}_N=K_N$ is the complete graph, we recover the compartmental model of Part I of this volume.
\end{example}

\begin{example}[Household models]The above mechanisms can of course be generalized. For household models \cite{ballnealIII,ballsirltrapmanIII}, for example, the infection probability $\lambda$ depends on whether ego and alter belong or not to the same household. See Part II of this volume.
\end{example}

Notice also that for modelling real data, several studies require to take into account the dynamics of the social network itself (e.g.\ \cite{enrightkaoIII,volzancelmeyersIII}). For sexual network, for instance, accounting for the changes of sexual partners (contacts) is important (e.g.\ \cite{Kret96III,morriskretzschmar95III,schmidkretzschmarIII}). Also, the epidemics itself can act on the structure of the network (see \cite{kissberthouzemillersimonIII}), such as the changes of sexual behaviour due to the spread of the AIDS epidemic (e.g.\ \cite{mahsheltonIII}). These aspects are however not treated here.

\chapter{The Reproduction Number $R_0$}

We consider the early stage of the epidemics. Let us consider a single first infective of degree $d_1$ in a population of large size $N$.\\

For this, we proceed as in Section 1.2 of Part I of this volume and couple the process $(I_t)_{t\geq 0}$ with a branching process.
As for the mixing case, it is more precisely a stochastic domination. The coupling remains exact as long as no infected or
removed individual is contaminated for the second time, in which case the branching process creates an extra individual, who is
named `ghost'.

\begin{definition}[$R_0$]The basic reproduction number of the epidemic, denoted by $R_0$, is the mean offspring number of the branching process approximating the infectious population in early stages. \\
If we denote by $\beta(t)$ the birth rate at time $t>0$ in this branching process, then:
\begin{equation}
\label{R0:def}
R_0 = \int_0^{\infty} \beta(t) dt.
\end{equation}
\end{definition}
Notice that in the above definition, the measure $\beta(t)dt$ represents the intensity measure of the point process describing the occurrence of new infections due to a chosen infective (e.g.\ \cite{jagerslivreIII}).\\

A large literature is devoted to this indicator $R_0$ and extensions. Recall indeed that the nature and importance of the disease is usually classified according to whether $R_0>1$ or $R_0\leq 1$. \\
When $R_0>1$, the branching process is super-critical and with positive probability its size is infinite, in which case we say that there is a major outbreak of the disease. The probability for this to happen can be computed \cite[Eq. 3.10]{diekmannheesterbeekbritton-bookIII} and is less than 1. When the branching process does not get extinct, its size grows roughly proportional to $e^{\alpha t}$, where $\alpha$ is termed the (initial) epidemic growth rate (see \cite{jagerslivreIII}). In this case, the positive constant $\alpha$ depends on the parameters of the model through the equation
\begin{equation}
\label{malthus}
1= \int_0^{\infty} e^{-\alpha t} \beta(t) dt.
\end{equation}

When $R_0\leq 1$, the branching process is critical or subcritical and its size is almost surely finite. Then, the total number of individuals who have been infected when the epidemic stops (at the time $t$ when $I_t=0$) is upper bounded by an almost surely finite random variable with distribution independent of the total population size $N$, and we talk of a small epidemic. We refer to \cite{andersson_mathscientistIII,trapmanballdhersintranwallingabrittonIII} for reviews.

\section{Homogeneous mixing}
In the case where $\mathcal{G}_N=K_N$ is the complete graph, as stated in Part I of this volume, many results for epidemics in large homogeneous mixing populations can be obtained since the initial phase of the epidemic is well approximated by a branching process (see e.g.\ \cite{balldonnellyIII}).

\begin{proposition}[$R_0$ for homogeneous mixing]
The reproduction number is given by:
\[R_0 = \frac{\lambda}{\gamma}.\]
In the case where $\lambda>\gamma$, then $\alpha=\lambda-\gamma$ and
\[R_0  = \frac{\lambda}{\gamma}=1+\frac{\alpha}{\gamma}.\]
\end{proposition}
Notice that the second expression of $R_0$ does not depend on $\lambda$, which is sometimes complicated to estimate, especially at the beginning of an epidemic, but only on the removal rate $\gamma$, that is usually documented, and on the Malthusian parameter $\alpha$, that can be estimated from the dynamics of the emerging epidemics.

\begin{proof}The reproduction number $R_0$ for the homogeneous mixing case has already been studied in Part I of this volume, but let us give here another proof of the proposition using \eqref{R0:def}. In this case, $\beta(t)=\lambda e^{-\gamma t}$.
This can be understood by observing that $\lambda$ is the rate at which an infected individual makes contacts if he or she is still infectious, while $e^{-\gamma t}$ is the probability that the individual is still infectious $t$ time units after he or she became infected.
Then, \eqref{malthus} and \eqref{R0:def} translate to
\begin{equation}
 \label{Homcom}
 1 = \frac{\lambda}{\gamma + \alpha} \qquad \mbox{and} \qquad R_0 = \frac{\lambda}{\gamma}=1+\frac{\alpha}{\gamma}.
\end{equation}
This completes the proof.
\end{proof}

\section{Configuration model}\label{CHI2.2}

Assume that $\mathcal{G}_N$ is a configuration model graph whose degree distribution $\mathbf{p}$ admits a mean $\mu$ and a variance $\sigma^2$. Recall also the definition of the size-biased distribution
$\mathbf{q}$ in \eqref{def:size-biased-distr}, and of the mean excess degree $\kappa$ in \eqref{def:kappa}.
The mean excess degree $\kappa$, is in the context of SIR epidemics spreading on graphs, the mean number of susceptibles that are contaminated by a typical infective (other than his or her own infector).\\

Let us consider the following continuous time birth-death process $(X_t)_{t\geq 0}$. Individuals live during exponential independent times with expectation $1/\gamma$. To each individual is associated a maximal number of offspring $k-1$, where $k$ (the `degree' of the individual) is drawn in the size-biased distribution $\mathbf{q}$. We associate to such an individual $k-1$ independent exponential random variables with expectations $1/\lambda$. The ages at which the individual gives birth are the exponential random variables that are smaller than the lifetime of the individual. There is an intuitive coupling between $(X_t)_{t\geq 0}$ and $(I_t)_{t\geq 0}$ such as $X_t\geq I_t$ for every $t$, with the equality as long as no `ghost' has appeared.\\
We can associate with the process $(X_t)_{t\geq 0}$ its discrete-time skeleton (time counting the generations) that is a Galton--Watson process $(Z_n)_{n\geq 0}$ ($Z_0=1$). Conditionally on the degree $k$ and the fact that the chosen individual remains infectious for a duration $y$, the number of contacts contaminated by this individual follows a binomial distribution with parameters $k-1$ and $1-e^{-\lambda y}$. Summing over $k$ and integrating with respect to $y$, we can write the probability that in this Galton--Watson process an individual of generation $n\geq 1$ has $\nu=\ell$ offspring:
\begin{align*}
\P(\nu=\ell)= & \sum_{k=\ell+1}^{+\infty} \frac{k p_k}{m} {k-1 \choose \ell} \big(\frac{\lambda}{\lambda+\gamma}\big)^\ell \big(\frac{\gamma}{\lambda+\gamma}\big)^{k-1-\ell}.
\end{align*}

\begin{proposition}[$R_0$ for CM]\label{prop:R0CM}
Recall the definition of the mean excess degree $\kappa$ in \eqref{def:kappa}. We have:
\begin{equation}\label{R0:CM}
R_0=\frac{\kappa \lambda}{\lambda+\gamma}.
\end{equation}
In the super-critical case, $R_0$ can also be rewritten as
\[R_0= \frac{\gamma + \alpha}{\gamma + \alpha/\kappa}=1+\frac{\alpha}{\lambda+\gamma}.\]
\end{proposition}

\begin{proof}
With the description of the process $(Z_n)_{n\geq 1}$:
\begin{align*}
R_0= & \sum_{k\geq 0} \frac{kp_k}{m} \int_0^{+\infty} (k-1)(1-e^{-\lambda y}) \ \gamma e^{-\gamma y} dy\\
= & \sum_{k\geq 0} (k-1)\frac{kp_k}{\mu}\frac{\lambda}{\lambda+\gamma}\\
= & \big(\frac{g''(1)}{g'(1)}-1\big)\frac{\lambda}{\lambda+\gamma}\\
= & \frac{\kappa \lambda}{\lambda+\gamma}.
\end{align*}
We obtain \[\beta(t) = \kappa \lambda e^{-(\lambda + \gamma) t}.\] This can be seen by noting that $\kappa$ is the expected number of susceptible acquaintances a typical newly infected individual has in the early stages of the epidemic, while $e^{-\lambda t}$ is the probability that a given susceptible individual is not contacted by the infective over a period of $t$ time units, and $e^{-\gamma t}$ is the probability that the infectious individual is still infectious $t$ time units after he or she became infected. From \eqref{malthus}, we obtain that
\[\alpha=\kappa \lambda -\lambda-\gamma,\]
from which we conclude the proof.
\end{proof}

\begin{example}Let us compute $R_0$ for particular choices of degree distribution $\mathbf{p}$: \\
(i) For a Poisson distribution with parameter $a>0$,
\[R_0=
\frac{a\lambda}{\lambda+\gamma}.\]
Thus, $R_0>1$ if and only if
    $a>1+\gamma/\lambda$.\\
(ii) For a Geometric distribution with parameter $a\in (0,1)$,
   $ R_0= \frac{\lambda}{\lambda+\gamma}\frac{2(1-a)}{a}.$
Thus, $R_0>1$ if and only if $a<2\lambda/(3\lambda+\gamma)$.\hfill $\Box$
\end{example}

We can now connect the considerations on the skeleton with the epidemic in continuous time.

\begin{proposition}\label{propbrancht}Let us consider the continuous time birth-death process $(X_t)_{t\geq 0}$.
\begin{enumerate}
\item[(i)] If $R_0\leq 1$, the process $(X_t)_{t\geq 0}$ dies out almost surely.\\
\item[(ii)]If $R_0>1$, the process $(X_t)_{t\geq 0}$ dies with a probability $z\in (0,1)$ that is the smallest solution of
  \begin{equation}
    z =   \frac{\gamma}{g'(1)} \int_{\R_+} g'\big(z+e^{-\lambda y}(1-z)\big) e^{-\gamma y}dy.\label{pointfixeprobaext}
  \end{equation}
\item[(iii)] Let us define the times
$\tau_0=\inf\{t\geq 0\, \, |\, \, X_t=0\}$ and $\tau_{\varepsilon n}=\inf\{t\geq 0\,\,|\, \, X_t\geq \varepsilon n\}.$ If $R_0>1$, then for all sequence $(t_n)_{n\in \Z_+}$ such that $\lim_{n\rightarrow +\infty}t_n/\log(n)=+\infty$,
\begin{align}
  & \lim_{n\rightarrow +\infty} \mathbb{P}(\tau_0\leq t_n\wedge \tau_{\varepsilon n})=z \label{decollage1}\\
  &\lim_{n\rightarrow +\infty} \mathbb{P}(\tau_{\varepsilon n}\leq
  t_n\wedge \tau_{0})=1-z.\label{decollage2}
\end{align}
\end{enumerate}
\end{proposition}

\begin{proof}
Points (i) and (ii) are consequences of Proposition \ref{prop:R0CM} and the connections between the discrete time Galton--Watson tree and the continuous time birth-death process $(X_t)_{t\geq 0}$ that is coupled with $(I_t)_{t\geq 0}$ as long as no ghost has appeared.\\

The proof of (iii) is an adaptation of Lemma A.1 in M\'{e}l\'{e}ard and Tran \cite{meleardtranIII} (see also \cite{champagnatmeleardtranIII,chihdrIII}). Heuristically, (iii) says that at the beginning of the epidemics, the population either gets extinct with probability $z$ or, with probability $1-z$, reaches the size $\varepsilon n$ before time $t_n$ and before extinction. The time $t_n$ should be thought of as of order $\log(n)$, since the supercritical process has an exponential growth when it does not go to extinction.\\
For the birth-death process $(X_t)_{t\geq 0}$ there is no accumulation of birth and death events and almost surely,
\[\lim_{n\rightarrow +\infty} t_n\wedge \tau_{\varepsilon n }=+\infty.\]
So, we have by dominated convergence that
$\lim_{n\rightarrow +\infty} \mathbb{P}(\tau_0\leq t_n\wedge \tau_{\varepsilon n})=\P(\tau_0<+\infty)$. This last probability is the extinction probability of the process $(X_t)_{t\geq 0}$ which solves \eqref{pointfixeprobaext}. For the second limit, we have:
\begin{equation}\P(\tau_{\varepsilon n} \leq t_n\leq \tau_0)= \P(\tau_{\varepsilon n} \leq t_n\mbox{ and }\tau_0=+\infty)+ \P(\tau_{\varepsilon n} \leq t_n\leq \tau_0<+\infty).\label{bdlinear:eq1}\end{equation}
The second term of \eqref{bdlinear:eq1} is upper bounded by $\P(t_n\leq \tau_0<+\infty)$ which converges to 0 by dominated convergence when $n\rightarrow +\infty$.
For the second term, we can prove that with martingale techniques (e.g.\ \cite{jagerslivreIII}) that:
\begin{equation}
\lim_{t\rightarrow +\infty}\frac{\log X_t}{t}=\alpha,
\end{equation}where $\alpha$ is the initial epidemic growth rate defined in \eqref{malthus} and that is positive when $R_0>1$.\\

Let us consider $n>1/\varepsilon$, so that $\log(\varepsilon n)>0$. Since $\lim_{n\rightarrow +\infty}\tau_{\varepsilon n}=+\infty$ almost surely, we have on $\{\tau_0=+\infty\}$ that:
\[\lim_{n\rightarrow +\infty}\frac{\log(\varepsilon n)}{\tau_{\varepsilon n}}\geq \lim_{n\rightarrow +\infty}\frac{\log(X_{\tau_{\varepsilon n}-})}{\tau_{\varepsilon n}}=\alpha>0.\]
We deduce that:
\begin{align*}
\lim_{n\rightarrow +\infty}\P(\tau_{\varepsilon n}\leq t_n, \ \tau_0=+\infty)= & \lim_{n\rightarrow +\infty} \P\big(\frac{\tau_{\varepsilon n}}{\log(\varepsilon n)}\leq \frac{t_n}{\log(\varepsilon n)},\ \tau_0=+\infty\big)\\
= & \P(\tau_0=+\infty)=1-z,
\end{align*}since by our choice of $t_n$, $\lim_{n\rightarrow +\infty} t_n/\log(\varepsilon n)=+\infty$.
\end{proof}

Using similar results and fine couplings with branching properties, Barbour and Reinert \cite{barbourreinertIII} approximate the epidemic curve from the initial stages to the extinction of the disease.

\section{Stochastic block models}

We assume that there are $K$ types of individuals, labeled $\{1,2, \cdots, K\}$ and that for $k=1, \cdots, K$ a fraction $\eta_k$ of the $N$ individuals in the population is of type $k$. We assume that the infection rate from an ego of type $i$ to an alter of type $j$ is $\lambda_{ij}/N$.

\begin{proposition}[$R_0$ for SBM]Consider a SBM as in Definition \ref{def:SBM}. Denote by $\rho$ be the largest eigenvalue of the matrix with elements $\lambda_{ij} \rho_j$. Then:
$$R_0 = \frac{\rho}{\gamma}=1 + \frac{\alpha}{\gamma}.$$
\end{proposition}

\begin{proof}
We can hence couple here the infection process with a multi-type branching process. The rate at which a given $i$ individual gives birth to a $j$ individual corresponds to the rate, in the epidemic process, at which an $i$ individual infects $j$ individuals at time $t$ since infection: it is $a_{ij}(t)=\lambda_{ij} \rho_j e^{-\gamma t}$. Here, $\lambda_{ij}/N$ is the rate at which the $i$ individual contacts a given $j$ individual, $N \rho_j$ is the number of $j$ individuals and $ e^{-\gamma t}$ is the probability that the $i$ individual is still infectious $t$ time units after being infected. For multi-type branching processes, it is well known (e.g.\ \cite{ballclancyIII,diekmanngyllenbergmetzthiemeIII,diekmannheesterbeekbritton-bookIII}) that the basic reproduction number $R_0= \rho_M$ is the largest eigenvalue of the matrix $M$ with elements $m_{ij} = \int_0^{+\infty} a_{ij}(t) dt,$ and the epidemic growth rate  $\alpha$ is such that $1=\int_0^{\infty} e^{-\alpha t} \rho_{A(t)} dt,$ where $\rho_{A(t)}$ is the largest eigenvalue of the matrix $A(t)$ with elements $a_{ij}(t)$. Note that $\rho_{A(t)} = \rho e^{-\gamma t}$. Therefore,
\[R_0 =  \rho\int_0^{\infty} e^{-\gamma t} dt= \frac{\rho}{\gamma}\]and\[
1= \rho \int_0^{+\infty} e^{-(\alpha+\gamma) t} dt \quad \mbox{ leading to }   \quad \rho=\alpha+\gamma. \]
These equalities imply that
$$R_0 = 1 + \frac{\alpha}{\gamma},$$
which shows that the relation between $R_0$ and $\alpha$ for a multi-type Markov SIR epidemic is the same as for such an epidemic in a homogeneous mixing population (cf.\ equation \eqref{Homcom}).
\end{proof}

\section{Household structure}

It is possible to define several different measures for the reproduction numbers for household models \cite{ballpellistrapmanIIII,ballpellistrapmanIII,beckerdietzIII,goldsteinpaurfraserkenahwallingalipsitchIII}.
For this model it is hard to find explicit expressions for $R_0$. We refer to Part II of this volume, for discussion on the early stages of the an epidemic spreading on a household graph or on a two-level mixing graph.


\section{Statistical estimation of $R_0$ for SIR on graphs}

Since we often have observations on symptom onset dates of cases for a new, emerging epidemic, as was the case for the
Ebola epidemic in West Africa, it is often possible to estimate $\alpha$ from observations. In addition, we often have
observations on the typical duration between time of infection of a case and infection of its infector, which allow us
to estimate, assuming a Markov SIR model, the average duration of the infectious period, $1/\gamma$ \cite{wallingalipsitchIII}.\\


In \cite{trapmanballdhersintranwallingabrittonIII}, it is shown that estimates of $R_0$ obtained by assuming homogeneous mixing
are always larger than the corresponding estimates if the contact structure follows the configuration network model. 
For virtually all standard models studied in the literature, assuming homogeneous mixing leads to conservative estimates.

\section{Control effort}\label{CHI2.6}

\begin{definition}The control effort $v_c$ is defined as the proportion of infected individuals that we should prevent from spreading the disease and immunize to stop the outbreak (have $R_0<1$), the immunized people being chosen uniformly at random.
\end{definition}

For the homogeneous mixing contact structure, the required control effort for epidemics on the network structures
under consideration, is known to depend solely on $R_0$ through equation \cite[p.~69]{Brit07III}
\begin{proposition}On the complete graph $K_N$, we have that:
\begin{equation}\label{eq:vc-mixing}v_c=1-\frac{1}{R_0}=\frac{\alpha}{\alpha+\gamma}.\end{equation}
\end{proposition}

\begin{proof}Consider a given infectious non-immunized individual whose infectious period is of length $y>0$. In case we immunize a fraction $v_c$ of the infected individuals, the number of new infectious and non-immunized individuals contaminated by this individual is not a Poisson random variable with parameter $\lambda y$, but a thinned Poisson random variable of parameter $\lambda(1-v_c) y$.
The condition that the new $R_0=\lambda (1-v_c)/\gamma$ is less than 1 provides the expression of $v_c$ announced in the proposition.
\end{proof}

Notice that if we estimate the initial epidemic growth rate $\alpha$ and the mean duration of the infectious period $1/\gamma$ from the data, \eqref{eq:vc-mixing} allows us to propose a natural estimator of $v_c$.\\

For CM graphs, we can establish a similar formula for $v_c$ that depends also on the mean excess degree
  $\kappa$:

\begin{proposition}[$v_c$ for CM graphs]For a CM graph with degree distribution $\mathbf{p}$ and mean excess degree $\kappa$:
\[v_c = \frac{\kappa -1}{\kappa}\frac{\alpha}{\alpha + \gamma}.\]
\end{proposition}


The results obtained  for Markov SIR epidemics in the complete graph model, CM and SBM are summarized in Table \ref{tab:imp}. The results from household models are not in the table, since the expressions are hardly insightful. These results are taken from \cite{trapmanballdhersintranwallingabrittonIII}.\\

\begin{table}[h]
\centering
\begin{tabular}{lllll}
\hline
 & Quantity of &  \multicolumn{2}{c}{Quantity of interest as function of}  & Ratio with \\
Model & interest & $\lambda$, $\gamma$ and $\kappa$ & $\alpha$, $\gamma$ and $\kappa$ & complete graph\\
\hline
Complete graph & $\alpha$ & $\lambda-\gamma$ & - & - \\
 & $R_0$ & $\frac{\lambda}{\gamma}$ & $1+\frac{\alpha}{\gamma}$ & -\\
 & $v_c$ & $\frac{\lambda -\gamma}{\lambda}$ & $\frac{\alpha}{\alpha + \gamma}$ & -\\
\hline
CM & $\alpha$ & $(\kappa -1) \lambda-\gamma$ & - & -\\
 & $R_0$ & $\frac{\kappa \lambda}{\lambda + \gamma}$ &  $\frac{\gamma + \alpha}{\gamma + \alpha/\kappa}$ &
 $1+\frac{ \alpha}{\gamma \kappa}$\\
 & $v_c$ &  $1-\frac{\lambda +\gamma}{\kappa \lambda}$ & $\frac{\kappa -1}{\kappa}\frac{\alpha}{\alpha + \gamma}$ &  $1+ \frac{1}{\kappa-1}$\\
\hline
SBM & $\alpha$ & $\gamma (\rho_{M}-1)$ &  - & -\\
 & $R_0$ & $\rho_M$ &  $1+\frac{\alpha}{\gamma}$ & 1 \\
 & $v_c$ &  $1-\frac{1}{\rho_M}$ & $\frac{\alpha}{\alpha + \gamma}$ & 1\\
\hline
\end{tabular}
 \caption{{\small \textit{The epidemic growth rate $\alpha$, the basic reproduction number $R_0$ and required control effort $v_c$ for a Markov SIR epidemic model as function of model parameters in the complete graph $K_N$, in the CM and in the SBM. In the fourth column, the ratio has been made between the $R_0$ in the CM and SBM cases (numerators) and the $R_0$ obtained in mixing populations (complete graphs) given the estimations of $\alpha$, $\gamma$ and $\kappa$.}}}
\label{tab:imp}
\end{table}

Let us comment on these results. First, we find that the estimator of $R_0$ obtained assuming homogeneous mixing (complete graph) overestimates by a factor
$1+ \frac{1}{\kappa-1}$ the $R_0$ in configuration models. This factor is always strictly greater than 1, since the mean excess degree $\kappa$ is strictly greater than 1.
Thus, $v_c$ obtained by assuming homogeneous mixing is always larger than that of the configuration model.
Consequently, if the actual infectious contact structure is made up of a CM and a
perfect vaccine is available, we need to vaccinate a smaller proportion of the population than predicted assuming homogeneous mixing.\\
The overestimation of $R_0$ is small whenever $R_0$ is not much larger than 1 or when $\kappa$ is large. The same conclusion applies to the required control effort $v_c$.
The observation that the $R_0$ and $v_c$ for the homogeneous mixing model exceed the corresponding values for the network model extends to the full epidemic model allowing for an arbitrarily distributed latent period followed by an arbitrarily distributed independent infectious period, during which the infectivity profile (the rate of close contacts) may vary over time but depends only on the time since the start of the infectious period.
Figure \ref{ratiosfig}(a) shows that for SIR epidemics with Gamma distributed infectious periods, the factor by which the homogeneous mixing estimator overestimates the actual $R_0$ increases with increasing epidemic growth rate $\alpha$, and suggests that this factor increases with increasing standard deviation of the infectious period. Figure \ref{ratiosfig}(b) shows that the factors by which the homogeneous mixing estimator overestimates the actual $v_c$, decreases with increasing $\alpha$ and increases with increasing standard deviation of the infectious period. When the standard deviation of the infectious period is low, which is a realistic assumption for most emerging infectious diseases (see e.g.\ \cite{Cori12III}), and $R_0$ is not much larger than 1, then ignoring the contact structure in the network model and using the simpler estimators for the homogeneous mixing results in a slight overestimation of $R_0$ and $v_c$.\\

\begin{figure}[!ht]
\centering
    \includegraphics[width=.3\textwidth]{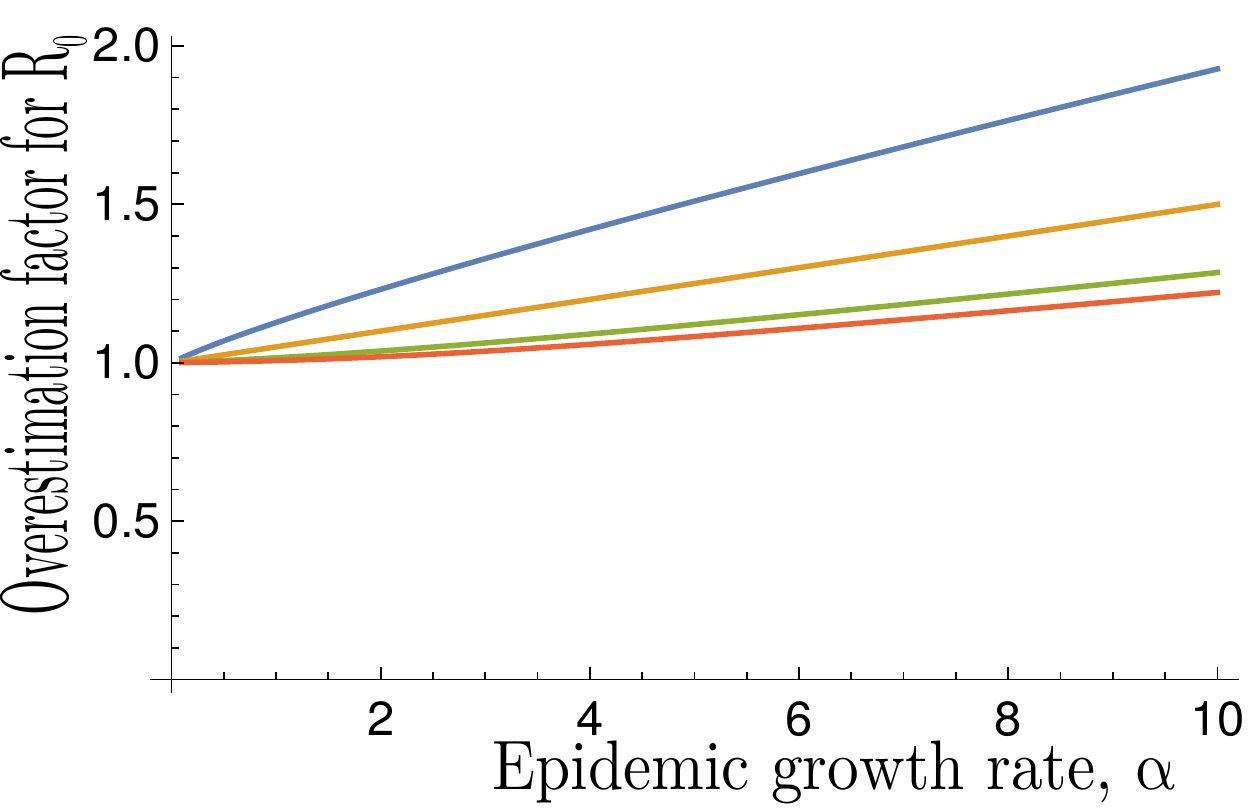}
   \includegraphics[width=.3\textwidth]{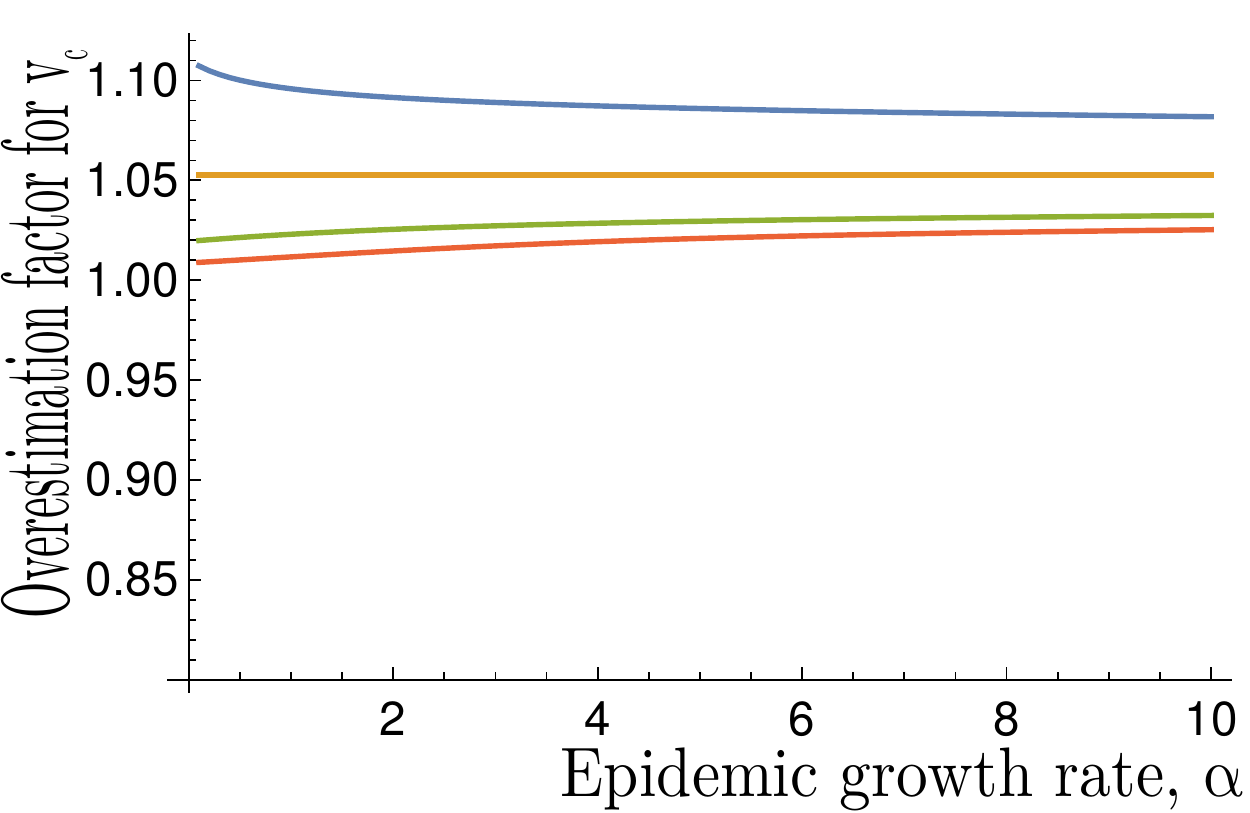}
\caption{\textit{{\small The factor by which estimators based on homogeneous mixing will overestimate (a) the basic reproduction number $R_0$ and (b) the required control effort $v_c$ for the network case. Here the epidemic growth rate $\alpha$ is measured in multiples of the mean infectious period $1/\gamma$. The mean excess degree $\kappa= 20$. The infectious periods are assumed to follow a gamma distribution with mean 1 and standard deviation $\sigma\!=\!1.5$, $\sigma\!=\!1$, $\sigma\!=\!1/2$ and $\sigma\!=\!0$, as displayed from top to bottom. Note that the estimate of $R_0$ based on homogeneous mixing is $1\!+\!\alpha$. Furthermore, note that $\sigma\!=\!1$, corresponds to the special case of an exponentially distributed infectious period, while if $\sigma\!=\!0$, the duration of the infectious period is not random.}}}
\label{ratiosfig}
\end{figure}

When considering epidemics spreading on SBM graphs (see \cite[Supplementary materials]{trapmanballdhersintranwallingabrittonIII}), we can derive that estimators for $R_0$ and (if control measures are independent of the types of individuals) $v_c$  are exactly the same as for homogeneous mixing in a broad class of SEIR epidemic models. This class includes the full epidemic model allowing for arbitrarily distributed latent and infectious periods and models in which the rates of contacts between different types keep the same proportion all of the time, although the rates themselves may vary over time  (cf.\ \cite{Diek98III}).\\
We illustrate our findings on multitype structures through simulations of SEIR epidemics in an age stratified population with known contact structure as described in \cite{Wall06III}.
We use values of the average infectious period $1/\gamma$  and the average latent period $1/\delta$ close to the estimates for the 2014 Ebola epidemic in West Africa \cite{teamebolaIII}.\\
Two estimators for $R_0$ are computed. The first of these estimators is based on the average number of infections among the people who were infected early in the epidemic. This procedure leads to a very good estimate of $R_0$ if the spread of the disease is observed completely. The second estimator for $R_0$ is based on $\hat{\alpha}$, an estimate of the epidemic growth rate $\alpha$,
and known expected infectious period $1/\gamma$ and expected latent period $1/\delta$. This estimator of $R_0$ is $(1+ \hat{\alpha}/\delta)(1+ \hat{\alpha}/\gamma)$. We calculate estimates of $R_0$ using these two estimators for 250 simulation runs. As predicted by the theory, the simulation results show that for each run the estimates are close to the actual value (Figure \ref{actvsestmult}(a)), without a systematic bias (Figure \ref{actvsestmult}(b)).\\

\begin{figure}[!ht]
\begin{center}
    \includegraphics[width=.34\textwidth]{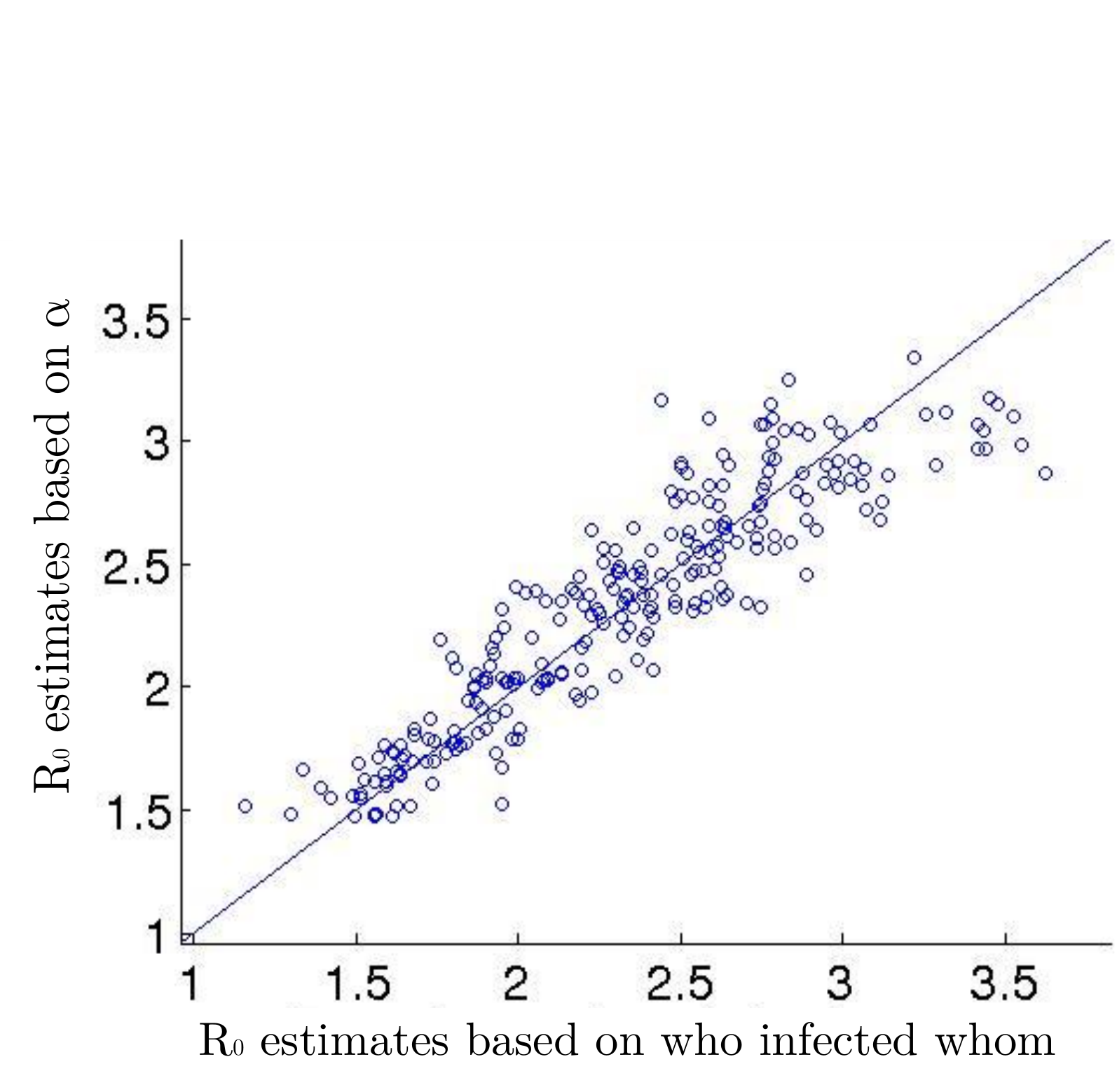}
   \includegraphics[width=.13\textwidth]{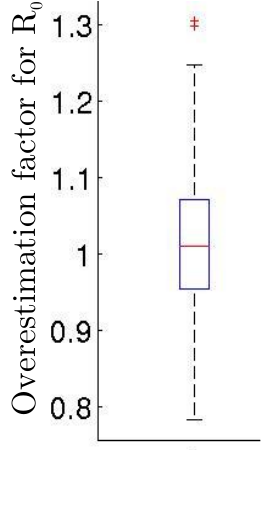}
\caption{{\small \textit{The estimated basic reproduction number, $R_0$, for a Markov SEIR model in a multi-type population as described in \cite{Wall06III}, based on the real infection process (who infected whom) plotted against the computed $R_0$, assuming homogeneous mixing, based on the estimated epidemic growth rate, $\alpha$, and given expected infectious period (5 days) and expected latent period (10 days). The infectivity is chosen at random, such that the theoretical $R_0$ is uniform between 1.5 and 3. The estimate of $\alpha$ is based on the times when individuals become infectious. In the right plot, a boxplot of the ratios is given.}}}
\label{actvsestmult}
\end{center}
\end{figure}

Let us now consider an epidemic spreading on a household structure. It is also argued that the required control effort satisfies $v_c \geq 1-1/R_0$ for this model, which implies that if we know $R_0$ and we base our control effort on this knowledge, we might fail to stop an outbreak. However, we usually do not have direct estimates for $R_0$ and even though it is not true in general that using $R_0$ leads to conservative estimates for $v_c$ \cite{Ball14III}, numerical computations suggest that the approximation of $v_c$ using $\alpha$  and the homogeneous mixing assumption is often conservative.\\

To illustrate this last point, we consider in Figure \ref{housepictures} a household structure with within and global infectivities. The within household infection rate is $\lambda_H$. In the simulations, we show estimates for $R_0$ and $v_c$ over a range of values for the relative contribution of the
within-household spread. For each epidemic growth rate $\alpha$, the estimated values remain below the value obtained for homogeneous mixing (neglecting the partition into households).  \\
We use two types of epidemics: in (a) and (b) the Markov SIR epidemic is used, while in (c)  the so-called Reed--Frost model is used, which can be interpreted as an epidemic in which infectious individuals have a long latent period of non-random length, after which they are infectious for a very short period of time. We note that for the Reed--Frost model the relationship between $\alpha$ and $R_0$ does not depend on the household structure (cf.\ \cite{Ball14III})
and therefore, for this model, only the dependence of $v_c$ on the relative contribution of the within household spread is shown in Figure \ref{housepictures}.\\
The household size distribution is taken from a 2003 health survey in Nigeria \cite{nige2003III}. For Markov SIR epidemics, as the within-household infection rate $\lambda_H$ is varied, the global infection rate is varied in such a way that the computed epidemic growth rate $\alpha$ is kept fixed. For this model, $\alpha$ is calculated using the matrix method described in Section 4.1 of \cite{PFF11III}.\\
For the Reed--Frost epidemic model, the probability that an infectious individual infects a given susceptible household member during its infectious period, $p_H$ is varied, while the corresponding probability for individuals in the general population varies with $p_H$ so that $\alpha$ is kept constant. For this model, $R_0$ coincides with the initial geometric rate of growth of infection, so $\alpha= \log(R_0)$. From Figure \ref{housepictures}, we see that estimates of $v_c$ assuming homogeneous mixing are reliable for Reed--Frost type epidemics, although as opposed to all other analysed models and structures, the estimates are not conservative. We see also that for the Markov SIR epidemic, estimating $R_0$ and $v_c$ based on the homogeneously mixing assumption might lead to conservative estimates which are up to 40\% higher than the real $R_0$ and $v_c$.

\begin{figure*}[ht]
\centering
   \includegraphics[width=.32\textwidth]{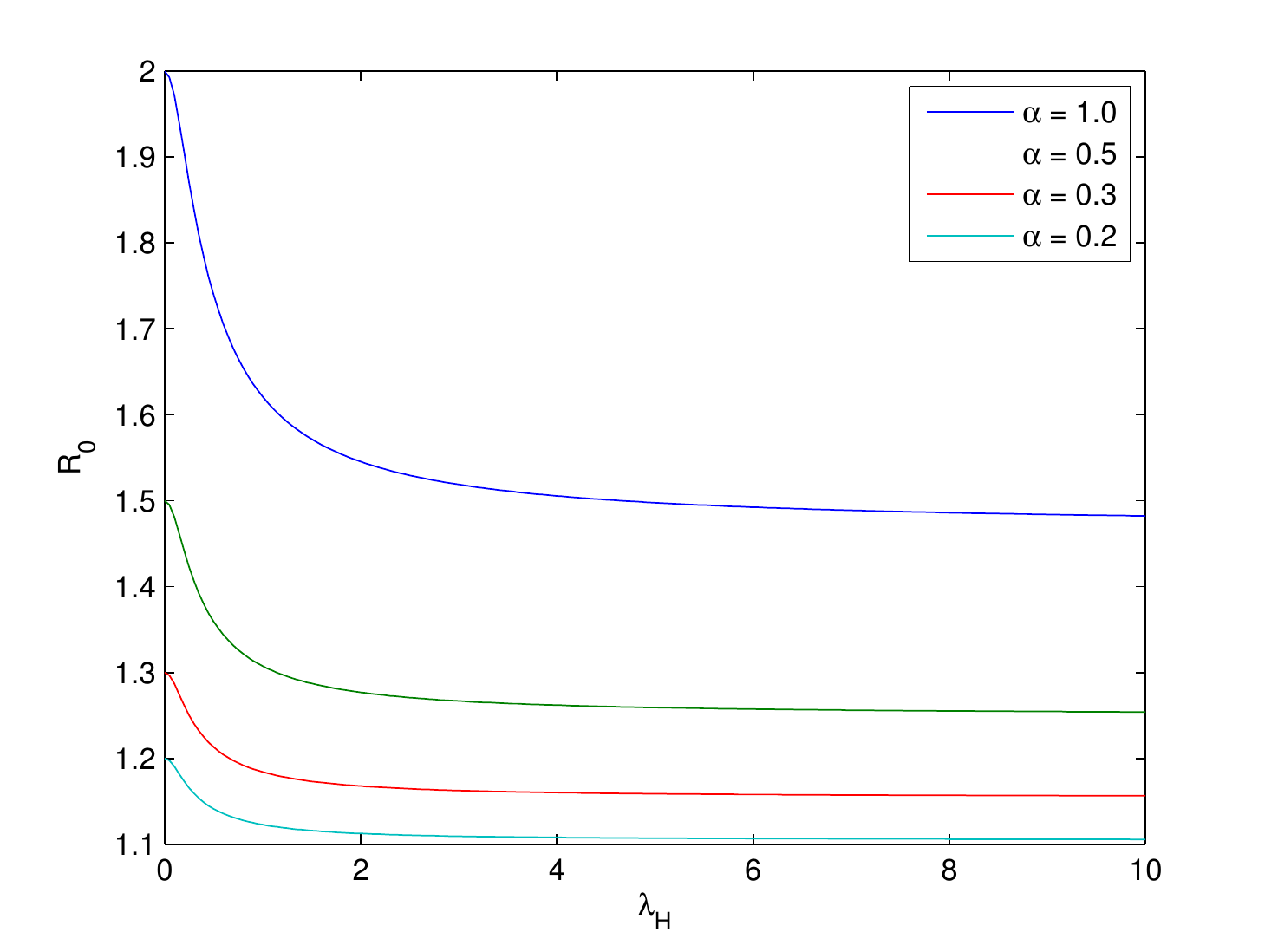}
    \includegraphics[width=.32\textwidth]{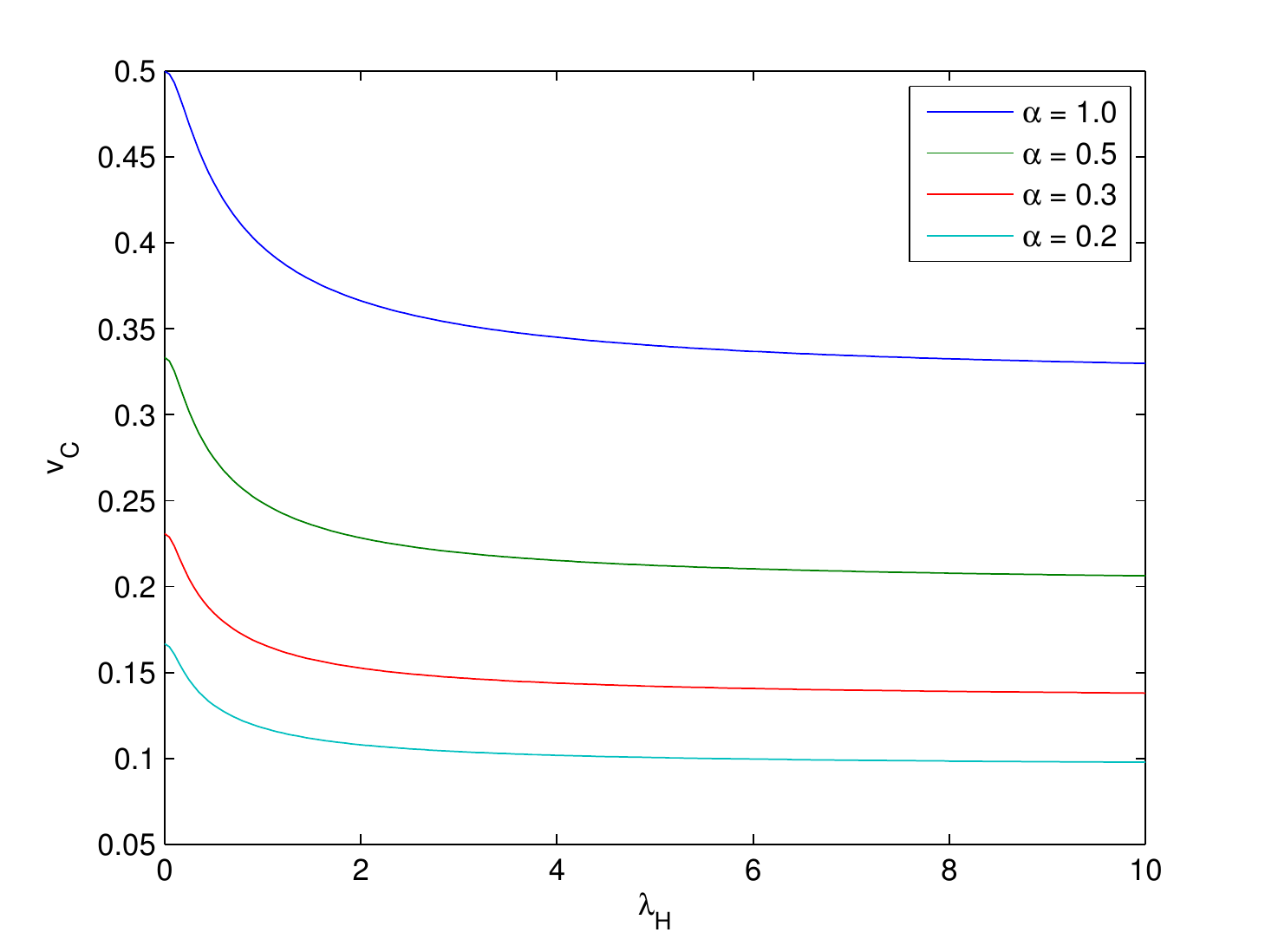}
   \includegraphics[width=.32\textwidth]{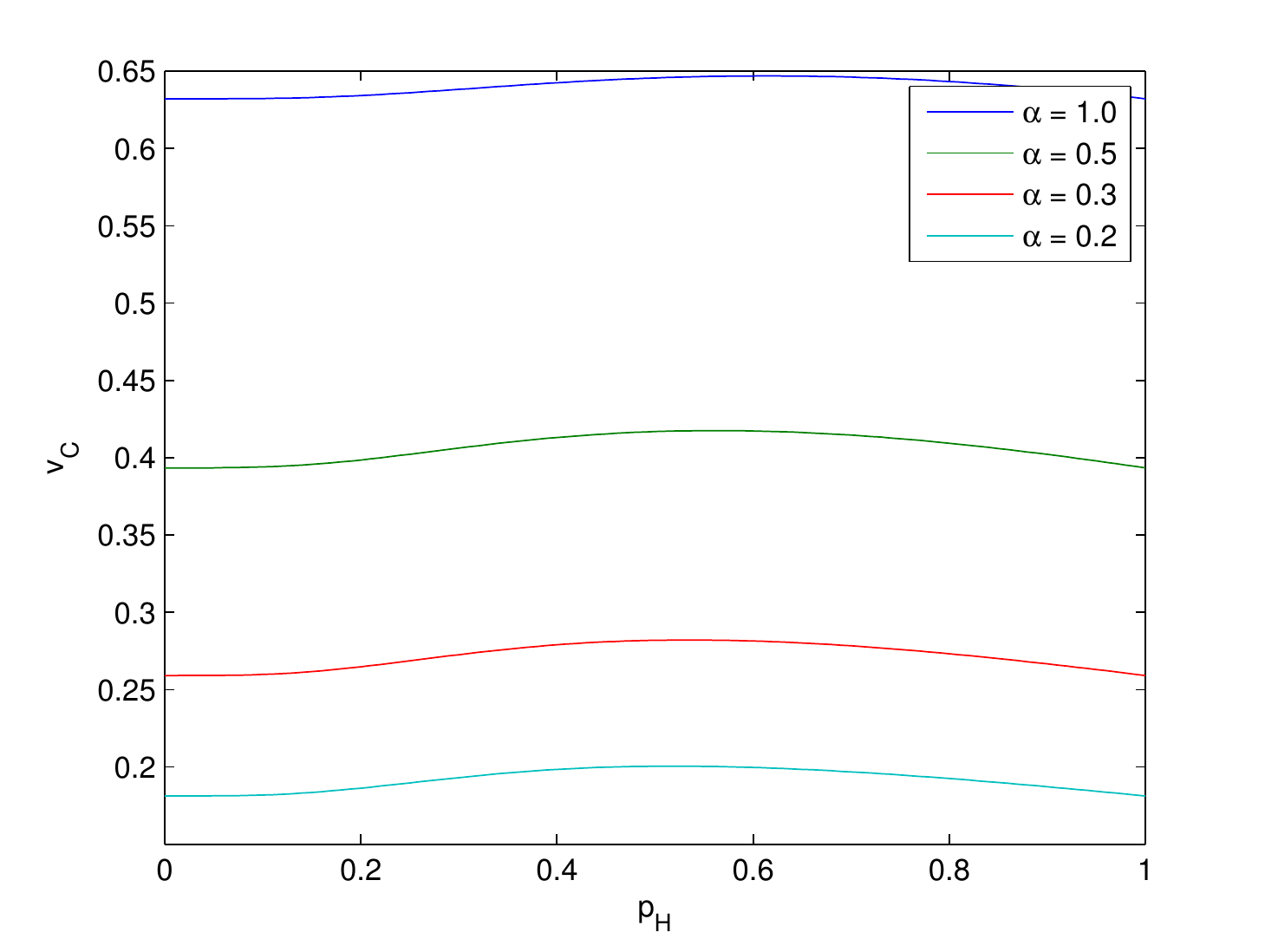}
\caption{\textit{{\small Estimation of key epidemiological variables in a population structured by households (see Part II of this volume). The basic reproduction number $R_0$ for Markov SIR epidemics (a), critical vaccination coverage $v_c$  for Markov SIR epidemics (b)  and $v_c$ for Reed--Frost epidemics (c), as a function of the relative influence of within household transmission, in a population partitioned into households. The household size distribution is given by
$m_1 = 0.117, m_2=0.120, m_3=0.141, m_4=0.132, m_5=0.121, m_6= 0.108, m_7= 0.084, m_8=0.051, m_9= 0.126$, for  $i=1,2, \cdots, 9$, $m_i$ is the fraction of the households with size $i$. The global infectivity is chosen so that the epidemic growth rate $\alpha$ is kept constant while the within household transmission varies. Homogeneous mixing corresponds to $\lambda_H=p_H=0$.}}}
\label{housepictures}
\end{figure*}

\chapter{SIR Epidemics on Configuration Model Graphs}\label{CHI3}

We now turn to establishing limit theorems for approximating the dynamics of the disease in large populations,
when $N\rightarrow +\infty$, similarly to Chapter \ref{chap_MarkovMod} in Part I of this book. We focus here on the case where
$\mathcal{G}_N$ is a Configuration model graph, and we will let $N\rightarrow +\infty$.
 Several strategies have been developed for epidemics spreading on such random graphs (see e.g.\ Newman \cite{newmanIII, newman_SIAMIII}, Durrett \cite{durrettIII}, Barth\'{e}lemy et al.\
 \cite{barthelemybarratpastorsatorrasvespignaniIII}, Kiss et al.\ \cite{kissmillersimonIII}).\\

Contrarily to the classical mixing compartmental SIR epidemic models (e.g.\ \cite{kermackmckendrickIII,
   bartlettIII} see also Part I of this book for a presentation), heterogeneity in the number of contacts makes it
 difficult to describe the dynamical behaviour of the epidemic. An important literature, starting from Andersson
 \cite{andersson_mathscientistIII}, deals with moment closure, mean field approximations (e.g.\ \cite{pastorsatorrasvespignaniIII,
   barthelemybarratpastorsatorrasvespignaniIII, durrettIII, kissmillersimonIII}) or large population approximations
   (e.g.\ \cite{ballnealIII}, see also Eq.\ (3) of
 \cite{anderssonIII} in discrete time). In 2008, Ball and Neal \cite{ballnealIII} proposed to describe the dynamics with an infinite system of ordinary differential equations, by obtaining an equation for each subpopulation of individuals with same degree $k$, $k\in \Z_+$. The same year, Volz \cite{volzIII} proposed a large population approximation with only 5 ordinary differential equations and without moment closure, which was a major advance for prediction and tractability. The key concept behind his work was to focus not only on node-based quantities, but rather of edge-based ones (see also \cite{millerIII}). Rigorous proofs have then been proposed by \cite{decreusefonddhersinmoyaltranIII,barbourreinertIII,jansonluczakwindridgeIII}).\\

Recall that we have denoted the sets of $\cS$, $\cI$ and $\cRr$ vertices at time $t$ by $\cS_t$, $\cI_t$ and $\cRr_t$ (see Section \ref{sec:graphes}). The sizes of these sub-populations are $S_t$, $I_t$ and $R_t$.
We will say that an edge linking an infectious ego and susceptible alter is of type $\cI-\cS$
(accordingly $\cRr-\cS$, $\cI-\cI$ or $\cI-\cRr$).

\section{Moment closure in large populations}\label{sec:moment}

For the presentation in this section, we follow the work of \cite{andersson_mathscientistIII}. Let us introduce some notation.
For $u\in V$, denote
\[S_u(t)=\ind_{u\in \cS_t}\qquad \mbox{ and }\qquad I_u(t)=\ind_{u\in \cI_t}.\]
Then, $S_t=\sum_{u\in V} S_u(t)$ and $I_t=\sum_{u\in V} I_u(t).$ Because the size $N$ of the graph $\mathcal{G}_N$ converges to infinity,
we will be lead to study the proportions of susceptible, infectious and removed individuals, that are denoted by:
\begin{equation}\label{renorm:S-I-R}S^N_t= \frac{S_t}{N},\qquad I^N_t=\frac{I_t}{N},\qquad R^N_t=\frac{R_t}{N}.\end{equation}
Notice that $S^N_t+I^N_t+R^N_t=1$ since our population is closed. Hence, knowing the evolution of $S^N_.$ and $I^N_.$ is sufficient
for describing the size and evolution of the outbreak.\\
For $A$, $B$, $C$ being $S$ or $I$, we denote by
\begin{gather*}
[a]=\lim_{N\rightarrow +\infty}\frac{1}{N}\sum_{u\in V} A_u=a,\qquad [ab]=\lim_{N\rightarrow +\infty}\frac{1}{N}\sum_{u,v\in V}A_u G_{uv} B_v,\\
[abc]=\lim_{N\rightarrow +\infty}\frac{1}{N}\sum_{u,v,w\in V}A_u G_{uv} B_v G_{vw} C_w,
\end{gather*}
where we recall that $G$ is the adjacency matrix of the graph (see Definition \ref{def:adjmatrix}).\\

In the sequel, we will work under the following assumption.
\begin{assumption}\label{hyp:initialcond-conv}We assume that $\lim_{N\rightarrow +\infty}(S^N_0,I^N_0)=(s_0,i_0)\in (\R_+\setminus \{0\})^2$ and that for all $N$, $R_0^N=0$.
\end{assumption}

The idea is that in the large population limit, the initial fraction of
infectious individuals should be positive to allow the observation of an outbreak. That is why we assume that it is of order $i_0 N$
with $i_0>0$ but possibly small with respect to 1.\\

Let us present a system of limiting deterministic equations. The limit theorems allowing to obtain the following equations from the
finite stochastic system are not shown here. In fact, we will later detail how Volz' equations are obtained.\\

Andersson \cite{andersson_mathscientistIII} proposes the following ODEs
for the sizes of the $\cS$ and $\cI$ classes.
\begin{equation}
\frac{ds_t}{dt}=  -\lambda [s_t i_t],\qquad \frac{di_t}{dt}=  \lambda [s_t i_t]-\gamma i_t. \label{eq:moment1}
\end{equation}
Let us comment on these equations. In a closed population, susceptible individuals
disappear when they are contaminated, i.e.\  when an edge with susceptible ego and infectious alter transmits the disease. Thus, the
rate at which the number of susceptible individuals decreases due to infection (which equals to the rate at which the number of infectious
individuals increases) should be proportional to the proportion of edges
with susceptible ego and infectious alter, $[s_t i_t]$. The rate at which infectious individuals disappear is $-\gamma i_t$ as in the compartmental
case, since removals are node-related events and not edge-related events like infections.\\

Equations \ref{eq:moment1} are not closed, and this leads Andersson to propose the following assumption.
\begin{assumption}\label{hyp:momentclosure}
Let $A$, $B$, $C$ be $S$ or $I$.
If $\{u,w\}\notin E$, we assume that
$$\P(A_u=1 \ |\ B_v C_w=1)=\P(A_u=1\ |\ B_v=1)=\frac{\P(A_u=1,\ B_v=1)}{\P(B_v=1)}.$$
\end{assumption}
Let us comment on this assumption. As the Bayes formula says that:
$$\P(A_u B_v C_w=1)=\P(A_u=1\ |\ B_v C_w=1)\P(B_v C_w=1),$$
Assumption \eqref{hyp:momentclosure} implies that
\[\P(A_u C_w=1\ |\ B_v=1)=\P(A_u=1\ |\ B_v=1)\P( C_w=1\ |\ B_v=1).\]
Thus, Assumption \ref{hyp:momentclosure} amounts to assuming that conditionally on having a $B$ friend, having an $A$ and a $C$ friends are independent
 events, and is heuristically true when
$$[abc]\approx \frac{[ab][bc]}{[b]}.$$
This assumption fails when we are in graphs with strong correlations between edges so that `the friend of my friend is also my friend'.\\

Let us define the selection pressure by
\begin{equation}
\widetilde{i}_t=\frac{[s_t i_t]}{s_t}.\end{equation}It is the mean number of edges toward $\cI_t$ for individuals in
$\cS_t$. This quantity allows Andersson \cite{andersson_mathscientistIII} to close the system of ODEs \eqref{eq:moment1} under Assumption \ref{hyp:momentclosure}.

\begin{theorem}\label{th:moments}
Under Assumption \ref{hyp:momentclosure}, the epidemic on the network can be described by the following equations:
\begin{align}
& \frac{ds_t}{dt}=  -\lambda s_t \widetilde{i}_t,\\
& \frac{di_t}{dt}=  \lambda s_t \widetilde{i}_t-\gamma i_t\\
& \frac{d\widetilde{i}_t}{dt}=  \big(C \lambda s_t  -\lambda  -\gamma \big)\widetilde{i}_t.\label{eq:momentclosure}
\end{align}
\end{theorem}

\begin{proof}
The equations proposed in Theorem \ref{th:moments} are derived in several steps.
Recall Equations \eqref{eq:moment1}. To close them, it is needed to describe how the quantities of edges $[s_t s_t]$ and $[s_t i_t]$ evolve.
An edge $\cS-\cS$ disappears when one of its vertices is infected. For each motif $\cS-\cS-\cI$, the edge $\cS-\cI$ transmits the disease
independently with rate $\lambda$. Thus, the rate of disappearance of $\cS-\cS$ edges is proportional to the $\lambda [s_t s_t i_t]$.\\
Similarly, $\cS-\cI$ edges appear when edges $\cS-\cS$ become $\cS-\cI$, and disappear when becoming $\cI-\cI$ (which happens when the susceptible
vertex is infected by its infectious alter, or by another infectious contact) or when becoming $\cS-\cRr$ (when the infectious individual is removed). Then:
\begin{align}& \frac{d[s_t s_t]}{dt}= -2 \lambda [s_t s_t i_t], \nonumber\\
&\frac{d[s_t i_t]}{dt}= \lambda \big([s_t s_t i_t]-[i_t s_t i_t]-[s_t i_t]\big)-\gamma [s_t i_t].\label{eq:moment2}
\end{align}
These equations are still not closed, as they depend on the numbers of motifs $\cS-\cS-\cI$ and $\cI-\cS-\cI$ renormalized by $N$.
The equations that we might write for these quantities depend on motifs with four vertices etc. To close the equations, we use
Assumption \ref{hyp:momentclosure}. Then, the equations \eqref{eq:moment2} become:
\begin{align*}
& \frac{d[s_t s_t]}{dt}= -2 \lambda \frac{[s_t s_t][s_t i_t]}{s_t},\\
& \frac{d[s_t i_t]}{dt}= \lambda \Big(\frac{[s_t s_t][s_t i_t]}{s_t}-\frac{[s_t i_t]^2}{s_t}-[s_t i_t]\Big)-\gamma [s_t i_t].
\end{align*}
Notice that
$$\frac{d(s_t^2)}{dt}= 2 s_t \frac{ds_t}{dt}=-2 \lambda s_t [s_t i_t]=-2\lambda \frac{[s_t i_t]}{s_t}  s_t^2.$$
Thus, $(s_t^2)$ and $[s_t s_t]$ satisfy the same ODE and we deduce that there exists a $C>0$ such that $[s_t s_t]= C s_t^2$.\\

Using the definition of the selection pressure $\widetilde{i}_t$,
\begin{align*}\frac{d\widetilde{i}_t}{dt}= & \frac{d[s_t i_t]}{dt}\frac{1}{s_t}-\frac{[s_t i_t]}{s_t^2}\frac{ds_t}{dt}\\
= & \frac{1}{s_t}\Big(\lambda \big(C s_t^2 \times \widetilde{i}_t s_t \times \frac{1}{s_t} - \widetilde{i}^2_t s_t^2 \times \frac{1}{s_t}-\widetilde{i}_t s_t\big)-\gamma \widetilde{i}_t s_t\Big)+\frac{\widetilde{i}_t s_t}{s_t^2} \times \lambda \widetilde{i}_t s_t\\
= & \big(C \lambda s_t  -\lambda  -\gamma \big)\widetilde{i}_t.
\end{align*}
The system can then be reformulated as the announced system with three ODEs in $s_t$, $i_t$ and $\widetilde{i}_t$.
\end{proof}

When the infection rate is low and the number of $\cS-\cS$ edges is very high, we recover the Kermack--McKendrick ODEs describing the
dynamics of an epidemic in a homogeneous case:

\begin{proposition}If $C\rightarrow +\infty$ and $\lambda\rightarrow 0$ with $\lambda'=C \lambda$ constant, we recover in the limit the Kermack--McKendrick system of
ODE:
\begin{align*}
& \frac{ds_t}{dt}=  -\lambda' s_t i_t\\
&\frac{di_t}{dt}=  \lambda' s_t i_t-\gamma i_t.
\end{align*}
\end{proposition}

\begin{proof}
If $C\rightarrow +\infty$ and $\lambda\rightarrow 0$ with $\lambda'=C \lambda$ constant, then `in the limit':
$$\frac{d\widetilde{i}_t}{dt}=  \lambda' \ s_t  \widetilde{i}_t  -\gamma \widetilde{i}_t.$$
Consider $f(t)=\widetilde{i}_t-C i_t$. This quantity satisfies
$$\frac{df}{dt}(t)= -\gamma f_t.$$
Applying Gronwall's inequality, this yields that $\widetilde{i}_t=C i_t$. We recover as announced, the Kermack--McKendrick ODEs with infection rate $\lambda'$.
\end{proof}

From Equation \eqref{eq:momentclosure}, we can for example predict the total size of the epidemics, i.e.\  the number of removed individuals when the
infective population vanishes and the epidemics stops.

\begin{proposition}
Based on the equations \eqref{eq:momentclosure}, we can compute the final size of the epidemics:
$$z:=s_0-s_\infty=s_0 \ \big(1-\exp\big(-\frac{\lambda}{\lambda+\gamma}(Cz+\widetilde{i}_0)\big)\big).$$
\end{proposition}

\begin{proof}Because $t\mapsto s_t$ is a continuous non-negative decreasing function, it converges to a limit $s_\infty$ when $t\rightarrow +\infty$. From \eqref{eq:momentclosure}:
\[ \frac{d\widetilde{i}_t}{dt}=  -\lambda s_t \widetilde{i}_t \big( -C +\frac{1}{s_t}+\frac{\gamma}{\lambda s_t}\big)=\frac{ds_t}{dt} \big(-C +\frac{1+\frac{\gamma}{\lambda}}{s_t}\big)\]
from which we obtain by integration:
\[
\widetilde{i}_t-\widetilde{i}_0=-C(s_t-s_0)+\big(1+\frac{\gamma}{\lambda}\big)\log \frac{s_t}{s_0}.
\]
Since $\widetilde{i}_\infty=0$:
$$-\widetilde{i}_0 + C(s_\infty-s_0)=\big(1+\frac{\gamma}{\lambda}\big)\log \frac{s_\infty}{s_0}.$$
Computing $z:=s_0-s_\infty$, we recover the announced result.
\end{proof}

For further and recent developments on moment closures, we refer the reader to e.g.\ \cite{pellishousekeelingIII} or \cite{kissmillersimonIII}.

\section{Volz and Miller approach}\label{sec:volzmiller-graphdescription}


In 2008, Volz \cite{volzIII} proposed a system of only 5 ODEs to describe the spread of an epidemic on a random CM graph.
Volz approximation is based on an edge-centered point of view, in an `infinite' CM graph setting, without any assumption of moment closure.
We present Volz equations and then explain how to recover them with Miller's approach \cite{millerIII}. The derivation of these equations
as limit of epidemics spreading on finite graphs is detailed following the approach of Decreusefond et al.\ \cite{decreusefonddhersinmoyaltranIII}. \\

The spread of diseases on random graphs involves two sources of randomness: one for the random graph, the other for describing the
way the epidemic propagates on this random environment. An idea coming from statistical mechanics is to build the random graph
progressively as the epidemic spreads over it, instead of first constructing the random graph, conditioning on it and studying the epidemic on
the frozen environment. We detail the process that we will consider in the rest of the section.\\
Assume that only the edges joining the $\cI$ and $\cRr$ individuals are observed. This means that the cluster of
infectious and removed individuals is built, while the network of susceptible individuals is still not defined. We further assume that the
degree of each individual is known.
To each $\cI$ individual is associated an exponential random clock with rate $\gamma$ to determine its removal.
To each open edge (directed to $\cS$), we associate a random exponential clock with
rate $\lambda$.
When it rings, an infection occurs. The infectious ego chooses the edge of a
susceptible alter at random. Hence the latter individual is chosen proportionally to her/his degree, in the size biased distribution,
as explained in \eqref{def:size-biased-distr}. When this susceptible
individual becomes infected, she/he is connected and
uncovers the edges to neighbours that were already in the subgraph: we determine whether her/his remaining edges are linked with $\cI$, $\cRr$-type individuals (already in the observed cluster)
or to $\cS$, in which case the edges remains `open' (the alter is not chosen yet).\\

Let us consider the limit when the size of the graph converges to infinity, and let us denote as before by $s_t$ and $i_t$ the proportion of susceptible and infectious individuals in the population at time $t$. A key quantity in the approach of Volz \cite{volzIII} and Miller \cite{millerIII} is the probability $\theta(t)$ that an directed edge
picked uniformly at random at $t$ has not transmitted the disease.
Let $u\in V$ be a vertex of degree $k$. The vertex $u$ is still susceptible at time $t$ if none of its $k$ edges has transmitted
the disease.
By the construction of the stochastic process, where the random graph is built simultaneously to the spread of the disease on it,
any infectious individual that transmits the disease pairs one of her/his half-edge with a half-edge of a susceptible individual chosen
uniformly at random. Thus, the probability that none of the $k$ edges of a susceptible has transmitted the disease up to time $t$ is $\theta^k(t)$.
Hence,
\begin{equation}\label{eq:volz-S}s_t=\sum_{k=0}^{+\infty} \theta(t)^k p_k=g(\theta(t)),
\end{equation}where $g$ is the generating function of the probability distribution $(p_k)_{k\geq 0}$ (see \eqref{CM:fgeneratrice}).
Notice that in Equation \eqref{eq:volz-S}, the proportion $s_t$ of susceptibles is assumed to coincide with the expectation of the proportion of
the number of susceptible individuals at $t$. We recall that a rigorous derivation of Volz' equations is given in Section \ref{sec:proof} below.

\subsection{Dynamics of $\theta(t)$}

To deduce an equation for $s_t$ from $\eqref{eq:volz-S}$, an equation for $\theta(t)$ is needed.

\begin{proposition}\label{prop:theta}We have that:
$$\frac{d\theta}{dt}=-\lambda \theta(t) +\gamma (1-\theta(t))+\lambda\frac{g'(\theta(t))}{g'(1)}. $$
\end{proposition}

\begin{proof}
Denote by $h(t)$ the probability that the alter is still susceptible
at time $t$. Define $\phi(t)$ as the probability that a random edge has not transmitted the disease and that its alter is infectious.
Notice that
\begin{equation}
\frac{d\theta}{dt}=-\lambda \phi(t).\label{eq:theta}
\end{equation}
Given an edge satisfying the definition of $\phi(t)$ (an edge that has not transmitted the disease yet and whose alter is infectious), the probability that the alter is of degree $k$ is given by \eqref{def:size-biased-distr} and given its degree, the
probability that it is still susceptible at time $t$ is $\theta^{k-1}(t)$, because the considered edge did not transmit the disease before $t$. Then:
\[h(t)=\sum_{k=0}^{+\infty}\frac{k p_k}{m} \theta^{k-1}(t)=\frac{g'(\theta(t))}{g'(1)},\]
from which we deduce that
\[\frac{dh}{dt}=\frac{g''(\theta(t))}{g'(1)}\frac{d\theta}{dt}= -\lambda \phi(t)\frac{g''(\theta(t))}{g'(1)}. \]

An equation for the evolution of $\phi(t)$ can be written by noticing that:
\begin{itemize}
\item An edge stops satisfying the definition of $\phi$ if it transmits the disease or if the alter is removed.
\item An edge starts satisfying the definition of $\phi$ if its alter becomes infectious.
\end{itemize}
Thus
\begin{align}
\frac{d\phi}{dt}= & -(\lambda+\gamma)\phi(t)-\frac{dh}{dt}\nonumber\\
= & -(\lambda+\gamma)\phi(t)+\lambda \phi(t)\frac{g''(\theta(t))}{g'(1)}\nonumber\\
= & \frac{\lambda+\gamma}{\lambda} \frac{d\theta}{dt}-\frac{g''(\theta(t))}{g'(1)} \frac{d\theta}{dt},\label{eq:phi}
\end{align}
which gives for a constant $C$:
$$\phi(t)=\frac{\lambda+\gamma}{\lambda} \theta(t)-\frac{g'(\theta(t))}{g'(1)}+C.$$
Using that $\phi(0)=0$ and $\theta(0)=1$, we deduce that $C=-\gamma/\lambda$ and hence
\begin{equation}
\label{eq:theta2}\phi(t)=\theta(t)- \frac{\gamma}{\lambda}(1-\theta(t))-\frac{g'(\theta(t))}{g'(1)}.
\end{equation}
We deduce the announced result from \eqref{eq:theta} and \eqref{eq:theta2}.
\end{proof}

\subsection{Miller's equations}

We can now deduce the equations for the proportions $s_t$, $i_t$ and $r_t$ of susceptible, infectious and recovered individuals proposed by Miller \cite{millerIII}.

\begin{proposition}[Miller's equations \cite{millerIII}]\label{prop:miller}We have:
\begin{align*}
& s_t=g(\theta(t))\\
& \frac{dr_t}{dt}=\gamma i_t\\
& \frac{di_t}{dt}=  -g'(\theta(t)) \big(-\lambda \theta(t)+\gamma (1-\theta(t))+\lambda \frac{g'(\theta(t))}{g'(1)}\big)-\gamma i_t.\\
 &\frac{d\theta}{dt}=  -\lambda \theta(t) +\gamma (1-\theta(t))+\lambda\frac{g'(\theta(t))}{g'(1)}.
\end{align*}
\end{proposition}

\begin{proof}By \eqref{eq:volz-S}, we have that $s_t=g(\theta(t))$. From the node-centered removal dynamics of infectious nodes, we have that $\frac{dr_t}{dt}=\gamma i_t$. Using $i_t=1-s_t-r_t$ and Proposition \ref{prop:theta}, we obtain the two last equations.
\end{proof}

We can now recover the equations proposed by Volz \cite{volzIII} by introducing the proportion of edges $\cI-\cS$ that have not transmitted the disease yet
\begin{equation}\label{eq:pI(t)}
p_I(t)=\frac{\phi(t)}{\theta(t)}
\end{equation}
and the proportion of edges $\cS-\cS$ that have not transmitted the disease
\begin{equation}
p_S(t)=\frac{g'(\theta(t))}{\theta(t) g'(1)}.
\end{equation}

From Miller's equations, we obtain by straightforward computation:
\begin{proposition}[Volz' equations \cite{volzIII}]\label{prop:volz}
We have:
\begin{align*}
& \theta(t) =  \exp\Big(-\lambda \int_0^t p_I(s)\;ds\Big),\qquad  s_t=  g(\theta(t)),\\
& \frac{di_t}{dt}=  \lambda p_I(t) \theta(t) g'(\theta(t))-\gamma i_t\\
& \frac{dp_I}{dt}= \lambda \,p_I(t) p_S(t)\theta(t) \frac{ g''(\theta(t))}{g'(\theta(t))}-\lambda\,p_I(t)(1-p_I(t)) -\gamma p_T(t).\\
& \frac{dp_S}{dt}=  \lambda p_I(t)p_S(t) \big(1-\theta(t)\frac{g''(\theta(t))}{g'(\theta(t))}\big).
\end{align*}
\end{proposition}

Let us compare Volz' equations with the Kermack--McKendrick equations:
\begin{align*}
\frac{ds}{dt}= & -\lambda\ s_t i_t,\qquad \frac{di}{dt}=  \lambda\ s_t i_t-\gamma i_t.
\end{align*}In Volz' equations, denoting by $\bNS_t= p_I(t) \theta(t) g'(\theta(t))$ the `quantity' of edges from $\cI$ to $\cS$:
\begin{align*}
\frac{ds_t}{dt}= & g'(\theta(t))\frac{d\theta}{dt}
=-\lambda g'(\theta(t))\theta(t) p_I(t)=-\lambda\bNS_t p_I(t)=-\lambda\bNSI_t\\
\frac{di_t}{dt}= &\lambda \times \bNSI_t-\gamma i_t.
\end{align*}
These equations account for the fact that not all the $\cI$ and $\cS$ vertices are connected, which modifies the infection pressure compared with the mixing models (Part I of this volume).

\section{Measure-valued processes}

Decreusefond et al.\ \cite{decreusefonddhersinmoyaltranIII} proved the convergence that was left open by
 Volz \cite{volzIII}. The proof that we now present underlines the key objects that lie at the core of the phenomenon: because degree distributions
 are central in CMs, these objets are not surprisingly measures representing some particular degree distributions.
 Three degree distributions are sufficient to
 describe the epidemic dynamics which evolve in the space of measures on the set of nonnegative integers, and of which Volz'
 equations are a by-product.\\
 A rigorous individual-based description of the epidemic on a random graph is provided.  Starting with a node-centered
 description, we show that the individual dimension is lost in the large graph limit. Our
 construction heavily relies on the choice of a CM for
 the graph underlying the epidemic, which was also made in \cite{volzIII}.


\subsection{Stochastic model for a finite graph with $N$ vertices}\label{sec:graphNvertices}

Recall the notation of Section \ref{sec:graphes}. The idea of Volz is to use network-centric quantities (such as the number of edges from $\cI$ to $\cS$) rather than node-centric quantities. For a vertex $u\in \cS$, $D_u$ corresponds to the degree of $u$. For $u\in \cI$ (respectively $\cRr$), $D_u(\cS)$ represents the number of edges with $u$ as infectious (resp.\ removed) ego and susceptible alter. The numbers of edges with susceptible ego (resp.\ of edges of types $\cI-\cS$ and $\cRr-\cS$) are denoted by $N^\cS_t$ (resp.\ $\NIS_t$ and $\NRS_t$). All these quantities are in fact encoded into three degree distributions, that we now introduce and on which we will work to establish Volz' equations.
Notice that with the notations of Section \ref{sec:moment}, $\frac{1}{N}\NIS_t=[SI]_t$ and $\frac{1}{N}\NRS_t=[SR]_t$. However, we drop this
notation with brackets for simplification of later formula and because we will not need motifs other than edges.

\begin{definition}We consider here the following three degree distributions of $\M_F(\Z_+)$, given for $t\geq 0$ as:
\begin{equation}
\muS_t=\sum_{u\in \cS_t}\delta_{D_u},\quad \muSI_t=\sum_{u\in \cI_t}\delta_{D_u(\cS_t)},\quad \mu^{\cRr\cS}_t=\sum_{u\in \cRr_t}\delta_{D_u(\cS_t)},\label{muS}
\end{equation}where we recall that $\delta_{D}$ is the Dirac mass at $D\in \Z_+$ (see Notation \ref{notation:part3}).
\end{definition}Notice that the measures $\muS_t/S_t$, $\muSI_t/I_t$ and $\muSR_t/R_t$ are probability measures that correspond to usual (probability) degree distributions. The degree distribution $\muS_t$ of susceptible individuals is needed to describe
the degrees of the new infected individuals. The measure $\muSI_t$ provides information on the number of edges from $\cI_t$ to
 $\cS_t$, through which the disease can propagate. Similarly, the measure $\muSR_t$ is used to describe the evolution of the
 set of edges linking $\cS_t$ to $\cRr_t$.\\
Using Notation \ref{notation:part3}, we can see that
$$I_t=\langle \muSI_t, 1\rangle,\qquad \NIS_t=  \langle \muSI_t, \chi\rangle = \sum_{u\in \cI_t}D_u(\cS_t),$$
and accordingly for $\NS_t$, $\NRS$, $S_t$ and $R_t$.

\begin{definition}[Labelling the nodes]\label{def:labelnode}For an integer-valued measure  $\mu\in \mathcal{M}_F(\Z_+)$, we can rank its atoms by increasing degrees and label them with this order. A way of deducing this
labelling from $\mu$ by using its cumulative distribution function is proposed in \cite{decreusefonddhersinmoyaltranIII}. We omit it here for the sake of simplicity.
\end{definition}

\begin{example}\label{exempleFmu}
Consider for instance the measure $\mu=2\delta_1 + 3\delta_5+\delta_7.$ If $\mu$ is a degree distribution, this means that 2 individuals have degree 1,
3 individuals have degree 5 and 1 individual has degree 7. Ranking the atoms by increasing degrees, we can label them from 1 to 6 such that
$D_1=D_2=1$, $D_3=D_4=D_5=5$, $D_6=7$.\hfill $\Box$
\end{example}

\subsection{Dynamics and measure-valued SDEs}

Suppose that at initial time, we are given a set of $\cS$ and $\cI$  nodes together with their degrees. The graph of relationships
between the $\cI$ individuals is in fact
irrelevant for studying the propagation of the disease. The minimal information consists in the sizes of the classes
$\cS$, $\cI$, $\cRr$ and the number of edges to the class $\cS$ for every infectious or removed node. Each node of
class $\cS$ comes with a
given number of half-edges of undetermined types ; each node of class $\cI$ (resp.\ $\cRr$) comes with
a number of $\cI-\cS$ (resp.\ $\cRr-\cS$) edges. The numbers of $\cI-\cRr$, $\cI-\cI$ and $\cRr-\cRr$ edges need not to be retained.
The three descriptors in \eqref{muS} are hence sufficient to describe
the evolution of the SIR epidemic. \\

Recall the graph construction of Section \ref{sec:volzmiller-graphdescription} explaining how to handle simultaneously the two sources of randomness of the problem. The
random network of social relationships is explored while the disease spreads on it: only the clusters of $\cI$ and $\cRr$ individuals are observed and constructed, with $\cI-\cS$ and $\cRr-\cS$ edges having their $\cS$ alter still unaffected. Susceptible individuals remain unattached until they become infected, in which case their connections to the cluster of $\cI$'s and $\cRr$'s are revealed. We assume that the degree distribution of $\cS_0$ and the size $N$ of the total population are known.\\

We now explain the dynamics, that is summarized in Figure \ref{fig:explication}. Recall that to each half-edge of type $\cI-\cS$, an
independent exponential clock with parameter $\lambda$ is associated, and
to each $\cI$ vertex, an independent exponential clock
with parameter $\gamma$ is associated. The first of all these clocks that rings
determines the next event.
\begin{description}
\item[Case 1] If the clock that rings is associated to an $\cI$ individual, the latter is removed. Change her status
from $\cI$ to $\cRr$ and the type of her emanating half-edges accordingly: $\cI-\cS$ half-edges become $\cRr-\cS$ half-edges for example.
\item[Case 2] If the clock that rings is associated with a half $\cI-\cS$-edge (with unaffected susceptible alter), an infection occurs.
  \begin{description}
\item[Step 1] Match randomly the $\cI-\cS$-half-edge whose clock has rung to a half-edge of a susceptible: this determines the susceptible becoming infected.
\item[Step 2] Let $k$ be the degree of the newly infected individual.
Choose uniformly $k-1$ half edges among the open half-edges of the cluster of $\cI$ and $\cRr$ individuals ($\cI-\cS$ or $\cRr-\cS$ edges of this cluster, with susceptible alter still unaffected) and among the half edges of susceptible individuals. Let $j$, $\ell$ and $m$ be the respective number of $\cI-\cS$, $\cRr-\cS$ and $\cS-\cS$ edges chosen among the $k-1$ picks.
\item[Step 3] The chosen half-edges of type $\cI-\cS$ and $\cRr-\cS$ determine the infectious or removed neighbours of the newly infected individual who become the new (infectious) alter of these edges. The remaining $m$ edges of type $\cS-\cS$ remain open in the sense that the susceptible neighbour is not fixed. \\
    Change the status of the newly infected from $\cS$ to $\cI$. Change the status of the $m$ (resp.\ $j$, $\ell$) $\cS-\cS$-type (resp.\ $\cI-\cS$-type,
$\cRr-\cS$-type) edges considered to $\cI-\cS$-type (resp.\ $\cI-\cI$-type, $\cRr-\cI$-type).\hfill $\Box$
\end{description}
\end{description}
We then wait for another clock to ring and repeat the procedure.

\begin{figure}[!ht]
  \begin{center}
    \begin{tabular}{|c|c|c|}
      \hline
      (a) & (b) & (c) \\
      \hline
      \includegraphics[width = 0.3\textwidth,trim=0cm 0cm 0cm 0cm]{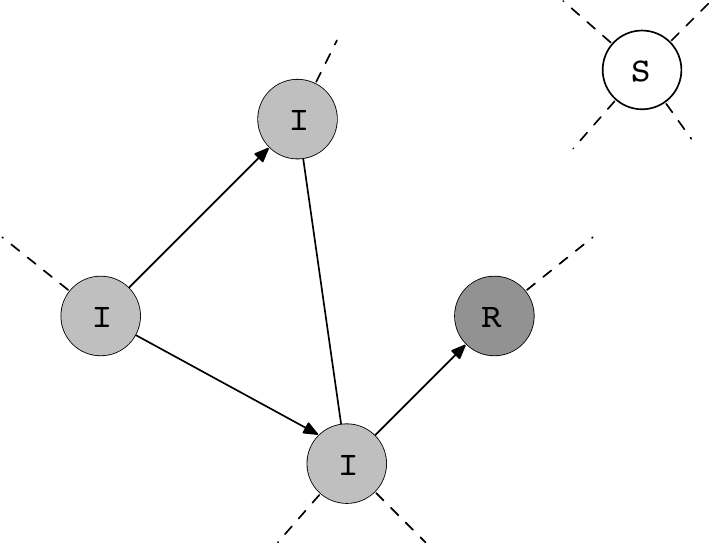} &
      \includegraphics[width = 0.3\textwidth,trim=0cm 0cm 0cm 0cm]{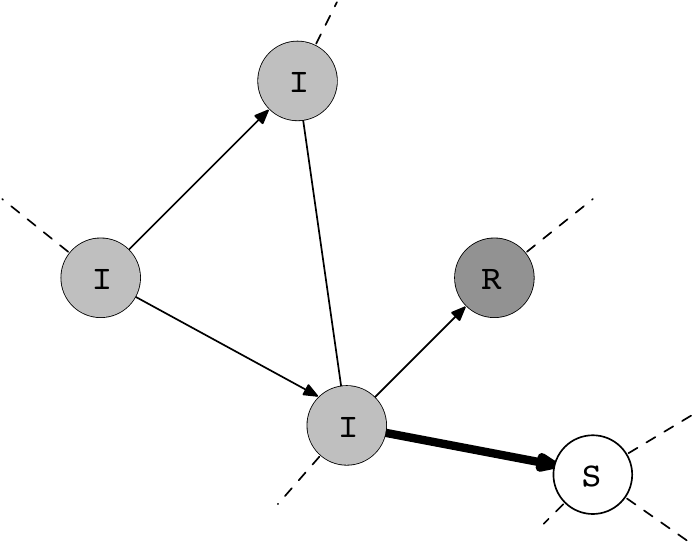} &
      \includegraphics[width = 0.3\textwidth,trim=0cm 0cm 0cm 0cm]{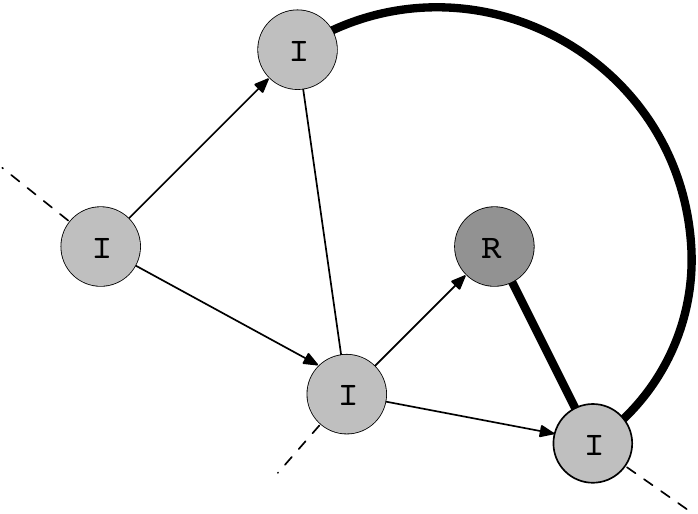}\\
      \hline\end{tabular}
      \caption{{\small \textit{Infection process. Arrows provide the infection tree. Susceptible, infectious and removed
      individuals are colored in white, grey and dark grey respectively. (a) The degree of each individual is known, and
      for each infectious (resp.\ removed) individual, we know his/her number of edges of type $\cI\cS$ (resp.\ $\cRr\cS$).
       (b) A contaminating half edge is chosen and a susceptible of degree $k$ is infected
      at time $t$ with the rate $\Lambda_t(k)$ defined in (\ref{eq:deflambdan}). The contaminating edge is drawn in bold line. The
      number
 $\NIS_{t_-}$ of edges from $\cI$ to $\cS$ momentarily becomes $\NIS_{t_-}-1+(k-1)$. (c) Once the susceptible individual has
  been infected, we determine how many of its remaining arrows are linked to the classes $\cI$ and $\cRr$. If we denote by $j$
   and $\ell$ these numbers, then $\NIS_t=\NIS_{t_-}-1+(k-1)-j-\ell$ and $\NRS_t=\NRS_{t_-}-\ell$.}}}
   \label{fig:explication}
      \end{center}
\end{figure}

From the dynamics described above, we can read that the global force of infection at time $t$ is
\[\lambda \NIS_{t_-}.\]
When an infection occurs, a half-edge of a susceptible individual is chosen and determines who is the contaminated person.
Therefore, a given susceptible of degree $k$ has a probability $k/\NS_{t_-}$ to be the next infected individual.
So that the rate of infection of a given susceptible of degree $k$ at time $t$ is:
\begin{equation}
\label{eq:deflambda}
\Lambda_{t_-}(k)=\lambda k\frac{\NIS_{t_-}}{\NS_{t_-}}=\lambda k p_I(t_-),
\end{equation}
where $p_I(t)$ is defined by
\begin{align*}
p_I(t)=\frac{\NIS_t}{\NS_t},
\end{align*}is the proportion of edges linked to susceptible individuals that can transmit the disease. It is the discrete stochastic quantity that we expect will converge to \eqref{eq:pI(t)}. \\

Starting from $t$, and because of the properties of
exponential distributions, the next event will take place after an
exponentially distributed time with parameter $\lambda\NIS_t+\gamma I_t.$ Let $T$
denote the time of this event after $t$.
\begin{description}
\item[Case 1] The next event corresponds to a removal, i.e., a node
  goes from status $\cI$ to status $\cRr$. Choose uniformly $u\in I_{T^-}$ (with probability $1/I_{T^-}$, then update the measures $\muSI_{T_-}$ and $\muSR_{T_-}$:
  \begin{equation*}
    \muSI_T=\muSI_{T^-} - \delta_{D_u(\cS_{T_-})} \text{ and }\,
    \muSR_T=\muSR_{T^-}+\delta_{D_u(\cS_{T_-})}.
  \end{equation*}

\item[Case 2] The next event corresponds to a new infection.
We choose uniformly a half-edge with susceptible alter, and this alter becomes infectious. The new infective has degree $k$ with probability
$k\muS_{T_-}(k)/\NS_{T_-}$.
%
When the new individual is `discovered' by the disease, she/he reveals her/his links with other infectious or removed individuals. The probability, given that the degree of the individual is $k$ and that
$j$ (resp.\ $\ell$) out of her $k-1$ other half-edges (all but the contaminating $\cI\cS$ edge) are chosen to be
of type $\cI\cI$  (resp.\ $\cI\cRr$), according to Step 2', is given by the following multivariate hypergeometric distribution:
\begin{equation}
\label{eq:defp}
p_{T_-}(j,\ell\,|\, k-1)=\frac{{\NIS_{T_-}-1 \choose j} {\NRS_{T_-} \choose \ell} {\NS_{T_-}-\NIS_{T_-}-\NRS_{T_-} \choose k-1-j-\ell}}{{\NS_{T_-}-1 \choose k-1}}\cdotp
\end{equation}
Finally, to update the values of $\muSI_T$ and $\muSR_T$ given $k$, $j$ and $\ell$, we have to choose the infectious and removed individuals to which the newly infectious is linked: some of their edges, which were $\cI\cS$ or $\cRr\cS$, now become $\cI\cI$ or $\cRr\cI$.
We draw two sequences of integers $\underline{n}=(n_1,\dots,n_{I_{T_-}})$ and $\underline{m}=(m_1,\dots,m_{R_{T_-}})$ that will
indicate how many links each infectious or removed
individual has to the newly contaminated individual. There exist constraints on these sequences: the number of edges recorded for each individual by the vectors $\underline{n}$ and $\underline{m}$
can not exceed the number of existing edges.
Let us define the set
\begin{equation}
\label{eq:defU}
   \mathcal{L}=\bigcup_{m=1}^{+\infty} \Z_+^m,
 \end{equation}
and for all finite integer-valued measure $\mu$ on $\Z_+$, corresponding to a degree distribution as in Section \ref{sec:graphNvertices}, and whose atoms are labelled say, according to Definition \eqref{def:labelnode} and for all integer $\ell \in \Z_+$, we define the subset
\begin{multline}
\label{eq:defU2}
\mathcal{L}(\ell,\mu)=\Big\{\underline{n}=(n_1,...,n_{\langle \mu,\ind\rangle}) \in \Z_+^{\langle \mu,\ind\rangle}\quad  \mbox{ such that }\\
\forall u\in \{1,\dots,\langle \mu,\ind\rangle\},\, n_u\leq D_u(\mu)\mbox{ and }\sum_{u=1}^{\cro{\mu,\mathbf 1}} n_u=\ell\Big\},
\end{multline}where $D_u(\mu)$ stands for the degree of the vertex $u$, read from the measure $\mu$ (see Example \ref{exempleFmu}).
Each sequence $\underline{n} \in \mathcal{L}(\ell,\mu)$ provides a possible configuration of how the $\ell$ connections of a given individual
can be shared between neighbours whose degrees are summed up by $\mu$. The component $n_u$, for $1\leq u\leq \langle \mu, 1\rangle$, provides
the number of edges that this individual shares with the individual $u$. This number is necessarily smaller than the degree $D_u(\mu)$ of individual $u$.
Moreover, the components of the vector $\underline{n}$ sum to $\ell$. The probabilities of the draws of $\underline{n}$ and $\underline{m}$ that provide respectively the number of edges
 $\cI-\cS$ which become $\cI-\cI$ per infectious individual and the number of edges $\cRr-\cS$ which become $\cRr-\cI$ per removed individual are given by:
\begin{align}
& \rho(\underline{n} | j+1,\muSI_{T_-})=  \frac{\prod_{u \in \cI_{T_-}}{D_u(\cS_{T_-}) \choose n_u}}{{ \NIS_{T_-} \choose j+1}}
\ind_{\underline{n} \in\mathcal{L}(j+1,\muSI_{T_-})}\nonumber\\
 & \rho(\underline{m} |\ell ,\muSR_{T_-})=  \frac{\prod_{v\in \cRr_{T_-}}{ D_v(\cS_{T_-})  \choose m_v}}{{ \NRS_{T_-} \choose \ell}}
 \ind_{\underline{m}\in\mathcal{L}(\ell,\muSR_{T_-})}.\label{etape14}
\end{align}
Note that $I_{T_-}=\langle \muSI_{T_-},1\rangle$ is the total mass of the measure $\muSI_{T_-}$ and that $D_u(\cS_{T_-})$ corresponds to the degree of the
individual $u$ encoded by $\muSI_{T_-}$ with the labelling of Definition \ref{def:labelnode}, i.e.\  to the number of edges from $u$ to $\cS$ before time $T$.\\

Then, we update the measures as follows:
\begin{align}
\muS_T&=\mu^\cS_{T^-}-\delta_k \nonumber\\
\muSI_T&=\muSI_{T^-} +\delta_{k-1-j-\ell}+\sum_{u\in \cI_{T_-}} \big(\delta_{D_u(\cS_{T_-})-n_{u}}-\delta_{D_u(\cS_{T_-})}\big)\nonumber\\
\muSR_T&=\muSR_{T^-} +\sum_{v\in \cRr_{T_-}} \big( \delta_{D_{v}(\cS_{T_-})-m_{v}}-\delta_{D_{v}(\cS_{T_-})}\big).\label{etape1546}
\end{align}
\end{description}

Here, we propose stochastic differential equations (SDEs) driven by Poisson point measures (PPMs) to describe the evolution of
the degree distributions (\ref{muS}) as in \cite{decreusefonddhersinmoyaltranIII}.\\

We consider two Poisson point measures $Q^1$ and $Q^2$ on
$E_1:= \Z_+ \times \R_+ \times \Z_+ \times \Z_+\times \R_+  \times \mathcal{L}\times \R_+\times\mathcal{L}\times \R_+ $ and $\R_+\times \Z_+$ with
intensity measures the product of Lebesgue measures on $\R_+$ and the of counting measures on each discrete set.
The atoms of the point measure $Q^1$ are of the form $(s,k,\theta_1,j,\ell,\theta_2, \underline{n},\theta_3,\underline{m},\theta_4)$. They provide possible times $s$ at which an infection may occur, and gives an integer $k$ corresponding to the degree of the susceptible being possibly infected, the numbers $j+1$ and
 $\ell$ of edges that this individual has to the sets $\cI_{s_-}$ and $\cRr_{s_-}$. The marks $\underline{n}$ and $\underline{m}\in \mathcal{L}$ are as in the
 previous section. The marks $\theta_1$, $\theta_2$ and $\theta_3$ are auxiliary variables used for the construction
 (see (\ref{eqmuI})--(\ref{eqmuR})) below.\\
The atoms of the point measure $Q^2$ are of the form $(s,u)$ and give possible removal times $s$ associated with
the label $u$ of the individual that may be removed.\\

The following SDEs describe the evolution of the epidemic: for all $t\ge 0$,
\begin{align}
\label{eqmuS}
&  \muS_t =  \muS_0  - \int_0^t\int_{E_1} \delta_{k}\ind_{\theta_1\leq \Lambda_{s_-}(k)\muS_{s_-}(k)}\\
& \hspace{3.5cm}\ind_{\theta_2\leq p_{s_-}(j,\ell|k-1)} \ind_{\theta_3\leq \rho(\underline{n}|j+1,\muSI_{s_-})}  \ind_{\theta_4\leq \rho(\underline{m}|\ell,\muSR_{s_-})}
\d Q^1 \nonumber\\
\label{eqmuI}
&  \muSI_t=  \muSI_0+\int_0^t\int_{E_1}  \Big(\delta_{k-(j+1+\ell)}+\sum_{u\in \cI_{s_-}}\big(\delta_{D_u(\muSI_{s_-})-n_u}-\delta_{D_u(\muSI_{s_-})}\big)\Big)\\
   & \hspace{1cm} \times \ind_{\theta_1\leq \Lambda_{s_-}(k)\mu_{s_-}^\cS(k)}\ind_{\theta_2\leq p_{s_-}(j,\ell|k-1)} \ind_{\theta_3\leq \rho(\underline{n}|j+1,\muSI_{s_-})}
\ind_{\theta_4\leq \rho(\underline{m}|\ell,\muSR_{s_-})}\,
\d Q^1 \nonumber\\
 & \hspace{7.5cm} - \int_0^t\int_{\Z_+} \delta_{D_u(\muSI_{s_-})} \ind_{u\in  \cI_{s_-}}\d Q^2 \nonumber\\
&  \muSR_t= \muSR_0+ \int_0^t\int_{E_1} \Big(\sum_{v\in \cRr_{s_-}} \big(\delta_{D_{v}(\muSR_{s_-})-m_{v}}-\delta_{D_{v}(\muSR_{s_-})}\big)\Big)\label{eqmuR}\\
 &  \hspace{1cm}\times \ind_{\theta_1\leq \Lambda_{s_-}(k) \mu^{\cS}_{s_-}(k)}\ind_{\theta_2\leq
    p_{s_-}(j,\ell|k-1)}
\ind_{\theta_3\leq \rho(\underline{n}|j+1,\muSI_{s_-})}
\ind_{\theta_4\leq \rho(\underline{m}|\ell,\muSR_{s_-})}\,\d Q^1\nonumber\\
 & \hspace{7.5cm} + \int_0^t\int_{\Z_+} \delta_{D_u(\muSI_{s_-})} \ind_{u\in \cI_{s_-}}\d Q^2,\nonumber
\end{align}
where we write $\d Q^1$ and $\d Q^2$ instead of $\d Q^1(s,k,\theta_1,j,\ell,\theta_2, \underline{n},\theta_3,\underline{m},\theta_4)$ and $\d Q^2(s,u)$
to simplify the notation.

\begin{proposition}For any given initial conditions $\mu^\cS_0$, $\mu^{\cS\cI}_0$ and $\muSR_0$ that are integer-valued measures on $\Z_+$ and for
  PPMs $Q^1$ and $Q^2$, there exists a unique strong
  solution to the SDEs (\ref{eqmuS})--(\ref{eqmuR}) in the space
  $\D\big(\R_+,(\mathcal{M}_F(\Z_+))^3\big)$, the Skorokhod space of c\`{a}dl\`{a}g functions with values in $(\mathcal{M}_F(\Z_+))^3$.
\end{proposition}

\begin{proof}For the proof, we notice that for every $t\in \R_+$, the
  measure $\muS_t$ is dominated by $\muS_0$ and the measures
  $\muSI_t$ and $\muSR_t$ have a mass bounded by
  $\langle \muS_0+\muSI_0+\muSR_0,1\rangle$ and a support included in
  $\left[\!\!\left[
      0,\max\{\max(\supp(\muS_0)),\max(\supp(\muSI_0)),\max(\supp(\muSR_0))\}\right]\!\!\right]$. The
  result then follows the steps of \cite{fourniermeleardIII} and
  \cite{chitheseIII} (Proposition 2.2.6) where a pathwise construction of the solution on the positive real line is given using the Poisson point processes $Q^1$ and $Q^2$.
\end{proof}

The course of the epidemic can be deduced from (\ref{eqmuS}), (\ref{eqmuI}) and (\ref{eqmuR}). For the sizes
$(S_t,I_t,R_t)_{t\in \R_+}$ of the different classes, for instance, we have with the choice of $f\equiv 1$ that
for all $t\ge 0$, $S_t=\langle \mu^\cS_t,\mathbf{1} \rangle,$ $  I_t= \langle \muSI_t,\mathbf 1\rangle$ and
$  R_t= \langle \muSR_t,\mathbf 1\rangle$ (see Notation \ref{notation:part3}). Writing the semi-martingale decomposition that results from standard
stochastic calculus for jump processes and SDE driven by PPMs (e.g.\ \cite{fourniermeleardIII,ikedawatanabeIII, jacodIII}),
we obtain for example:
\begin{align}
  I_t= & \langle \muSI_t,\mathbf 1\rangle= I_0+\int_0^t\Big( \sum_{k\in \Z_+}
  \muS_s(k)\Lambda_s(k)
  - \gamma \,I_s\Big) \d s+M^{\cI}_t,\label{eqI_t}
\end{align}
where $M^\cI$ is a square-integrable martingale that can be written explicitly as a stochastic integral with respect to the compensated PPMs of $Q^1$ and $Q^2$, and with predictable quadratic variation given for all $t\ge 0$ by
\begin{align*}
  \langle M^\cI\rangle_t= & \int_0^t \sum_{k\in \Z_+} \Big(\muS_s(k)\Lambda_s(k)+\gamma I_s\Big)\d s.
\end{align*}
\noindent Other quantities of interest are the numbers of edges of the different types $NS_t$, $\NIS_t$, $\NRS_t$.
The latter appear as the first moments of the measures $\muS_t$, $\muSI_t$ and $\muSR_t$:
\[NS_t=  \langle \muS_t,\chi\rangle,\quad \NIS_t= \langle \muSI_t,\chi\rangle\quad \mbox{ and }\quad \NRS_t= \langle \muSR_t,\chi\rangle.\]

\subsection{Rescaling}\label{sectionrenorm}

We consider a sequence of larger and larger graphs $(\mathcal{G}_N)_{N\geq 1}$ with $N\rightarrow +\infty$. The degree distribution
$\mathbf{p}$ underlying these CM graphs remains unchanged with $N$.\\
The sequences of measures $(\muSn)_{N\in \N }$,
$(\muISn)_{N\in \N }$ and $(\muRSn)_{n\in \N }$
are defined as
\begin{equation}\label{eqmuSrenorm}
\muSn_t= \frac{1}{N} \muS_t ,\qquad \muSIn_t=\frac{1}{N}\muSI_t,\qquad \muSRn_t=\frac{1}{N}\muSR_t
\end{equation}where the measures non-rescaled $\muS$, $\muSI$ and $\muSR$ are defined as in \eqref{muS} and implicitly depend on $N$:
\[ \langle \muSn_t,1\rangle +\langle \muISn,1\rangle+\langle \muRSn,1\rangle =\frac{N}{N}=1.\]
The proportions $\Sn_t$, $\In_t$ and $\Rn_t$ defined in \eqref{renorm:S-I-R} can then be rewritten as
$\Sn_t=\langle \muSn,1\rangle$, $\In_t=\langle \muISn,1\rangle$ and $\Rn_t=\langle \muSRn,1\rangle$. Also, we have $\NSnn_t=\langle \muSn,\chi\rangle$,
$\NISnn_t=\langle \muISn,\chi\rangle$ and $\NRSnn_t=\langle \muRSn,\chi\rangle $, the numbers, renormalized by $N$, of edges with susceptible ego, infectious ego and susceptible alter, removed ego and
susceptible alter.\\

We assume that the initial conditions satisfy:
\begin{assumption}\label{hypconvcondinit}
  The sequences $(\muSn_0)_{n\in \N }$,
  $(\muSIn_0)_{n\in \N }$ and $(\muSRn_0)_{n\in \N}$ converge to measures
  $\bmuS_0$, $\bmuSI_0$ and $\bmuSR_0$ in $\mathcal{M}_F(\Z_+)$
  equipped with the topology of weak convergence.
\end{assumption}

\begin{remark}\label{rqueconvcondinit}
1. Assumption \ref{hypconvcondinit} entails that the initial (susceptible and infectious) population size is of order $N$ if $\bmuS_0$ and $\bmuSI_0$ are nontrivial.\\
2. If the distributions underlying the measures $\muSn_0$, $\muSIn_0$ and $\muSRnn_0$ do not
    depend on the total number of vertices (e.g.\ Poisson, power-laws or
    geometric distributions), Assumption
    \ref{hypconvcondinit} can be viewed as a law of large
    numbers. When the distributions depend on the total number of
    vertices $N$ (as in Erd\"{o}s-Renyi graphs), there may be
    scalings under which Assumption \ref{hypconvcondinit}
    holds. For Erd\"{o}s-Renyi graphs for instance, if the probability
    $p_N$ of connecting two vertices satisfies $\lim_{N\rightarrow
      +\infty} N p_N=\lambda$, then we obtain in the limit a Poisson
    distribution with parameter $\lambda$.\\
3. Notice the
    appearance in Equation (\ref{eq:deflambda}) of the size biased degree distribution.  The latter reflects the fact
    that, in the CM, individuals having large degrees have higher
    probability to connect than individuals having small
    degrees. Thus, there is no reason why the degree distributions of
    the susceptible individuals $\bmuS_0/\bar{S}_0$
    and the distribution $\sum_{k\in \Z_+}p_k\delta_k$ underlying the CM
    should coincide. This is developed in Section \ref{sec:deg-init-cond}. \hfill $\Box$
\end{remark}

It is possible to write rescaled SDEs which are the same as the SDEs (\ref{eqmuS})--(\ref{eqmuR}) parameterized by $N$
(see \cite{decreusefonddhersinmoyaltranIII} for details).
Several semi-martingale decompositions will be useful in the
sequel. We focus on $\muISn$ but similar decompositions hold
for $\muSn$ and $\muRSn$, which we do not detail since
they can be deduced by direct adaptation of the computation which follows.
\begin{proposition}
Define:
\begin{equation}
\label{eq:deflambdan}
\Lambda^N_s(k)=\lambda k\frac{\NISnn_s}{\NSnn_s},\mbox{ and } p^N_s(j,\ell \mid k-1)
=\frac{{\NISnn_{s}-1 \choose j} {\NRSnn_{s} \choose \ell} {\NSnn_{s}-\NISnn_{s}-\NRSnn_{s}\choose k-1-j-\ell}}{{\NSnn_{s}-1 \choose k-1}}.
\end{equation}
  For all $f \in\mathcal B_b(\Z_+)$, for all $t\ge 0$,
  \begin{equation}
    \langle \muISn_t,f\rangle=\sum_{k\in \Z_+ }f(k) \muISn_0(k)+A^{\textsc{n},\cI\cS,f}_t+M^{\textsc{n},\cI\cS,f}_t,\label{musif}
  \end{equation}where the finite variation part $A^{\textsc{n},\cI\cS,f}_t$ of $\langle \mu^{\textsc{n},\cI\cS}_t,f\rangle$ reads
  \begin{multline}
    \label{defA(n)SI}
    A^{\textsc{n},\cI\cS,f}_t = \int_0^t \sum_{k\in \Z_+}\Lambda^N_{s}(k)\muSn_s(k)
\sum_{j+\ell+1\leq k}p_s^N(j,\ell|k-1) \sum_{\underline{n}\in \mathcal{L}} \rho(\underline{n}|j+1,\muISnn_{s}) \\
\times \Big(f(k-(j+1+\ell))
    +
  \sum_{u\in \In_s}  \big(f(D_u(\cS_s) - n_u)-f(D_u(\cS_{s}))\big)\Big)\d s\\
   -  \int_0^t \gamma\langle \muISn_s,f\rangle \d s,
  \end{multline}
and where the martingale part $M^{\textsc{n},\cI\cS,f}_t$ of
  $\langle \mu^{\textsc{n},\cI\cS}_t,f\rangle$ is a square integrable
  martingale starting from 0 with quadratic variation
  \begin{multline*}
    \langle
    M^{N,\cI\cS,f}\rangle_t 
    = \frac{1}{N}\int_0^t \gamma \langle \muISn_s,f^2\rangle \d s\\
    +\frac{1}{N} \int_0^t \sum_{k\in \Z_+}
    \Lambda^N_s(k)\muSn_s(k)\sum_{j+\ell+1\leq
      k}p^N_s(j,\ell|k-1) \sum_{\underline{n}\in \mathcal{L}}
    \rho(\underline{n}|j+1,\muISnn_{s})\\
    \times \Big(f\left(k-(j+1+\ell)\right)+\sum_{u\in \In_s}\left(f\left(D_u(\muISnn_{s}) - n_u\right)
    -f\left(D_u(\muISnn_{s})\right)\right)\Big)^2 \d s.
  \end{multline*}
\end{proposition}

\begin{proof}
  The proof proceeds from standard stochastic
  calculus for jump processes, using the SDEs driven by Poisson point processes (see the appendices of Part I of this volume or \cite{decreusefonddhersinmoyaltranIII,ikedawatanabeIII}).
\end{proof}

\subsection{Large graph limit}\label{sectionlargegraph}

We prove that the rescaled degree distributions
mentioned above can then be approximated for large $N$, by the solution
 $(\bmuS_t,\bmuSI_t,\bmuSR_t)_{t\geq 0}$ of
 a system of deterministic measure-valued equations, with initial
 conditions $\bmuS_0$, $\bmuSI_0$ and $\bmuSR_0$. \\

We denote by $\bar{S}_t$ (resp.\ $\bar{I}_t$ and $\bar{R}_t$) the mass of the measure $\bmuS_t$ (resp.\ $\bmuSI_t$ and $\bmuSR_t$). As for the finite graph, $\bmuS_t/\bar{S}_t$
    (resp.\ $\bmuSI_t/\bar{I}_t$ and $\bmuSR_t/\bar{R}_t$) is the probability degree distribution
    of the susceptible individuals (resp.\ the probability distribution of the degrees of the infectious and removed individuals
    towards the susceptible ones). For all $t\ge 0$, we denote by $\bNS_t=\langle \bmuS_t,\chi\rangle$
 (resp.\ $\bNIS_t=\langle \bmuSI_t,\chi\rangle$ and $\bNRS_t=\langle \bmuRS_t,\chi\rangle$) the continuous number of edges
 with ego in $\cS$ (resp.\ $\cI-\cS$ edges, $\cRr-\cS$ edges). Following Volz \cite{volzIII}, pertinent quantities are the proportions
 $\bpI_t=\bNIS_t/\bNS_t$ (resp.\ $\bpR_t=\bNRS_t/\bNS_t$ and $\bpS_t=(\bNS_t-\bNIS_t-\bNRS_t)/\bNS_t$)
 of edges with infectious (respectively removed, susceptible) alter among those having susceptible ego. We also introduce
 \begin{equation}
 \theta_t = \exp\Big(-\lambda \int_0^t \bpI_s \d s\Big)\label{def:theta}
 \end{equation}
the probability that a degree one node remains susceptible until time $t$. The limiting measure-valued equation expresses for
 any bounded real function $f$ on $\Z_+$ as:
 \begin{align}
   \langle \bmuS_t,\, f\rangle= &   \sum_{k\in \Z_+} \bmuS_0(k)\, \theta^k_t f(k),\label{limitereseauinfiniS}\\
     \langle \bmuSI_t,\, f\rangle= &   \langle
   \bmuSI_0,\, f\rangle-\int_0^t \gamma \langle
   \bmuSI_s,\, f\rangle \d s \label{limitereseauinfiniSI}\\
  + & \int_0^t \sum_{k\in \Z_+} \lambda k \bpI_s  \sum_{\substack{j,\, \ell,\, m\in \Z_+\\
         j+\ell+m=k-1}}
         \binom{k-1}{j,\ell,m}
 (\bpI_s)^{j}(\bpR_s)^{\ell}(\bpS_s)^{m}
 f(m)\bmuS_s(k)\d s\nonumber\\
 + &\int_0^t \sum_{k\in \Z_+} \lambda k \bpI_s (1+(k-1) \bpI_s)  \sum_{k'\in \N}  \big(f(k'-1)-f(k')\big) \frac{k'\bmuSI_s(k')}{\bNIS_s}  \bmuS_s(k)\d s,
\nonumber\\
   \langle \bmuSR_t,\, f\rangle= & \langle \bmuSR_0,\, f\rangle+\int_0^t \gamma \langle \bmuSI_s,\,
   f\rangle \d s  \label{limitereseauinfiniSR}\\
   + & \int_0^t \sum_{k\in \Z_+} \lambda k\bpI_s (k-1) \bpR_s \sum_{k'\in \N}  \big(f(k'-1)-f(k')\big) \frac{k'\bmuRS_s(k')}{\bNRS_s} \bmuS_s(k)\d s.\nonumber
 \end{align}

Let us give a heuristic explanation of Equations \eqref{limitereseauinfiniS}--\eqref{limitereseauinfiniSR}. Notice that the limiting
graph is infinite. The probability that an individual of degree $k$ has been infected by none of her $k$ edges
is $\theta_t^k$ and Equation \eqref{limitereseauinfiniS} follows. In Equation \eqref{limitereseauinfiniSI}, the first integral
corresponds to infectious individuals being removed. In the second integral, $\lambda k\bpI_s$ is the rate of infection of a given
susceptible individual of degree $k$. Once she gets infected, the multinomial term determines the number of edges connected
to susceptible, infectious and removed neighbours. Multi-edges are not encountered in the limiting graph. Each infectious neighbour has a degree chosen according to
the size-biased distribution $k' \bmuSI(k')/\bNIS$ and the number of edges to $\cS$ is reduced by 1. This explains the third
integral. Similar arguments explain Equation \eqref{limitereseauinfiniSR}.\\

Before stating the theorem, let us introduce the following state space.
For any $\varepsilon \geq 0$ and $A>0$, we define the following closed set of
$\M_F(\Z_+)$ as
\begin{equation}
 \mathcal{M}_{\varepsilon,\, A}=\{\nu\in \mathcal{M}_F(\Z_+)\; ; \;
 \langle \nu,\ind+\chi^5\rangle\leq A\text{ and } \langle \nu,\,
 \chi\rangle\ge \varepsilon\}  \label{eq:defM}
\end{equation}and $\M_{0+,A}=\cup_{\varepsilon>0}\M_{\varepsilon,A}$.

\begin{theorem}\label{propconvergencemunS}
Suppose that Assumption \ref{hypconvcondinit} holds and that there exists an $A>0$ such that
\begin{equation}
\label{eq:hypmoments}
\big(\muSn_0,\,\muSIn_0,\, \muSRn_0\big) \mbox{ in } (\M_{0, A})^3\mbox{ for any }N,\mbox{ with }\langle \bmuSI_0,\chi\rangle>0.
\end{equation}
Then, as $N$ converges to infinity, the sequence
  $(\muSn,\muISn,\muRSn)_{N\in \N }$
  converges in distribution in $\D(\R_+,\M_{0,A}^3)$ to $(\bmuS,\bmuIS,\bmuRS)$ which is the unique solution
 of the deterministic system equations
  (\ref{limitereseauinfiniS})--(\ref{limitereseauinfiniSR}) in $\Co(\R_+,\M_{0,A}\times \M_{0+,A}\times \M_{0,A})$.
\end{theorem}

The proof is detailed in Section \ref{sec:proof} and follows standard arguments. First, tightness of the process is proved using the Roelly and Aldous--Rebolledo criteria \cite{roellyIII,joffemetivierIII}.
Then, the convergence of the generators is studied, which allows us to identify the limit, provided the number of edges $\cS-\cI$ remains of order at least $\varepsilon N$.
For proving uniqueness of the limiting value, we show using Gronwall's lemma that any two solutions of the limiting equation have the same mass and the same moments of order 1 and 2.
This allows us to show the uniqueness of the generating function of $\bmuSI$ which solves a transport equation.\\

The assumption of moments of order 5 are needed for the convergence of the generators and discussed in Section \ref{sec:deg-init-cond}.


\subsection{Ball--Neal and Volz' equations}

Choosing $f(k)=\ind_{i}(k)$, we obtain the
 following countable system of ordinary differential equations (ODEs).
 \begin{align}
 \bmuS_t(i)=  &  \bmuS_0(i) \theta^i_t,\nonumber\\
   \bmuSI_t(i)= &   \bmuSI_0(i) - \int_0^t \gamma \bmuSI_s(i) \d s\nonumber\\
  + &   \int_0^t  \lambda \bpI_s  \sum_{j,\ell\geq 0} (i+j+\ell+1)\bmuS_s(i+j+\ell+1) {i+j+\ell \choose i,j,\ell}
   (\bpS_s)^i(\bpI_s)^j (\bpR_s)^\ell \d s\nonumber\\
  +  &  \int_0^t  \biggl(\lambda(\bpI_s)^2\langle \bmuS_s,\chi^2-\chi\rangle+ \lambda\bpI_s \langle \bmuS_s,\chi\rangle \biggl)
\frac{(i+1)\bmuSI_s(i+1)-i\bmuSI_s(i)}{\langle \bmuSI_s,\chi\rangle} \d s,\nonumber\\
  \bmuSR_t(i) =  & \bmuSR_0(i) \nonumber\\
  + &  \int_0^t \Biggl\{ \beta
   \bmuSI_s(i)+ \lambda \bpI_s \langle
   \bmuS_s,\chi^2-\chi\rangle \bpR_s
   \frac{(i+1)\bmuSR_s(i+1)-i\bmuSR_s(i)}{\langle
     \bmuSR_s,\chi\rangle}\Biggl\} \d s,
   \label{systemelong}
 \end{align}
 It is noteworthy to say that this system corresponds to that in Ball and Neal \cite{ballnealIII}.\\

 The system (\ref{limitereseauinfiniS})--(\ref{limitereseauinfiniSR})
 allows us to recover the equations proposed by Volz \cite[Table 3, p.~297]{volzIII} (see also Proposition \ref{prop:volz}). The latter are obtained directly from \eqref{limitereseauinfiniS}--\eqref{limitereseauinfiniSR} and the definitions of $\bar{S}_t$, $\bar{I}_t$, $\bpI_t$ and $\bpS_t$ which relate these quantities to the measures $\bmuS_t$ and $\bmuSI_t$. Let
\begin{equation}
h(z)=\sum_{k\in
     \Z_+}\bmuS_0(k) z^k
\end{equation}be the generating function for the
   initial degree distribution of the susceptible individuals
   $\bmuS_0$. This generating function is \textit{a priori} different from the generating function of the degree distribution of the total CM graph:
   $g(z)=\sum_{k\in \Z_+}p_k z^k$. Let also $\theta_t= \exp(-\lambda \int_0^t \bpI_s \d s)$.
   Then:
   \begin{align}
     \bar{S}_t= & \langle \bmuS_t,\mathbf 1\rangle= h(\theta_t),\label{volz1}\\
     \bar{I}_t= & \langle \bmuIS_t,\mathbf 1\rangle = \bar{I}_0+\int_0^t \Big( \lambda \bpI_s \theta_s h'(\theta_s)
     -\gamma \bar{I}_s\Big) \d s,\label{volz2}\\
     \bpI_t= & \bpI_0
     +  \int_0^t \Big( \lambda \,\bpI_s \bpS_s\theta_s\frac{ h''(\theta_s)}{h'(\theta_s)}-\lambda\,\bpI_s(1-\bpI_s)
     -\gamma \bpI_s\Big)\d s,\label{volz3}\\
     \bpS_t= & \bpS_0+\int_0^t \lambda \bpI_s \bpS_s
     \Big(1-\theta_s\frac{h''(\theta_s)}{h'(\theta_s)}\Big)\d s.
     \label{volz4}
   \end{align}
Here, the graph structure appears through the generating
 function $g$. In (\ref{volz2}), we
 see that the classical contamination terms $\lambda \bar{S}_t \bar{I}_t$
 (mass action) or $\lambda\bar{S}_t\bar{I}_t/(\bar{S}_t+\bar{I}_t)$
 (frequency dependence) of mixing SIR models (e.g.\ Part I of this volume or \cite{andersonbrittonIII,
   arazozaclemencontranIII}) are replaced by $\lambda\bpI_t \theta_t
 h'(\theta_t)=\lambda \bNIS_t$. The fact that new infectious individuals are chosen in the size-biased distribution
 is hidden in the term $h''(\theta_t)/h'(\theta_t)$.\\


 \begin{proposition}\label{prop_volz}
The system \eqref{limitereseauinfiniS}--\eqref{limitereseauinfiniSR} implies Volz' equations \eqref{volz1}--\eqref{volz4}.
 \end{proposition}

Before proving Proposition \ref{prop_volz}, we begin with a corollary of Theorem \ref{propconvergencemunS}.

\begin{corollary}\label{corol:nbrearete}
For all $t\in \R_+$ 
  \begin{align}
    \bNS_t= & \theta_t h'(\theta_t)\nonumber\\
    \bNSI_t= & \bNSI_0+\int_0^t \lambda \bpI_s \theta_s h'(\theta_s)\Big((\bpS_s-\bpI_s)\theta_s\frac{h''(\theta_s)}{h'(\theta_s)}-1\Big)
    -\gamma \bNSI_s\; \d s\nonumber\\
    \bNRS_t=&\int_0^t \Big(\gamma
    \bNSI_s-\lambda\bpR_s\bpI_s\theta_s^2
    h''(\theta_s)\Big)\d s. \label{eqcorol}
  \end{align}
\end{corollary}

\begin{proof}
In the proof of Proposition \ref{pro:horizon}, we will show below that when $N\rightarrow +\infty$, $(\NISn_t)_{N\in \N }$ converges uniformly, as $N\rightarrow +\infty$ and on compact intervals $[0,T]$,
and in probability to the deterministic and continuous solution $\bNIS$ such that for all $t$,
  $\bNIS_t=\langle \bmuIS_t,\chi\rangle$.
(\ref{limitereseauinfiniS}) with $f=\chi$ reads
  \begin{align}
    \bNS_t=\sum_{k\in \Z_+}\bmuS_0(k) k \theta_t^k=\theta_t
    \sum_{k=1}^{+\infty}\bmuS_0(k) k\theta_t^{k-1}=\theta_t
    h'(\theta_t),\label{etape9}
  \end{align}i.e.\  the first assertion of (\ref{eqcorol}).\\

  \par Choosing $f=\chi$ in (\ref{limitereseauinfiniSI}), we obtain
  \begin{multline*}
    \bNIS_t=  \bNIS_0-\int_0^t \gamma \bNIS_s\d s
    +  \int_0^t \sum_{k\in\Z_+} \Lambda_s(k) \sum_{j+\ell\leq k-1}\big(k-2j-2-\ell\big)\\
 \times    \Big[\frac{(k-1)!}{j!(k-1-j-\ell)!\ell!}(\bpI_s)^{j}(\bpR_s)^{\ell}(\bpS_s)^{k-1-j-\ell}
    \Big]  \bmuS_s(k)\d s.\nonumber
  \end{multline*}
  Notice that the term in the square brackets is the probability of obtaining  $(j,\ell,k-1-j-\ell)$ from a draw in
  the multinomial distribution of parameters
  $(k-1,(\bpI_s,\bpR_s,\bpS_s))$. Hence,
  \begin{align*}
    & \sum_{j+\ell\leq k-1}j\times
    \Big(\frac{(k-1)!}{j!(k-1-j-\ell)!\ell!}(\bpI_s)^{j}(\bpR_s)^{\ell}(\bpS_s)^{k-1-j-\ell}
    \Big) =(k-1)\bpI_s
  \end{align*}as we recognize the
  mean number of edges to $\cI_s$ of an individual of degree
  $k$. Other terms are treated similarly. Hence, with the definition of $\Lambda_s(k)$,
  (\ref{eq:deflambda}),
  \begin{align*}
    \bNIS_t=  \bNIS_0+ & \int_0^t
    \lambda\,\bpI_s\Big(\langle \bmuS_s,\chi^2-2\chi\rangle
    -(2\bpI_s+\bpR_s) \langle
    \bmuS_s,\chi(\chi-1)\rangle\Big) \d s
    -  \int_0^t \gamma
    \bNIS_s\d s.
  \end{align*}
  But since
  \begin{align*}
    &\langle \bmuS_t,\chi(\chi-1)\rangle=  \sum_{k\in \Z_+}\bmuS_0(k) k(k-1)\theta_t^k = \theta_t^2 h''(\theta_t)\\
     &\langle \bmuS_t,\chi^2-2\chi\rangle= \langle
     \bmuS_t,\chi(\chi-1)\rangle - \langle \bmuS_t,\chi\rangle =
     \theta_t^2 h''(\theta_t)-\theta_t h'(\theta_t),
  \end{align*}we obtain by noticing that
  $1-2\bpI_s-\bpR_s=\bpS_s-\bpI_s$,
  \begin{align}
    \bNIS_t= & \bNIS_0+\int_0^t
    \lambda\,\bpI_s\Big( (\bpS_s-\bpI_s)
    \theta_s^2 h''(\theta_s)-\theta_s h'(\theta_s) \Big) \d s-\int_0^t
    \gamma \bNIS_s\d s,
  \end{align}which is the second assertion of (\ref{eqcorol}). The
  third equation is obtained similarly.
\end{proof}

We are now ready to prove Volz' equations:
\begin{proof}[Proof of Proposition \ref{prop_volz}]
  We begin with the proof of (\ref{volz1}) and (\ref{volz2}). Fix again $t\ge 0$. For the
  size of the susceptible population, taking $f=\ind$ in
  (\ref{limitereseauinfiniS}) gives (\ref{volz1}). For the size
  of the infective population, setting $f=\ind$ in (\ref{limitereseauinfiniSI}) entails
  \begin{align*}
    \bar{I}_t= & \bar{I}_0+\int_0^t \Big(\sum_{k\in \Z_+} \lambda k\bpI_s \bmuS_s(k)-\gamma \bar{I}_s\Big) \d s\nonumber\\
    = & \bar{I}_0+\int_0^t \Big( \lambda\bpI_s \sum_{k\in \Z_+}
    \bmuS_0(k) k \theta_s^k-\gamma \bar{I}_s\Big) \d s\\
    = &
    \bar{I}_0+\int_0^t \Big( \lambda\bpI_s \theta_s h'(\theta_s)-\gamma
    \bar{I}_s\Big) \d s
  \end{align*}by using (\ref{limitereseauinfiniS}) with $f=\chi$ for the second equality.\\

  Let us now consider the probability that an edge with a susceptible
  ego has an infectious alter. Both equations
  (\ref{volz1}) and (\ref{volz2}) depend on $\bpI_t=\bNIS_t/\bNS_t$. It is thus
  important to obtain an equation for this quantity. In Volz \cite{volzIII},
  this equation also leads to introduce the quantity
 $\bpS_t$. \\
 From Corollary \ref{corol:nbrearete}, we see that $\bNS$ and $\bNIS$ are differentiable and:
 \begin{align*}
 \frac{\d\bpI_t}{\d t}= & \frac{\d}{\d t}\Big(\frac{\bNIS_t}{\bNS_t}\Big)=\frac{1}{\bNS_t}\frac{\d}{\d t}(\bNIS_t)
 -\frac{\bNIS_t}{(\bNS_t)^2}\frac{\d}{\d t}(\bNS_t)\\
 = & \Big(\lambda\bpI_t(\bpS_t-\bpI_t)\theta_t \frac{h''(\theta_t)}{h'(\theta_t)}-\lambda\bpI_t-\gamma \bpI_t\Big)\\
  & \hspace{1cm}-\Big(\frac{\bpI_t}{\theta_t h'(\theta_t)}\big(-\lambda\bpI_t \theta_t h'(\theta_t)
  +\theta_t h''(\theta_t)(-\lambda\bpI_t \theta_t)\big)\Big)\\
 = & \lambda\bpI_t\bpS_t \theta_t \frac{h''(\theta_t)}{h'(\theta_t)} -\lambda\bpI_t(1-\bpI_t)-\gamma \bpI_t,
 \end{align*}by using \eqref{eqcorol} for the derivatives of $\bNS$ and $\bNIS$ in the second line. This achieves the proof of (\ref{volz3}).\\

  \par For (\ref{volz4}), we notice
  that $
    \bpS_t=1-\bpI_t-\bpR_t$ and achieve the proof by showing that
  \begin{align}
    \bpR_t= \int_0^t \Big(\gamma \bpI_s-\lambda\bpI_s
    \bpR_s \Big)\d s
  \end{align}by using arguments similar as for $\bpI_t$.
\end{proof}


\subsection{Degree distribution of the ``initial condition''}\label{sec:deg-init-cond}

\begin{center}
\includegraphics[angle=0,width=0.4\linewidth]{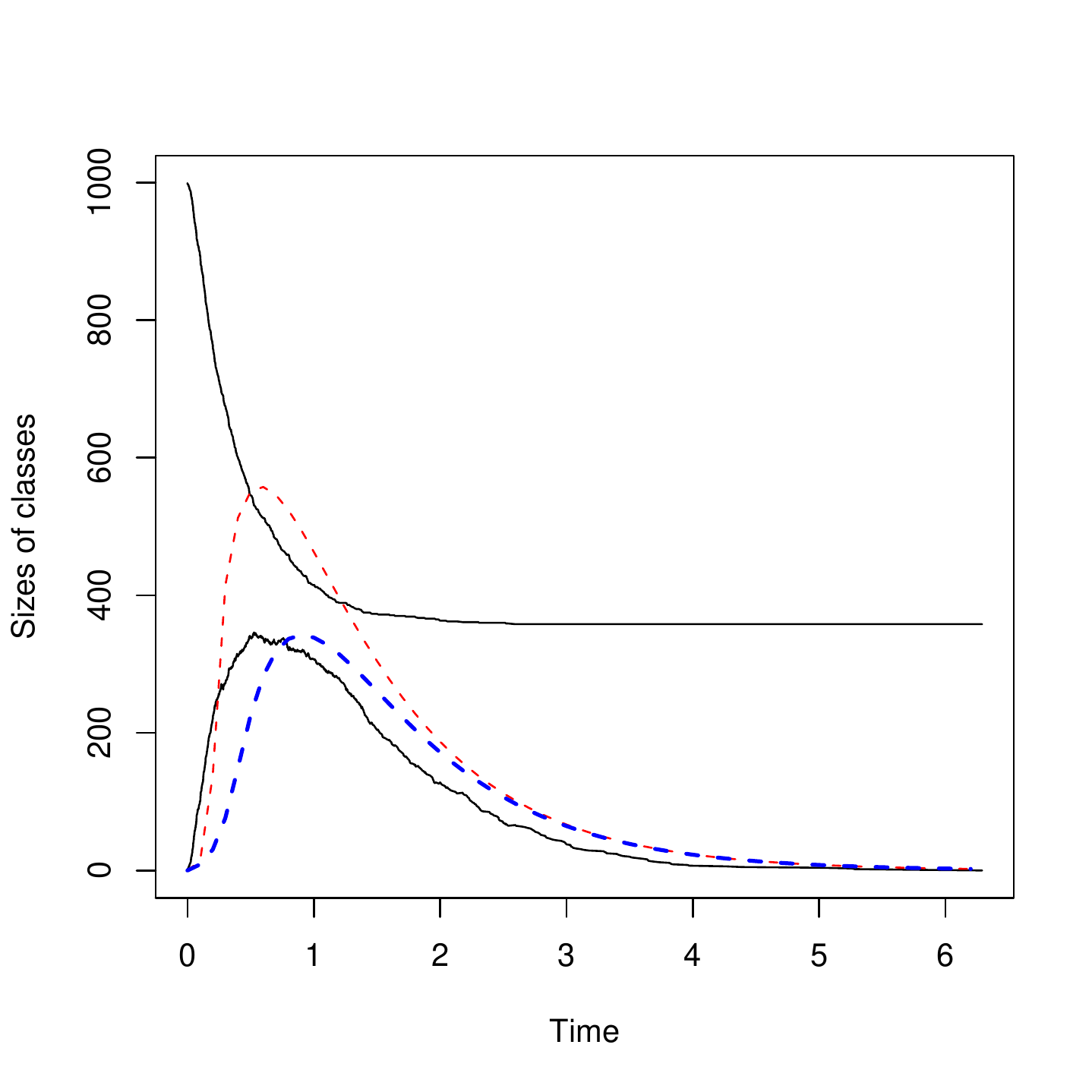}
\end{center}

The assumption of moments of order 5 in the Theorem \ref{propconvergencemunS} may seem restrictive.
Janson et al.\ \cite{jansonluczakwindridgeIII} showed that this assumption was not necessary if Volz' equations are established
by considering the process $(\Sn_t,\In_t,\Rn_t,\NSnn_t,N^{\textsc{n},\cI}_t,N^{\textsc{n},\cRr}_t)_{t\in \R_+}$
where $N^{\textsc{N},\cS}_t=\langle \muSn_t,k\rangle $, $N^{\textsc{N},\cI}_t$ and $N^{\textsc{N},\cRr}_t$ are
respectively the numbers of half-edges of the susceptible, infectious and removed individuals that are not attached to the
cluster. This process contains less information than the process $(\muSn_t,\muSIn_t,\muSRn_t)_{t\in\R_+}$, and an assumption on the
existence of moments of order 2 uniformly bounded in $N$ is sufficient. Janson and coauthors emphasize that if we allow the CM graph to have self-loops
and multiple edges, then only the uniform integrability of the degree distribution of an individual chosen at random is needed,
which seems to be the minimal assumption...\\

However, when considering the beginning of the epidemics, it appears that the assumption corresponding to Equation \eqref{eq:hypmoments} is not so restrictive.
Indeed, we emphasize that it should be distinguished between the degree distribution of the graph $\textbf{p}$, associated with the generating function $g$, and the degree distribution of
the $ \cS$ individuals when the proportion of infectious individuals has reached a non-negligible value, and which we associate with the generating function $h$. If we consider the degree distribution of the susceptible individuals, we see that the individuals with highest degrees
will be infected first, since individuals are chosen in the size-biased distribution \eqref{def:size-biased-distr} when
pairing the half-edges at random. After the $[\varepsilon N]$ first infections, with $\varepsilon>0$, when the Theorem \ref{propconvergencemunS}
starts to apply, all the susceptible individuals of highest degree have disappeared from $\muSn$. Then, $\muSn$ will even admit exponential moments.\\

For a population of size $N$, whose individuals have degrees $D_1,\dots D_N$, let us define, for all $k\in \Z_+$, the number
of vertices with degree $k$ among them by
\[N^{\textsc{N}}_k=\Card\{u\in \{1,\dots,N\}, \ D_u=k\}.\]
To each of the $D_u$ half-edges of individual $u$, we associate an independent uniform random variable on $[0,1]$. The vertex $u$ is infected
before the vertex $v$ if the minimal value $Z_u$ of the random variables attached to its half-edges is smaller than the minimal value $Z_v$ of the random variables attached to the half-edges of $v$.
 This construction has been used by Riordan \cite{riordanIII} and is related to size-biased orderings.

\begin{proposition}
(i) The degree distribution $(\widehat{p}^{\varepsilon,\textsc{N}}_k)_{k\geq 1}$ of the remaining susceptible
individuals after the $[\varepsilon\, N]$ first infections is:
\begin{equation}
\widehat{p}^{\varepsilon,\textsc{N}}_k=\frac{1}{N-[\varepsilon \, N]} \sum_{u=1}^N \ind_{D_u=k}
\ind_{Z_u> Z_{([\varepsilon \, N])}} \label{def:pchapeau}
\end{equation}
where $(Z_{(1)},\dots,Z_{(N)})$ are the order statistics of $(Z_1,\dots,Z_N)$, and where
\[\P(Z_u\leq z\, |\, D_u)=1-(1-z)^{D_u}.\]
(ii) For $z\in (0,1)$, let $M(z)$ be the survival function of the distribution of the $Z_i$ and let $M_N(z)$ be
the empirical survival function of $(Z_1,\dots,Z_N)$:
$$M_N(z)=\frac{1}{N}\sum_{u=1}^N \ind_{Z_u> z},\quad\mbox{ and }\quad M(z)=\sum_{k\geq 0} p_k(1-z)^k =g(1-z),$$where
$g(z)=\sum_{k\geq 0}p_k z^k$ is the generating function of the degree distribution $\textbf{p}$ of the CM graph.
Let $\varepsilon$ defined by $z^\varepsilon=\inf\{z\in (0,1),\ M(z)\geq \varepsilon\}$ be the quantile of order $\varepsilon$ of the $Z_u$.
Then, provided $M$ is continuous and strictly increasing at $z^\varepsilon$,
$$\lim_{N\rightarrow +\infty}Z_{[\varepsilon N]}=z^\varepsilon\qquad \mbox{almost surely}.$$
(iii) For such an $\varepsilon$, the degree distribution of the remaining susceptible individuals after the $[\varepsilon N]$
first infections converges weakly to:
\begin{equation}
\lim_{N\rightarrow +\infty}\sum_{k\geq 0}\widehat{p}^{\varepsilon,\textsc{N}}_k \delta_k=\frac{1}{1-\varepsilon}\sum_{k\geq 0} p_k(1-z^\varepsilon)^k\ \delta_k,\label{dist_limite_deg}
\end{equation}where $z^\varepsilon$ is solution of $1-\varepsilon=g(1-z^\varepsilon)$.
Moreover, we have convergence of the moments of order 5:
\begin{equation}
\lim_{N\rightarrow +\infty}\sum_{k\geq 0} k^5 \widehat{p}^{\varepsilon,\textsc{n}}_k
=\frac{1}{1-\varepsilon}\sum_{k\geq 0} k^5 p_k(1-z^\varepsilon)^k<+\infty.
\end{equation}
\end{proposition}

In particular, the limiting distribution \eqref{dist_limite_deg} admits moments of all orders.

\begin{proof}Let us prove (ii). Let $z\in [0,1]$. The proportion of vertices of degree $k$ whose minimal value of the $Z_u$ is smaller
than $z$ is $M^N_k(z)=\frac{1}{N}\sum_{u=1}^N \ind_{D_u=k}\ind_{Z_u>z}$. By the law of large numbers,
$\lim_{N\rightarrow +\infty} M^N_k(z)= p_k (1-z)^k$ a.s., which implies that
\begin{equation*}
\lim_{N\rightarrow +\infty}\frac{1}{N}\sum_{u=1}^N \ind_{Z_u> z}\delta_{D_u}=\sum_{k\geq 0} p_k(1-z)^k\delta_k
\end{equation*}for the weak convergence.

Assume that $\varepsilon>0$ is such that $M$ is continuous and strictly increasing at $z^\varepsilon$.
Then, $M(z^\varepsilon)=\varepsilon$. Let $\delta>0$ and
\[\eta=\min(|M(z^\varepsilon-\delta)-M(z^\varepsilon)|,|M(z^\varepsilon+\delta)-M(z^\varepsilon)|).\]
By the Kolmogorov--Smirnov theorem: $\lim_{N\rightarrow +\infty} \|M_N-M\|_\infty=0$. Then, there exists $\P(d\omega)$-a.s.
an integer $N_0(\omega)$ sufficiently large such that for all $N\geq N_0$,  $\|M_N-M\|_\infty< \eta/2$. Since $M$
is non-decreasing and since $Z_{[\varepsilon N]}$ is such that  $M_N(Z_{[\varepsilon N]})= \frac{[\varepsilon N ]}{N}$,
then,
\begin{align*}
\big|M(Z_{([\varepsilon N])})-\varepsilon \big|\leq  & \big|M(Z_{([\varepsilon N])})-M_N(Z_{([\varepsilon N])})\big|+
\big|M_N(Z_{([\varepsilon N])})-\varepsilon\big|\\
\leq & \frac{\eta}{2} + \big|\frac{[\varepsilon N]}{N}-\varepsilon\big|.
\end{align*}
Thus, for $N\geq \max(N_0,2/\eta)$, $|M(Z_{[\varepsilon N]})-\varepsilon|<\eta$ and hence $Z_{[\varepsilon N]}\in
(z^\varepsilon-\delta,z^\varepsilon+\delta)$ a.s.
This implies that $(Z_{[\varepsilon N]})_{N\geq 1}$ converges a.s. to $z^\varepsilon$.\\

If $(\widehat{p}^{\varepsilon,N}_k)_{k\in \Z_+}$ is the degree distribution after the $[\varepsilon N]$ first infections, then
\begin{equation}\label{eq:2}
\lim_{N\rightarrow +\infty} \sum_{k\geq 0} \widehat{p}^{\varepsilon,N}_k \delta_k
=\frac{1}{1-\varepsilon}\sum_{k\geq 0} p_k(1-z^\varepsilon)^k \delta_k.
\end{equation}
The convergence, for every $k\in\Z_+$, of $\widehat{p}^{\varepsilon,N}_k$ to $p_k(1-z^\varepsilon)^k/(1-\varepsilon)$ implies the convergence of \eqref{eq:2} for the vague topology. Because \eqref{eq:2} deals with probability measures, the criterion of \cite[Proposition 2]{meleardroellyIII} implies that the convergence also holds for the weak topology.


Since
\begin{multline}
\lim_{N\rightarrow +\infty}\E\Big(\frac{1}{N}\sum_{i=1}^N \big|\ind_{Z_i> Z_{[\varepsilon N]}}
-\ind_{Z_i>z^\varepsilon}\big| \ind_{d_i=k}\Big)
=  \lim_{N\rightarrow +\infty}\P\Big(Z_1\in [Z_{[\varepsilon N]}\wedge z^\varepsilon,Z_{[\varepsilon N]}\vee z^\varepsilon] ,\ d_1=k\Big)=0,
\end{multline}
and since $N/(N-[\varepsilon \, N])$ converges to $1/(1-\varepsilon)$, we obtain \eqref{dist_limite_deg}.
 \par For the convergence of the moments of order 5, we notice that for large $K\in \N$,
 \smallskip

\resizebox{0.95\linewidth}{!}{
  \begin{minipage}{\linewidth}
\begin{align*}
\lefteqn{\E\Big(\Big|\sum_{k\geq 0} k^5 M^N_k(Z_{[\varepsilon N]}) -\sum_{k\geq 0} k^5 p_k(1-z^\varepsilon)^k
\Big|\Big)}\\
\leq & \E\Big(\Big|\sum_{k\leq K} k^5 \big(M^N_k(Z_{[\varepsilon N]}) -p_k(1-z^\varepsilon)^k\big)
\Big|\Big) + \E\Big(\sum_{k> K} k^5 M^N_k(Z_{[\varepsilon N]})\Big)+\sum_{k> K} k^5 p_k(1-z^\varepsilon)^k.
\end{align*}
\end{minipage}
}
\smallskip

\noindent The first term converges to 0 with the preceding arguments.
The third term is controlled for $K$ sufficiently large. For the second term, we use that for all $z\in (0,1)$,
\[\E\big(\sum_{k> K} k^5 M^N_k(z)\big)=\sum_{k> K} k^5 p_k(1-z)^k\]and that $Z_{[\varepsilon N]}$ converges a.s. to $z^\varepsilon$.
\end{proof}




\subsection{Proof of the limit theorem}\label{sec:proof}

We now prove Theorem \ref{propconvergencemunS}.

In the proof, we will see that the epidemic remains large and described by a deterministic equation provided the number of edges from $\cI$ to $\cS$ remains of the order of $N$.
Let us thus define, for all $\varepsilon >0$, $\varepsilon'>0$ and $n\in \N$,
\begin{equation}
\label{eq:deftepsilon}
t_{\varepsilon'}:= \inf\{t \ge 0, \cro{\bmuIS_t,\chi} < \varepsilon' \}
\end{equation}
and:
\begin{equation}
\label{eq:deftauepsilonn}
  \tau^N_\varepsilon=\inf\{t\geq 0,\, \cro{\muISn_t,\chi} < \varepsilon\}.
\end{equation}

 In the sequel, we choose
$0<\varepsilon<\varepsilon'<\langle \bmuIS_0,\chi\rangle$.\\


  \noindent \textbf{Step 1} Let us prove that $(\muSn,\muISn,
  \muRSn)_{N\in \N}$ is tight. Let $t\in \R_+$ and $N\in \N$. By
  hypothesis, we have that
  \begin{multline}\langle \muSn_t,\ind + \chi^5 \rangle+\langle
    \muISn_t,\ind + \chi^5 \rangle + \langle \muRSn_t,\ind + \chi^5
    \rangle  \leq \langle \muSn_0,\ind + \chi^5 \rangle+\langle
    \muISn_0,\ind + \chi^5\rangle\le 2A.
    \label{domination0}
  \end{multline}Thus the sequences of marginals $(\muSn_t)_{N\in \N}$, $(\muISn_t)_{N\in \N}$ and $(\muRSn_t)_{N\in \N}$ are tight for
  each $t\in \R_+$. Now by the criterion of Roelly \cite{roellyIII}, it remains to prove that for each bounded function $f$ on $\Z_+$,
  the sequence $(\langle \muSn_.,f\rangle, \langle \muISn_.,f\rangle, \langle \muRSn_.,f\rangle)_{N\in \N} $ is tight in
  $\D(\R_+,\R^3)$. Since we have the semi-martingale decompositions of these processes, it is sufficient, by using the Rebolledo criterion, to prove that the finite variation part and the bracket of the martingale satisfy the Aldous criterion
  (see e.g.\ \cite{joffemetivierIII}). We only prove that $\langle \muISn_.,f\rangle$ is tight. The
  computations are similar for the other components.
  \par The Rebolledo--Aldous criterion is satisfied if for all
  $\alpha>0$ and $\eta>0$ there exists $N_0\in \Z_+ $ and $\delta>0$
  such that for all $N>N_0$ and for all stopping times $S_N$ and $T_N$
  such that $S_N<T_N<S_N+\delta$,
  \begin{align}
&    \P\big(|A^{\textsc{N},\cI\cS,f}_{T_N}-A^{\textsc{N},\cI\cS,f}_{S_N}|>\eta \big)\leq
    \alpha,\quad \mbox{ and }\label{criterealdous}\\
& \P\big(|\langle
    M^{\textsc{N},\cI\cS,f}\rangle_{T_N}-\langle
    M^{\textsc{N},\cI\cS,f}\rangle_{S_N}|>\eta \big)\leq
    \alpha.\nonumber
  \end{align}

  \par For the finite variation part,
  \begin{multline*}
    \E\Big[|A^{\textsc{n},\cI\cS,f}_{T_N}-A^{\textsc{n},\cI\cS,f}_{S_N}|\Big] \leq
    \E\left[\int_{S_N}^{T_N} \gamma \|f\|_\infty \langle \muSIn_s,1\rangle \d s\right]\\
    +\E\left[\int_{S_N}^{T_N} \sum_{k\in \Z_+} \Lambda^N_s(k) \muSn_s(k)
      \sum_{j+\ell\leq k-1} p_s^N(j,\ell|k-1) (2j+3)\|f\|_\infty\,\d
      s\right].
  \end{multline*}
  The term $\sum_{j+\ell\leq k-1}j p^N_s(j,\ell|k-1)$ is the mean
  number of links to $\I^N_{s_-}$ that the newly infected individual
  has, given that this individual is of degree $k$. It is bounded by
  $k$. Then, with (\ref{eq:deflambdan}),
\begin{align*}
\E\Big[|A^{\textsc{n},\cI\cS,f}_{T_N}-A^{\textsc{n},\cI\cS,f}_{S_N}|\Big] \leq &
\delta \E\Big[\beta \|f\|_\infty (\Sn_0+\In_0)
 + \lambda \|f\|_\infty \langle
\muSn_0,2\chi^2+3\chi\rangle\Big],\end{align*}by using that the number of infectives
is bounded by the size of the population and that
$\muSn_s(k)\le \muSn_0(k)$ for all $k$ and $s\geq 0$.
From (\ref{eq:hypmoments}), the r.h.s.\ is finite. Using Markov's
inequality,
$$
\P\big(|A^{\textsc{n},\cI\cS,f}_{T_N}-A^{\textsc{n},\cI\cS,f}_{S_N}|>\eta \big)\leq
\frac{(5 \lambda+2\gamma)A\delta
  \|f\|_\infty}{\eta},$$ which is smaller than $\alpha$ for $\delta$
small enough.
\par We use the same arguments for the bracket of the martingale:
\begin{multline}
  \label{etape3}
    \begin{aligned}
  \E\big[|\langle M^{\textsc{n},\cI\cS,f}\rangle_{T_N}-\langle M^{\textsc{n},\cI\cS,f}\rangle_{S_N}|\big]  &\le \E\Big[\frac{\delta \gamma \|f\|^2_\infty (\Sn_0+\In_0)}{N}
+ \frac{\delta \lambda\|f\|_\infty^2 \langle \muSn_0,\chi(2\chi+3)^2\rangle}{N}\Big]\\
    & \le \frac{(25 \lambda+2\gamma)A\delta \|f\|_\infty^2}{N},
  \end{aligned}
\end{multline}
using Assumption \ref{hypconvcondinit} and (\ref{eq:hypmoments}).  The r.h.s.\ can be made
smaller than $\eta \alpha$ for a small enough $\delta$, so the
second inequality of (\ref{criterealdous}) follows again from Markov's
inequality. By \cite{roellyIII}, this provides the
tightness in $\D(\R_+,\mathcal{M}_{0,A}^3)$, with $\mathcal{M}_{0,A}$ defined in \eqref{eq:defM}.\\

By Prohorov's theorem (e.g.\ \cite{ethierkurtzIII}, p.~104) and Step 1, we obtain that the
distributions of $(\muSn,\muISn, \muRSn)$, for $N\in \N$, form a
relatively compact family of bounded measures on $\D(\R_+,\M_{0,A}^3)$,
and so do the laws of the stopped processes $(\muSn_{.\wedge
  \tau^N_\varepsilon},\muISn_{.\wedge \tau^N_\varepsilon},
\muRSn_{.\wedge \tau^N_\varepsilon})_{N\in \N}$ (recall (\ref{eq:deftauepsilonn})). Because of the moment assumptions for the degree distributions, the limiting process is continuous. Let
$\bar{\mu}:=(\bmuS,\bmuIS,\bmuRS)$ be a limiting point in
$\Co(\R_+,\mathcal{M}_{0,A}^3)$ of the sequence of stopped processes and let us consider a subsequence again
denoted by
$\mu^{N}:=(\muSn,\muISn, \muRSn)_{N\in \N}$, with an abuse of notation, and that converges to
$\bar\mu$. Because the limiting values are continuous, the
convergence of $(\mu^{N})_{N\in \N}$ to $\bar\mu$ holds for the
uniform convergence on every compact subset of $\R_+$ 
(e.g.\ \cite{billingsleyIII} p.~112).

Now, let us define for all $t\in \R_+$ and for all bounded functions
$f$ on $\Z_+$, the mappings $\Psi_t^{\cS,f}$, $\Psi_t^{\cI\cS,f}$ and
$\Psi_t^{\cRr\cS,f}$ from $\D\big(\R_+,\mathcal{M}_{0,A}^3\big)$ into
$\D\big(\R_+,\R\big)$ such that
(\ref{limitereseauinfiniS})--(\ref{limitereseauinfiniSR}) read
\begin{multline}
  \label{eq:defPsi} \left(\cro{\bmuS_t,f},\cro{\bmuIS_t,f},\cro{\bmuRS_t,f}\right) 
  =\left(\Psi^{\cS,f}_t\left(\bmuS,\bmuIS,\bmuRS\right),\Psi^{\cI\cS,f}_t\left(\bmuS,\bmuIS,\bmuRS\right),
  \Psi^{\cRr\cS,f}_t\left(\bmuS,\bmuIS,\bmuRS\right)\right).
\end{multline}
Our purpose is to prove that the limiting values are the unique
solution of
(\ref{limitereseauinfiniS})--(\ref{limitereseauinfiniSR}).\\
Before proceeding to the proof, a remark is in order. A  natural
way of reasoning would be to prove that $\Psi^{\cS,f}, \,
\Psi^{\cI\cS,f}$ and $\Psi^{\cRr\cS,f}$ are Lipschitz continuous in some
spaces of measures. To avoid doing so by considering
the set of measures with moments of any order, which is a
set too small for applications, we circumvent this difficulty by first proving
that the mass and the first two moments of any solutions of the system
are the same. Then, we prove that the generating
functions of these measures satisfy a partial differential equation
known to have a unique solution.\\

\noindent \textbf{Step 2} We now prove that the
differential system
(\ref{limitereseauinfiniS})--(\ref{limitereseauinfiniSR}) has at most
one solution in ${\mathcal C}(\R_+,\, \M_{0,A}\times \M_{0+,A}\times \M_{0,A})$. Let $T>0$.
Let
$\bar{\mu}^i=(\bar{\mu}^{\cS,i},\bar{\mu}^{\cI\cS,i},\bar{\mu}^{\cRr\cS,i})$,
$i\in \{1,2\}$ be two solutions of
(\ref{limitereseauinfiniS})--(\ref{limitereseauinfiniSR}), started with
the same initial conditions in $\M_{0,A}\times \M_{\varepsilon,A}\times \M_{0,A}$ for some small $\varepsilon>0$. Set
\begin{multline*}
  \Upsilon_t= \sum_{j=0}^3 |\langle \bar{\mu}^{\cS,1}_t,\chi^j\rangle-\langle
  \bar{\mu}^{\cS,2}_t,\chi^j\rangle|
   + \sum_{j=0}^2\Big(|\langle
  \bar{\mu}^{\cIS,1}_t,\chi^j\rangle-\langle
  \bar{\mu}^{\cIS,2}_t,\chi^j\rangle|+|\langle
  \bar{\mu}^{\cRSr,1}_t,\chi^j\rangle-\langle
  \bar{\mu}^{\cRSr,2}_t,\chi^j\rangle|
\Big).
\end{multline*}
Let us first remark that for all $0\leq t <T$, $\bNS_t\geq \bNIS_t>\varepsilon$ and then
\begin{multline}
  |\bar{p}^{\cI,1}_t-\bar{p}^{\cI,2}_t|=
  \Big|\frac{\bar{N}^{\cIS,1}_t}{\bar{N}^{\cS,1}_t}-\frac{\bar{N}^{\cIS,2}_t}{\bar{N}^{\cS,2}_t}\Big|
  \leq \frac{A}{\varepsilon^2}\Big|\bar{N}^{\cS,1}_t-\bar{N}^{\cS,2}_t\Big|+
\frac{1}{\varepsilon}\Big|\bar{N}^{\cIS,1}_t-\bar{N}^{\cIS,2}_t\Big|\\
=  \frac{A}{\varepsilon^2}\Big|\cro{\bar{\mu}_t^{\cS,1},\,
  \chi}-\cro{\bar{\mu}_t^{\cS,2},\,
  \chi}\Big|+\frac{1}{\varepsilon}\Big|\cro{\bar{\mu}_t^{\cIS,1},\,
  \chi}-\cro{\bar{\mu}_t^{\cIS,2},\, \chi}\Big|\le \frac{A}{\varepsilon^2}\Upsilon_t.\label{etape81}
\end{multline}The same computations show a similar result for $
|\bar{p}^{\cS,1}_t-\bar{p}^{\cS,2}_t|$. 

Using that $\bar{\mu}^i$ are solutions to
(\ref{limitereseauinfiniS})--(\ref{limitereseauinfiniSI}) let us show that $\Upsilon$ satisfies a Gronwall inequality which implies
that it is equal to 0 for all $t\le T$.
For the degree distributions of the susceptible individuals, we have for $p\in \{0,1,2,3\}$ and $f=\chi^p$ in \eqref{limitereseauinfiniS}:

\begin{align*}
|\langle \bar{\mu}^{\cS,1}_t,\chi^p\rangle-\langle\bar{\mu}^{\cS,2}_t,\chi^p\rangle|= & \Big|\sum_{k\in \Z_+}
\bar{\mu}_0^{\cS}(k) k^p \big(e^{-\lambda\int_0^t \bar{p}^{\cI,1}_s ds}-e^{-\lambda\int_0^t \bar{p}^{\cI,2}_s ds}\big)\Big|
\nonumber\\
\leq & \lambda \sum_{k\in \Z_+} k^p \bar{\mu}_0^{\cS}(k) \int_0^t \big|\bar{p}^{\cI,1}_s-\bar{p}^{\cI,2}_s \big|ds
\leq  \lambda\frac{A^2}{\varepsilon^2}\int_0^t \Upsilon_s ds,
\end{align*}by using \eqref{etape81} and the fact that $\bmuS_0\in \mathcal{M}_{0,A}$. \\
For $\bmuSI$ and $\bmuSR$, we use \eqref{limitereseauinfiniSI} and \eqref{limitereseauinfiniSR} with the
functions $f=\chi^0=\ind $, $f=\chi$ and $f=\chi^2$. We proceed here with only one of the
computations, others can be done similarly. From \eqref{limitereseauinfiniSI}:
\begin{multline*}
  \cro{\bar{\mu}_t^{\cIS,1},\, \mathbf{1}}-\cro{\bar{\mu}_t^{\cIS,2},\,
    \mathbf{1}}= 
    \gamma\int_0^t\cro{
    \bar{\mu}_s^{\cIS,1}-\bar{\mu}_s^{\cIS,2},\, \mathbf{1} }\d s
+\lambda \int_0^t (\bar{p}_s^{\cI,1}\cro{\bar{\mu}^{\cS,1}_s,\,
    \chi}-\bar{p}_s^{\cI,2}\cro{\bar{\mu}^{\cS,2}_s,\,
    \chi})\d s.
\end{multline*}
Hence, with \eqref{etape81},
\begin{equation*}
  \left|  \cro{
    \bar{\mu}_t^{\cIS,1}-\bar{\mu}_t^{\cIS,2},\,  \mathbf{1}}\right|\le C(\lambda,\gamma,A,\varepsilon) \int_0^t \Upsilon_s\d s.
\end{equation*}
By analogous computations for the other quantities, we show that
\begin{equation*}
  \Upsilon_t\le C'(\lambda,\gamma,A,\varepsilon) \int_0^t \Upsilon_s\d s,
\end{equation*}
hence $\Upsilon\equiv 0$.
It follows that for all
$t<T$, and for all $j\in
\{0,1,2\}$,
\begin{align}
  \langle \bar{\mu}^{\cS,1}_t,\chi^j\rangle=\langle
  \bar{\mu}^{\cS,2}_t,\chi^j\rangle\quad \mbox{ and }\quad \langle
  \bar{\mu}^{\cIS,1}_t,\chi^j\rangle=\langle
  \bar{\mu}^{\cIS,2}_t,\chi^j\rangle,\label{etape10}
\end{align}and in particular, $\bar{N}^{\cS,1}_t=\bar{N}^{\cS,2}_t$ and
$\bar{N}^{\cIS,1}_t=\bar{N}^{\cIS,2}_t$. This implies that $
\bar{p}^{\cS,1}_t=\bar{p}^{\cS,2}_t$, $\bar{p}^{\cI,1}_t=\bar{p}^{\cI,2}_t$ and $\bar{p}^{\cRr,1}_t=\bar{p}^{\cRr,2}_t$.
From (\ref{limitereseauinfiniS}), we have that $\bar{\mu}^{\cS,1}=\bar{\mu}^{\cS,2}$.\\

Our purpose is now to prove that
$\bar{\mu}^{\cIS,1}=\bar{\mu}^{\cIS,2}$. Let us introduce the
following generating functions: for any
 $t\in
\R_+$, $i\in \{1,2\}$ and $\eta\in [0,1)$,
\begin{equation*}
  \mathcal{G}^i_t(\eta)=\sum_{k\geq 0}\eta^k
    \bar{\mu}^{\cIS,i}_t(k).
\end{equation*}
Since we already know that these
measures have the same total mass,  it remains to prove that $
\mathcal{G}^1\equiv   \mathcal{G}^2.$
Let us define
\begin{align}
  & H(t,\eta)=  \int_0^t \sum_{k\in \Z_+}\lambda k\bpI_s\sum_{\substack{j,\, \ell,\, m\in \Z_+\\
      j+\ell+m=k-1}}{k-1 \choose j,\ell,m}(\bpI_s)^j (\bpR_s)^\ell(\bpS_s)^m \eta^m \bmuS_s(k)\, ds,\nonumber\\
  & K_t=\sum_{k\in
    \Z_+}\lambda k\bpI_t(k-1)\bpR_t\frac{\bmuS_t(k)}{\bNIS_t}.\label{etape11}
\end{align}The latter quantities are respectively of class $\Co^1$ and
$\Co^0$ with respect to time $t$ and are well-defined and bounded on
$[0,T]$. 
Moreover, $H$ and $K$ do not depend
on the chosen solution because of \eqref{etape10}.  Applying
(\ref{limitereseauinfiniSI}) to $f(k)=\eta^k$ yields
\begin{align*}
  \mathcal{G}^i_t(\eta)= & \mathcal{G}^i_0(\eta)+H(t,\eta)+ \int_0^t \Big(K_s \sum_{k'\in \N}\big(\eta^{k'-1}-\eta^{k'}\big)k'
  \bar{\mu}^{\cIS,i}_s(k')-\gamma \mathcal{G}^i_s(\eta)\Big) ds\nonumber\\
  = & \mathcal{G}^i_0(\eta)+H(t,\eta)+ \int_0^t \Big(K_s
  (1-\eta)\partial_\eta \mathcal{G}^i_s(\eta) -\gamma
  \mathcal{G}^i_s(\eta)\Big) ds.
\end{align*}Then, the
functions $t\mapsto \widetilde{\mathcal{G}}^i_t(\eta)$ defined by
$\widetilde{\mathcal{G}}^i_t(\eta)=e^{\beta t}\mathcal{G}^i_t(\eta)$,
$i\in \{1,2\}$, are solutions of the following transport equation (of unknown function $g$):
\begin{align}
  \partial_t g(t,\eta)-(1-\eta)K_t \ \partial_\eta
  g(t,\eta)=\partial_t H(t,\eta)e^{\beta t}.
\end{align}
In view of the regularity of $H$ and $K$, it is known that this equation admits a unique solution (see e.g.\ \cite{evansIII}). Hence
$\mathcal{G}^1_t(\eta)=\mathcal{G}^2_t(\eta)$ for all $t\in \R_+$ and
$\eta\in [0,1)$.  The same method applies to $\bmuRS$. Thus there is at most
one solution to the differential system \eqref{limitereseauinfiniS}--\eqref{limitereseauinfiniSR}.\\

\noindent \textbf{Step 3} We now show that $\mu^{N}$ nearly satisfies
(\ref{limitereseauinfiniS})--(\ref{limitereseauinfiniSR}) as $N$ gets large. Recall (\ref{musif}) for a bounded function
$f$ on $\Z_+$. To identify the limiting values, we establish that for
all $N\in \N$ and all $t\ge0$,
\begin{align}
  \langle \muISn_{t\wedge \tau^N_\varepsilon},f\rangle=\Psi_{t\wedge
    \tau^N_\varepsilon}^{\cI\cS,f}(\mu^{N})+\Delta_{t\wedge
    \tau^N_\varepsilon}^{\textsc{n},f}+M^{\textsc{n},\cI\cS,f}_{t\wedge
    \tau^N_\varepsilon},\label{decompo}
\end{align}where $M^{N,\cI\cS,f}$ is defined in (\ref{musif}) and where $\Delta_{.\wedge \tau^N_\varepsilon}^{N,f}$
converges to 0 when $N\rightarrow +\infty$, in probability and uniformly in $t$ on compact time intervals. \\

Let us fix $t\in \R_+$. Computation similar to (\ref{etape3}) give:
\begin{align}
  \E\big((M^{\textsc{n},\cI\cS,f}_t)^2\big)=\E\big(\langle
  M^{\textsc{n},\cI\cS,f}\rangle_t\big) \leq \frac{(25\lambda+2\gamma)\, A t
    \|f\|_\infty^2}{N}.\label{etape777}
\end{align}Hence the sequence $(M^{\textsc{n},\cI\cS,f}_t)_{N\in \Z_+ }$ converges in $L^2$ and in probability to
zero.\\  

We now consider the finite variation part of (\ref{musif}), given
in (\ref{defA(n)SI}). The sum in \eqref{defA(n)SI} corresponds to the
links to $\cI$ that the new infected individual has. We separate this
sum into cases where the new infected individual only has simple edges to other
individuals of $\cI$, and cases where multiple edges
exist. The latter term is expected to vanish for large populations.
\begin{align}
  A^{\textsc{n},\cI\cS,f}_t
  = &  B^{\textsc{n},\cI\cS,f}_t + C^{\textsc{n},\cI\cS,f}_t ,\label{eq:shadow1}
\end{align}where
\begin{multline}\label{termeBn}
  B^{\textsc{n},\cI\cS,f}_t = -\int_0^t \gamma\langle \muSIn_s,f\rangle\d s \\
  +\int_0^t \sum_{k\in \Z_+}
  \Lambda^N_s(k)\muSn_s(k)\sum_{j+\ell+1\leq k}p^N_s(j,\ell|k-1)
 \Biggl\{f(k-(j+1+\ell))
 \\
  + \sum_{\substack{u\in \mathcal{L}(j+1,\muISnn_s);\\ \forall
      u\leq \In_{s_-},\,n_u\leq 1
    }}\rho(\underline{n}|j+1,\muISnn_s)\sum_{u\in \In_{s_-}}\left(f\left(D_{u}(\muSInn_{s_-})
      -
      n_u\right)-f\left(D_{u}(\muSInn_{s_-})\right)\right)\Biggl\}\,\d
  s
\end{multline}
and
\begin{multline}\label{termeCn}
  C^{\textsc{n},\cI\cS,f}_t = \int_0^t \sum_{k\in \Z_+} \Lambda^N_s(k)\muSn_s(k)\sum_{j+\ell+1\leq k}p^N_s(j,\ell|k-1)\\
  \times\sum_{\substack{\underline{n}\in \mathcal{L}(j+1,\muISnn_s); \\
      \exists u\leq \In_{s_-},\, n_u>1}}\rho(\underline{n}|j+1,\muISnn_s)
  \sum_{u\in \In_{s_-}}\left(f\left(D_{u}(\muSInn_{s_-}) -
      n_u\right)-f\left(D_{u}(\muSInn_{s_-})\right)\right)\,\d
  s.
\end{multline}
We first show that $C^{\textsc{n},\cS\cI,f}_t$ is a negligible term.  Let
$q^N_{j,\ell,s}$ denote the probability that the newly infected
individual at time $s$ has a double (or of higher order) edge to some
alter in $\In_{s_-}$, given $j$ and $\ell$.  The probability to have
a multiple edge to a given infectious $i$ is less than the number of
pairs of edges linking the newly infected to $i$, times the
probability that these two particular edges linking $i$ to a
susceptible alter at time $s_-$ actually lead to the newly infected. Hence,
\begin{align}
  q^N_{j,\ell,s}= & \sum_{\substack{\underline{n}\in \mathcal{L}(j+1,\muISnn_s);\\ \exists u\in \In_{s_-},\, n_u>1}}\rho(\underline{n}|j+1,\muISnn_s)
  \nonumber\\
  \leq & { j\choose 2} \sum_{u\in
    \In_{s_-}}\frac{D_u(\cS^{\textsc{n}}_{s_-})(D_u(\cS^{\textsc{n}}_{s_-})-1)}{\NISnn_{s_-}(\NISnn_{s_-}-1)}
  =   {j \choose 2}\frac{1}{N}\frac{\langle \muISn_{s_-},\chi(\chi-1)\rangle}{\frac{\NISn_{s_-}}{N}\big(\frac{\NISn_{s_-}}{N}-\frac{1}{N}\big)}\nonumber\\
  \leq & {j\choose 2}
  \frac{1}{N}\frac{A}{\varepsilon(\varepsilon-1/N)}\quad \mbox{ if
  }s<\tau^N_\varepsilon\mbox{ and }N>1/\varepsilon.\label{estimatep2}
\end{align}
Then, since for all $u\in \mathcal{L}(j+1,\muISnn_s)$,
\begin{equation}
  \Big|\sum_{u\in \In_{s_-}}\left(f\left(D_{u}(\muSInn_{s_-}) - n_u\right)-f\left(D_u(\muSInn_{s_-})\right)\right)\Big |
  \leq 2(j+1)\|f\|_\infty,\label{ancienneeq3.25}
\end{equation}
we have by (\ref{estimatep2}) and (\ref{ancienneeq3.25}), for $N>1/\varepsilon$,

\begin{align}
\label{etape5}&  |C^{\textsc{n},\cI\cS,f}_{t\wedge
    \tau^N_\varepsilon}| \\
  \leq & \int_0^{t\wedge\tau^N_\varepsilon} \sum_{k\in\Z_+} \lambda k \muSn_s(k)\sum_{j+\ell+1\leq k}
  p^N_s(j,\ell|k-1) 2(j+1)\|f\|_\infty \frac{j(j-1)A}{2N\varepsilon (\varepsilon-1/N)}\,\d s\nonumber\\
  \leq & \frac{A\, \lambda t\|f\|_\infty}{N\, \varepsilon
    (\varepsilon-1/N)}\langle \muSn_0,\chi^4\rangle,\nonumber
\end{align}which tends to zero in view of  (\ref{eq:hypmoments})
and thanks to the fact that $\muSn_s$ is dominated by $\muSn_0$ for
all $s\geq 0$ and $N\in \N $.

We now aim at proving that
$B^{\textsc{n},\cI\cS,f}_{.\wedge\tau^N_\varepsilon}$ is close to
$\Psi_{.\wedge \tau^N_\varepsilon}^{\cI\cS,f}(\mu^{N})$. First, notice
that

\begin{align}
  &\sum_{\substack{\underline{n}\in \mathcal{L}(j+1,\muISnn_s);\\ \forall u\in \In_{s_-},\, n_u \le 1}}
  \rho (u|j+1,\muSInn_s)\sum_{i\in \In_{s_-}}\left(f\left(D_{u}\big(\muSInn_{s_-}\big)-n_u\right)
  -f\left(D_{u}\big(\muSInn_{s_-}\big)\right)\right)\nonumber\\
    &= \sum_{u_0\not= \dots
        \not=u_j \in \In_{s_-}} \left(\frac{\prod_{k=0}^j
        D_{u_k}(\cS^{\textsc{n}}_{s})}{\NISn_{s_-}\dots
        (\NISn_{s_-}-(j+1))}\right)\nonumber\\
       & \qquad \qquad\qquad\times \sum_{m=0}^j\left(f\left(D_{u_m}(\cS^{\textsc{n}}_{s_-})
        -  1\right)-f\left(D_{u_m}(\cS^{\textsc{n}}_{s_-})\right)\right)\nonumber\\
    &= \sum_{m=0}^j\sum_{u_0\not=
        \dots \not=u_j \in \In_{s_-}} \left(\frac{\prod_{k=0}^j
        D_{u_k}(\cS^{\textsc{n}}_s)}{\NISn_{s_-}\dots
        (\NISn_{s_-}-(j+1))}\right)\nonumber\\
        & \qquad \qquad\qquad\times \left(f\left(D_{u_m}(\cS^{\textsc{n}}_{s_-})
        -  1\right)-f\left(D_{u_m}(\cS^{\textsc{n}}_{s_-})\right)\right)\label{eq:shadow2}\\
    &= \sum_{m=0}^j \left(\sum_{x\in \In_{s_-}}
      \frac{D_{x}(\cS^{\textsc{n}}_{s_-})}{\NISnn_{s_-}}\left(f\left(D_{x}(\cS^{\textsc{n}}_{s_-})
          -  1\right)-f\left(D_{x}(\cS^{\textsc{n}}_{s_-})\right)\right)\right)\nonumber\\
    & \qquad \qquad\qquad \times\Biggl(\sum_{u_0\not= \dots
        \not=u_{j-1}\in \In_{s_-}\setminus \{x\}}\frac{\prod_{k=0}^{j-1}
      D_{u_k}(\cS^{\textsc{n}}_{s})}{(\NISn_{s_-}-1)\dots (\NISn_{s_-}-(j+1))}\Biggl)\nonumber\\
    &= (j+1) \frac{\langle
      \muISn_{s_-},\chi\left(\tau_1f-f\right)\rangle}{\NISn_{s_-}}\left(1-q^N_{j-1,\ell,s}\right),\nonumber
  \end{align}
where we recall (see Notation \ref{notation:part3}) that $\tau_1 f(k)=f(k-1)$ for every
function $f$ on $\Z_+$ and $k\in \Z_+$.
In the third equality, we split the term $u_m$ from the other terms $(u_{m'})_{m'\not= m}$. The last sum in the r.h.s.\ of this
 equality is the probability of drawing $j$ different infectious individuals that are not $u_m$ and that are all different,
 hence $1-q^N_{j-1,\ell,s}$.\\

Define for $t>0$ and $N\in \Z_+ $,
\begin{align*}
  & p^{\textsc{n},\cI}_t=\frac{\cro{\muISnn_t,\chi}-1}{\cro{\muSn_t,\chi}-1},\\
  & p^{\textsc{n},\cRr}_t=\frac{\cro{\muRSnn_t,\chi}}{\cro{\muSn_t,\chi}-1},\\
  &
  p^{\textsc{n},\cS}_t=\frac{\cro{\muSnn_t,\chi}-\cro{\muISnn_t,\chi}-\cro{\muRSnn_t,\chi}}{\cro{\muSn_t,\chi}-1},
\end{align*}
the proportion of edges with infectious (resp.\ removed and
susceptible) alters and susceptible egos among all the edges with
susceptible egos but the contaminating edge. For all integers $j$ and
$\ell$ such that $j+\ell\leq k-1$ and $N\in \N $, denote by $$\tilde
p^N_t(j,\ell \mid
k-1)=\frac{(k-1)!}{j!(k-1-j-\ell)!\ell!}(p_t^{\textsc{n},\cI})^{j}(p_t^{\textsc{n},\cRr})^{\ell}(p_t^{\textsc{n},\cS})^{k-1-j-\ell},$$
the probability that the multinomial variable counting the number of
edges with infectious, removed and susceptible alters, among $k-1$ given edges, equals
$(j,\ell,k-1-j-\ell)$. We have that
\begin{align}
  |\Psi_{t\wedge
    \tau^N_\varepsilon}^{\cI\cS,f}(\mu^{N})-B^{\textsc{n},\cI\cS,f}_{t\wedge\tau^N_\varepsilon}|
  \leq & |D^{\textsc{n},\cI\cS,f}_{t\wedge
    \tau^N_\varepsilon}|+|E^{\textsc{n},\cI\cS,f}_{t\wedge
    \tau^N_\varepsilon}|,\label{etape778}
\end{align}where
\begin{align*}
  D^{\textsc{n},\cI\cS,f}_{t} = & \int_0^t \sum_{k\in\Z_+} \Lambda^N_s(k)\muSn_s(k)\sum_{j+\ell+1\leq k}\left(p^N_s(j,\ell|k-1)-
  \tilde p^N_s(j,\ell|k-1)\right)\nonumber\\
  & \hspace{0.5cm} \times \left(f(k-(j+\ell+1))+
    (j+1) \frac{\langle \muISn_{s_-},\chi\big(\tau_1f-f\big)\rangle}{\NISn_{s_-}}\right)\,\d s,\nonumber\\
  E^{\textsc{n},\cI\cS,f}_{t} = & \int_0^t \sum_{k\in\Z_+}
  \Lambda^N_s(k)\muSn_s(k)\\
  & \hspace{0.5cm}\times \sum_{j+\ell+1\leq k}p^N_s(j,\ell|k-1) (j+1)
  \frac{\langle
    \muISn_{s_-},\chi\big(\tau_1f-f\big)\rangle}{\NISn_{s_-}}q^N_{j-1,\ell,s}\,\d s.
\end{align*}

First,
\begin{align}
  |D^{\textsc{n},\cI\cS,f}_{t\wedge \tau^N_\varepsilon}|\leq & \int_0^{t\wedge
    \tau^N_\varepsilon} \sum_{k\in\Z_+} \lambda k \alpha^N_s(k) \|f\|_\infty
  \left(1+ \frac{2kA}{\varepsilon}\right)\,\muSn_s(k)\,\d
  s,\label{etape753}
\end{align}where for all $k\in \Z_+$ $$\alpha^N_t(k)=\sum_{j+\ell+1\leq
  k}\biggl |p^N_t(j,\ell|k-1)-\tilde p^N_t(j,\ell|k-1)\biggl |. $$ The
multinomial probability $\tilde p^N_s(j,\ell|k-1)$ approximates the
hypergeometric one, $p^N_s(j,\ell|k-1,s)$, as $N$ increases to
infinity, in view of the fact that the total population size,
$\cro{\muSn_0,\ind}+\cro{\muISnn_0,\ind}$, is of order $n$.
Hence, the r.h.s.\ of (\ref{etape753}) vanishes by dominated
convergence.

On the other hand, using (\ref{estimatep2}),
\begin{align}
  |E^{\textsc{n},\cI\cS,f}_{t\wedge \tau^N_\varepsilon}| \leq & \int_0^{t\wedge \tau^N_\varepsilon} \sum_{k\in\Z_+}  \lambda k^2 \muSn_s(k)
   \frac{2\|f\|_\infty A}{\varepsilon}\frac{k^2 A}{2N\varepsilon (\varepsilon -1/N)}\,\d s\nonumber\\
  \leq & \frac{A^3\,
    \lambda t\|f\|_\infty}{N\varepsilon^2(\varepsilon-1/N)},\label{etape779}
\end{align}in view of (\ref{eq:hypmoments}).  Gathering
(\ref{etape777}), (\ref{eq:shadow1}), (\ref{etape5}),
(\ref{etape778}), (\ref{etape753}) and (\ref{etape779}) concludes the
proof that
the rest of (\ref{decompo}) vanishes in probability uniformly over compact intervals.

As a consequence, the sequence $(\Psi_{.\wedge
    \tau^N_\varepsilon}^{\cI\cS,f}(\mu^{N}))_{N\in \N}$ is also tight in  $\D(\R_+,\mathcal{M}_{0,A}\times \mathcal{M}_{\varepsilon,A}\times \mathcal{M}_{0,A})$.\\

\noindent \textbf{Step 4} Recall that in this proof, $\bar{\mu}=(\bmuS,\bmuIS,\bmuRS)$ is the limit of the sequence
 $(\mu^{N}_{.\wedge \tau^N_\varepsilon})_{N\in \N}=(\muSn_{.\wedge \tau^N_\varepsilon},\muISn_{.\wedge \tau^N_\varepsilon},\muRSn_{.\wedge \tau^N_\varepsilon})_{N\in \N}$,
 and recall that these processes take values in the closed set $\M_{0,A}^3$. Our purpose is now to prove that $\bar{\mu}$
 satisfies (\ref{limitereseauinfiniS})--(\ref{limitereseauinfiniSR}). Using Skorokhod's representation theorem, there exists,
 on the same probability space as $\bar{\mu}$, a sequence, again denoted by $(\mu^{N}_{.\wedge \tau^N_\varepsilon})_{N\in \N }$
  with an abuse of notation, with the same marginal distributions as the original sequence, and that converges a.s. to $\bar{\mu}$.\\

The maps $\nu_.:=(\nu^1_.,\nu^2_.,\nu^3_.)\mapsto \langle
\nu^{1}_.,\ind\rangle/(\langle \nu_0^{1},\ind\rangle+\langle
\nu_0^{2},\ind\rangle+\langle
\nu_0^{3},\ind\rangle)$ (respectively $\langle \nu^{2}_.,\ind\rangle/(\langle
\nu_0^{1},\ind\rangle+\langle \nu_0^{2},\ind\rangle+\langle
\nu_0^{3},\ind\rangle)$ and $\langle
\nu^{3}_.,\ind\rangle/(\langle \nu_0^{1},\ind\rangle+\langle
\nu_0^{2},\ind\rangle+\langle
\nu_0^{3},\ind\rangle)$) are continuous from $\Co(\R_+,\mathcal{M}_{0,A}\times
\M_{\varepsilon,A}\times \M_{0,A})$ into
$\Co(\R_+,\R)$. \\
Using the moment assumption \eqref{eq:hypmoments}, the following mappings are also continuous for the same spaces: $\langle \nu^1_.,\chi\rangle/\langle \nu^2_.,\chi\rangle$, $\nu_.\mapsto \ind_{\langle \nu_.^{1},\chi\rangle>\varepsilon}/\langle \nu_.^{2},\chi\rangle$
and $\nu_.\mapsto \cro{\nu^2_.,\chi\left(\tau_1f-f\right)}$, for bounded function
$f$ on $\Z_+$ and where we recall that $\tau_1 f(k)=f(k-1)$ for every $k\in \Z_+$ (see Notation \ref{notation:part3}). Thus, using the continuity of the mapping $y\in \D([0,t],\R)\mapsto \int_0^t y_s\, \d s$, we obtain the continuity
of the mapping $\Psi^f_t$ defined in (\ref{eq:defPsi}) on $\D(\R_+,\mathcal{M}_{0,A}\times \mathcal{M}_{\varepsilon,A} \times
\mathcal{M}_{0,A})$. \\

By \eqref{eq:hypmoments}, the process
$(\NISn_{.\wedge \tau^N_\varepsilon})_{N\in \N }$ converges in
distribution to $\bNIS_.=\langle \bmuIS_.,\chi\rangle$. Since the
latter process is continuous, the convergence holds in
$(\D([0,T],\R_+),\|\cdot\|_{\infty})$ for any $T>0$ (see
\cite[p.~112]{billingsleyIII}). As $y\in \D(\R_+,\R) \mapsto \inf_{t\in
  [0,T]} y(t)\in \R$ is continuous, we have a.s. that:
\begin{equation*}
  \inf_{t\in [0,T]} \bNIS_t = \lim_{N\rightarrow +\infty} \inf_{t\in [0,T]} \NISn_{t\wedge \tau^N_\varepsilon} \quad
   \big(\geq \varepsilon\big).
\end{equation*}
Analogously to \eqref{eq:deftepsilon}, we consider $\bar{t}_{\varepsilon'}=\inf\{t\in \R_+,\,
\bNIS_t\leq \varepsilon'\}$ for $\varepsilon'>\varepsilon>0$.  A difficulty lies in the fact that we do not
know yet whether this time is deterministic. We have a.s.:
\begin{align}
  \label{etape2}
  \varepsilon'\leq
  \inf_{t\in [0,T]} \bNIS_{t\wedge \bar t_{\varepsilon'}} =
  \lim_{N\rightarrow +\infty} \inf_{t\in [0,T]} \NISn_{t\wedge
    \tau_\varepsilon^N \wedge \bar t_{\varepsilon'}}.
\end{align}
Hence, using Fatou's lemma:
\begin{align}
  1 = & \P\Big(\inf_{t\in [0,\bar t_{\varepsilon'}]}
  \bNIS_{t}>\varepsilon\Big)\nonumber\\
  \leq &  \lim_{N\rightarrow +\infty}
  \P\Big(\inf_{t\in [0,T\wedge \bar t_{\varepsilon'}]} \NISn_{t\wedge
    \tau^N_\varepsilon}>\varepsilon\Big)= \lim_{N\rightarrow +\infty}
  \P\Big(\tau^N_\varepsilon> T\wedge
 \bar t_{\varepsilon'}\Big).\label{etape_limite_taun}
\end{align}
We have hence
\begin{equation}\label{eq:etape-fin}
  \Psi_{.\wedge \tau^N_\varepsilon \wedge \bar t_{\varepsilon'}\wedge T}^{\cI\cS,f}(\mu^{N}) =
  \Psi_{.\wedge \tau^N_\varepsilon\wedge T}^{\cI\cS,f}(\mu^{N})\ind_{\tau^N_\varepsilon\leq \bar t_{\varepsilon'}\wedge T} +
  \Psi^{\cI\cS,f}_{.\wedge \bar t_{\varepsilon'}\wedge T}(\mu^{N}_{.\wedge \tau^N_\varepsilon})
  \ind_{\tau^N_\varepsilon>\bar t_{\varepsilon'}\wedge T}.
\end{equation}
From the estimates of the different terms in (\ref{decompo}), $\Psi_{.\wedge \tau^N_\varepsilon\wedge T}^{\cI\cS,f}(\mu^{N})$ is
upper bounded by a moment of $\mu^{N}$ of order 4. In view of (\ref{eq:hypmoments}) and (\ref{etape_limite_taun}),
the first term in the r.h.s.\ of \eqref{eq:etape-fin} converges in $L^1$ and hence in probability to zero. Using the continuity of $\Psi^{\cI\cS, f}$ on
$\D\left(\R_+,\M_{0,A} \times \M_{\varepsilon,A} \times \M_{0,A}\right)$, $\Psi^{\cI\cS,f}(\mu^{N}_{.\wedge \tau^N_\varepsilon})$
converges to $\Psi^{\cI\cS,f}(\bar{\mu})$ and therefore,
 $\Psi^{\cI\cS,f}_{.\wedge \bar t_{\varepsilon'}\wedge T}(\mu^{N}_{.\wedge \tau^N_\varepsilon})$ converges to
  $\Psi^{\cI\cS,f}_{.\wedge \bar t_{\varepsilon'}\wedge T}(\bar{\mu})$. Thanks to this and (\ref{etape_limite_taun}),
  the second term in the r.h.s.\ of \eqref{eq:etape-fin} converges to $\Psi_{.\wedge \bar t_{\varepsilon'}\wedge T}^{\cI\cS,f}(\bar{\mu})$ in
   $\D(\R_+,\R)$.\\
Then, the sequence 
$ (\langle \muISn_{.\wedge \tau^N_\varepsilon\wedge \bar t_{\varepsilon'}\wedge T},f\rangle - \Psi^{\cI\cS,f}_{.\wedge \tau_\varepsilon^N\wedge \bar t_{\varepsilon'}\wedge T}(\mu^{N}))_{N\in \N}$
 converges in probability to $\langle \bmuSI_{.\wedge \bar t_{\varepsilon'}\wedge T},f\rangle - \Psi^{\cI\cS,f}_{.\wedge \bar t_{\varepsilon'}\wedge T}(\bar{\mu})$.
 From (\ref{decompo}), this sequence also converges in probability to zero.

By identification of these limits, $\bmuIS$ solves (\ref{limitereseauinfiniSI}) on
$[0,\bar{t}_{\varepsilon'}\wedge T]$. If $\langle \bmuRS_0,\chi\rangle>0$ then similar techniques can be used. Else, the
 result is obvious since for all $t\in [0,t_{\varepsilon'}\wedge T]$, $\langle \muISn_t,\chi\rangle>\varepsilon$ and the term
 $p^N_t(j,\ell|k-1)$ is negligible when $\ell>0$. Thus $\bar{\mu}$ coincides a.s. with the only continuous deterministic solution
  of (\ref{limitereseauinfiniS})--(\ref{limitereseauinfiniSR}) on $[0,\bar t_{\varepsilon'}\wedge T]$. This implies that
  $\bar t_{\varepsilon'}\wedge T=t_{\varepsilon'}\wedge T$ and yields the convergence in probability of
   $(\mu^{N}_{.\wedge \tau^N_\varepsilon})_{N\in \N}$ to $\bar{\mu}$, uniformly on
   $[0,t_{\varepsilon'}\wedge T]$ since $\bar{\mu}$ is continuous.\\

We finally prove that the non-localized sequence $(\mu^{N})_{N\in
  \N}$ also converges uniformly and in probability to $\bar{\mu}$ in
$\D\left([0,t_{\varepsilon'}],\mathcal{M}_{0,A}\times \M_{\varepsilon,A} \times
\M_{0,A}\right)$. For a small positive $\eta$,
\begin{multline}
  \P\Big( \sup_{t \in [0,t_{\varepsilon'}]}\left|\langle \muISn_{t},f\rangle - \Psi^{\cI\cS,f}_{t}(\bar{\mu})\right|>\eta\Big)\\
  \leq \P\Big(\sup_{t\in [0,t_{\varepsilon'}]}\left|
    \Psi^{\cI\cS,f}_{t\wedge \tau^N_\varepsilon}(\mu^{N}) -
    \Psi^{\cI\cS,f}_{t
    }(\bar{\mu})\right|>\frac{\eta}{2} \, ;
  \,\tau^n_\varepsilon \geq t_{\varepsilon'} \Big)\\+
  \P\Big(\sup_{t\in
    [0,t_{\varepsilon'}]}\left|\Delta^{\textsc{n},f}_{t\wedge
      \tau^N_\varepsilon}+M_{t\wedge
      \tau^N_\varepsilon}^{\textsc{n},\cI\cS,f}\right|>
  \frac{\eta}{2}\Big)+\P\Big(\tau^N_\varepsilon<
  t_{\varepsilon'}\Big).\label{etape6}
\end{multline}
Using the continuity of $\Psi^f$ and the uniform convergence in probability proved above, the first term in the r.h.s.\ of
 (\ref{etape6}) converges to zero. We can show that the second term converges to zero by using Doob's inequality together
 with the estimates of the bracket of $M^{\textsc{n},\cI\cS,f}$ (similar to (\ref{etape3})) and of $\Delta^{\textsc{n},f}$ (Step 2). Finally,
 the third term vanishes in view of (\ref{etape_limite_taun}).\\

\par The convergence of the original sequence $(\mu^{N})_{N\in
  \N}$ is then implied by the uniqueness of the solution to
(\ref{limitereseauinfiniS})--(\ref{limitereseauinfiniSR}) proved in
Step 2.

\bigskip

\noindent \textbf{Step 5} When $N\rightarrow +\infty$, by taking the
limit in (\ref{eqmuSrenorm}), $\left(\muSn\right)_{N\in\N}$
converges in $\D(\R_+,\mathcal{M}_{0,A})$ to the solution of the following
transport equation: for
every bounded function $f\,:\,(k,t)\mapsto f_t(k)\in
\Co_b^{0,1}(\Z_+\times \R_+,\R)$ of class $\Co^1$ with bounded
derivative with respect to $t$,
\begin{align}
  \langle \bar{\mu}^\cS_t,f_t\rangle = & \langle
  \bar{\mu}_0^\cS,f_0\rangle -\int_0^t \langle \bar{\mu}^\cS_s,\lambda\chi \bpI_s
  f_s - \partial_sf_s \rangle\,\d s.
\end{align}
Choosing $f(k,s)=\varphi(k)\exp\big(-\lambda k\int_0^{t-s}\bpI(u)du\big)$, we
obtain that
\begin{equation}
  \langle \bar{\mu}^\cS_t,\varphi\rangle= \sum_{k\in \Z_+} \varphi(k)\theta_t^k  \bar{\mu}_0^\cS(k).
\end{equation}where $\theta_t=\exp\big(-\lambda\int_0^{t}\bpI(u)du\big)$ is the probability that a given degree 1 node remains susceptible at time $t$. This is the announced Equation (\ref{limitereseauinfiniS}).

The proof of Theorem \ref{propconvergencemunS} is now completed. \hfill $\Box$

\bigskip
Recall that the time $t_{\varepsilon'}$ has been defined in \eqref{eq:deftepsilon}. We end this section with a lower bound of the time $t_{\varepsilon'}$ until which we proved that the convergence to Volz'
equations holds.
\begin{proposition}
\label{pro:horizon}
Under the assumptions of Theorem \ref{propconvergencemunS}, 
\begin{equation}
\label{eq:deftauepsilon}
t_{\varepsilon'}>\bar\tau_{\varepsilon'}:=\frac{\log\left(\cro{\bmuS_0,\chi^2}+\bNIS_0\right)
-\log\left(\cro{\bmuS_0,\chi^2}+\varepsilon'\right)}{\max(\gamma, \lambda)}.
\end{equation}
\end{proposition}

\begin{proof}
Because of the moment Assumption \ref{hypconvcondinit} and (\ref{eq:hypmoments}), we can prove that (\ref{decompo}) also holds for $f=\chi$. This is obtained by replacing in (\ref{etape777}), (\ref{etape5}), (\ref{etape753}) and (\ref{etape779}) $\|f\|_\infty$ by $k$ and using the Assumption of boundedness of the moments of order 5 in (\ref{etape5}) and (\ref{etape779}). This shows that $(\NISn)_{N\in \N}$ converges, uniformly on $[0,t_{\varepsilon'}]$ and in
  probability, to the deterministic and continuous solution
  $\bNIS=\langle \bmuIS,\chi\rangle$. We introduce the event $\mathcal A^{N}_{\xi}= \{\mid \NISnn_0-N \bNIS_0\mid \le \xi\}$ where their differences are bounded by $\xi>0$.
Recall the definition (\ref{eq:deftauepsilonn}) and let us introduce the number of edges $Z^N_t$ that were $\cI\cS$ at time $0$ and that have been removed before $t$. For $t\geq \tau^N_{\varepsilon'}$, we have necessarily that $Z^N_t\geq \NISnn_0-N\varepsilon'$. Thus,
\begin{align}
\P\big(\{\tau^N_{\varepsilon'} \le t\}\cap \mathcal A^N_\xi\big)\leq & \P\big(\{Z^N_t > \NISnn_0 - N\varepsilon' \} \cap \mathcal A^N_{\xi}\big)\nonumber\\
\le &\P \Big(\big\{Z^N_t > N(\bNIS_0 -\varepsilon')-\xi\big\} \cap \mathcal A^N_{\xi} \Big).\label{etape12}
\end{align}
When susceptible (resp.\ infectious) individuals of degree $k$ are contaminated (resp.\ removed), at most
$k$ $\cI-\cS$-edges are lost. Let $X^{N,k}_t$ be the number of edges that, at time $0$, are $\cI-\cS$ with susceptible alter of
degree $k$, and that have transmitted the disease before time $t$. Let $Y^{N,k}_t$ be the number of initially infectious
individuals $x$ with $d_x(\cS_0)=k$ and who have been removed before time $t$. $X^{N,k}_t$ and $Y^{N,k}_t$ are bounded by
$k \muSnn_0(k)$ and $\muSInn_0(k)$. Thus:
\begin{equation}
Z^N_t\leq \sum_{k\in \Z_+}k\big(X^{N,k}_t+ Y^{N,k}_t\big).
\label{eq:majoreZ1}
\end{equation}
Let us stochastically bound $Z^N_t$ from above. 
Since each $\cI-\cS$-edge transmits the disease independently at rate $\lambda$, $X^{N,k}_t$ is stochastically dominated by a binomial r.v. of parameters
$k \muSnn_0(k)$ and $1-e^{-\lambda t}$. We proceed similarly for $Y^{N,k}_t$. Conditional on the initial condition, $X^{N,k}_t+Y^{N,k}_t$ is thus stochastically dominated by a binomial r.v. $\tilde Z_t^{N,k}$ of parameters $(k \muSnn_0(k)+  \muISnn_0(k))$ and $1-e^{-\max(\lambda,\gamma)t}$. Then \eqref{etape12} and \eqref{eq:majoreZ1} give:
\begin{align}
\P\big(\{\tau^N_{\varepsilon'} \le t\} \cap \mathcal A^N_\xi\big)\leq & \P\Big(\sum_{k\in \Z_+} \frac{k \tilde Z_t^{N,k}}{N}>
\bNIS_0 -\varepsilon'-\frac{\xi}{N}\Big).\label{eq:horizon1}
\end{align}Thanks to Assumption \ref{hypconvcondinit} and \eqref{eq:hypmoments}, the series
$\sum_{k\in \Z_+} k \tilde Z_t^{N,k}/N$ converges in $L^1$ and hence in probability to $(\langle \bmuS_0,\chi^2\rangle +
\bNIS_0)(1-e^{-\max(\lambda,\gamma)t})$ when $N\rightarrow +\infty$. Thus, for sufficiently large $N$,
\begin{align*}
\P\big(\{\tau^N_{\varepsilon'} \le t\}\cap \mathcal A^N_\xi \big)= &
1  \mbox{ if }t>\bar{\tau}_{\varepsilon'}\mbox{ and }  0  \mbox{ if }t<\bar{\tau}_{\varepsilon'}.
\end{align*}
For all $t<\bar{\tau}_{\varepsilon'}$, it follows from Assumption \ref{hypconvcondinit}, \eqref{eq:hypmoments} and
Lemma \ref{lemme:convfnonbornee} that:
\begin{equation*}
\lim_{N\rightarrow +\infty}\P\Big( \tau^N_{\varepsilon'} \le t \le \lim_{N\rightarrow +\infty}\P\big(\left\{\tau^N_{\varepsilon'}
\le t\right\}\cap \mathcal A^N_{\xi}\big)\Big)+\P\big((\mathcal A^N_\xi)^c\big)
= 0,
\end{equation*}
so that by Theorem \ref{propconvergencemunS}
\begin{align*}
1=  \lim_{N\rightarrow +\infty}\P(\tau^N_{\varepsilon'} \geq \bar{\tau}_{\varepsilon'}) =& \lim_{n\rightarrow +\infty}\P\Big(\inf_{t\leq \bar{\tau}_{\varepsilon'}}\NISn_t \geq \varepsilon'\Big)
= \P\Big(\inf_{t\leq \bar{\tau}_{\varepsilon'}}\bNIS_t \geq \varepsilon'\Big).
\end{align*}
This shows that $t_{\varepsilon'} \ge \bar\tau_{\varepsilon'}$ a.s., which concludes the proof.
\end{proof}



\chapter[Statistical Description of Epidemics Spreading on Networks: The Case of Cuban HIV]{Statistical Description of Epidemics Spreading on Networks: The Case of Cuban HIV}\label{sec:stat-reseau}
\chaptermark{Epidemics on Networks: The Case of Cuban HIV}

In this section, we turn our attention to epidemics spreading on networks. Probability models have been described in Section \ref{sec:graphes}. We now deal with the statistical treatment of data obtained from diseases propagating on networks. The statistical methods described here are illustrated on the sexual network obtained from the Cuban HIV contact-tracing system that we now describe. For a complete description of the Cuban network, we refer to \cite{clemenconarazozarossitranIII}. The Cuban graph is available as supplementary material of this book.\\

Since 1986, a contact-tracing detection system has been set up in Cuba in order to bring the spread of the HIV epidemic under control. It has also enabled the gathering of a considerable amount of detailed epidemiological data at the individual level. In the resulting database, any individual tested as HIV positive is indexed and anonymized for confidentiality reasons. Information related to uninfected individuals is not recorded in the data, and of course infected individuals not diagnosed yet are also absent. The network only consists of detected HIV+ individuals. However, note that the network is age-structured and data related to the infectious population of the first six years of the epidemic seems to show (e.g.\ \cite{arazozaclemencontranIII}) that this population has been discovered by now.\\
Individuals in the database are described through several attribute variables: gender and sexual orientation, way of detection, age at detection, date of detection, area of residence, etc. In the sequel, we will mainly focus on the gender/sexual orientation, for which three modalities are identified: `woman', `heterosexual man', `MSM'  (Men who have Sex with Men; men who reported at least one sexual contact with another man in the two years preceding HIV detection). Because Female-to-female transmission is neglected, no sexual orientation is distinguished for women  (e.g.\ \cite{chanetalIII}). It is worth recalling that in Cuba HIV spreads essentially through sexual transmission. Infection by blood transfusion or related to drug use are neglected. We refer to \cite{auvertIII} for a preliminary overview of the HIV/AIDS epidemics in Cuba, as well as a description and the context of the construction of the database used in the present study and the context in which it was constructed.\\

Importantly, for each HIV+ individual that is detected, the list of indices corresponding to the sexual partners appearing in the database she/he possibly named for contact-tracing is also available. In \cite{clemencon_etal_ESANN2011III,clemencon_etal_IWANN2011III,clemenconarazozarossitranIII} the graph of sexual partners that have been diagnosed HIV positive on the Cuban data repository is reconstructed and an exploratory statistical analysis of the resulting sexual contact network is carried. The network is composed of 5,389 vertices, or nodes, that correspond to the individuals diagnosed as HIV positive between 1986 and 2006 in Cuba, i.e.\  1,109 women (20.58\%) and 4,280 men (79.42\%); 566 (10.50\%) of which are heterosexual and 3,714 (68.92\%) are MSMs. Individuals declared as sexual contacts but who are not HIV positive are not listed in the database: the only observed vertices correspond to individuals who have been detected as HIV positive or AIDS. The vertices that depict the fact that two individuals have been sexual partners during the two years that preceded the detection of either one are linked by 4,073 edges. Only edges between observed HIV cases are hence observed, but the degree (total number of sexual partners) is known. Also, some information is documented on who infects whom, giving access to a partial infection tree. Our data exhibit a ``giant component", counting 2,386 nodes. The second largest component has only 17 vertices and there are about 2000 isolated individuals or couples. It is remarkable that in the existing literature on sexually transmitted diseases graph networks are generally smaller and/or do not exhibit such a large connected component and/or contain a very small number of infected persons (e.g.\ \cite{rothenbergwoodhousepotteratmuthdarrowklovdahlIII,wyliejollyIII}). \\

In Section \ref{subsec:visual-mining}, using graph-mining techniques, the connectivity/communication properties of the sexual contact network are described to understand the impact of heterogeneity (with respect to the attributes observed) in the graph structure. Particular attention is paid to the graphical representation of the data, as conventional methods cannot be used with databases of the size of the one used in this study. A clustering of the population is performed so as to represent structural information in an interpretable way.
Beyond global graph visualization, the task of partitioning the network into groups, with dense internal links and low external connectivity, is known as clustering. In contrast to standard multivariate analysis, in which the network structure of the data is ignored, our method has shed light on how different mechanisms  (e.g.\ social behaviour, detection system) have affected the epidemics of HIV in the past, and provide a way of predicting the future evolution of this disease. This study paves the way for building more realistic network models in the field of mathematical modelling of infectious diseases.

\section{Modularity and assortative mixing}

Assortative mixing coefficients can be computed to highlight the possible existence of selective linking in the network structure. Various measures have been proposed in the literature for quantifying the tendency for individuals to have connections with other individuals that are similar in regards to certain attributes, depending on the nature of the latter (quantitative vs. qualitative). For a partition of $J$ classes,
$\mathcal{P}=C_1,\;\ldots,\; C_J$,
one may calculate the proportion $m_{i,j}$ of edges in the graph connecting a node lying in group $i$ to another one in group $j$, $1\leq i \leq j\leq J$ and build the $J\times J$ mixing matrix $\mathcal{M}=(m_{i,j})$ (notice it is symmetric since edges are not directed here). We can then define the modularity coefficient $Q_{\mathcal{P}}$ (e.g.\ \cite{newman_girvan_PRE2004III}) by:
\begin{equation}
 Q_{\mathcal{P}}=\mathrm{Tr}(\mathcal{M})-\vert\vert\mathcal{M}^2\vert\vert=\sum_{i}\left\{m_{i,i}-\left(\sum_{j=1}^N m_{i,j}\right)^2\right\}, \label{eq:modularity}
 \end{equation}where $\vert\vert A\vert\vert=\sum_i\sum_j a_{i,j}$ denotes the sum of all the entries of a matrix $A=(a_{i,j})$ and $\mathrm{Tr}(A)$ its trace when the latter is square. \\

We can define the assortative coefficient as
\[r=Q_\mathcal{P}/(1-\vert\vert \mathcal{M}^2\vert\vert).\]
As pointed out in \cite{newman2003III}, large values of $r$ indicate "selective linking": values around $0$ correspond to randomly mixed network, whereas values close to $1$ are associated with perfectly assortative network. The assortative coefficient can also be negative.\\

\begin{table}[!ht]
\begin{center}
{\small
\begin{tabular}{|c|c|c|c|c|}
\hline
Ego & Alter is & Alter is & Alter is & Total\\
is a  & a woman & a heterosexual man & an MSM & \\
 \hline
Woman & 77 (\textit{1.9\%}) & 157 (\textit{3.9\%}) &  408 (\textit{10.0\%}) & 642 (\textit{15.8\%}) \\
HT man & 282 (\textit{6.9\%})& 4 (\textit{0.1\%})& 20 (\textit{0.5\%}) & 306 (\textit{7.5\%})\\
MSM & 800 (\textit{19.6\%})& 25 (\textit{0.6\%}) & 2300 (\textit{56.5\%}) & 3125 (\textit{76.7\%})\\
 \hline
 Total & 1159 (\textit{28.5\%}) & 186 (\textit{4.6\%}) & 2728 (\textit{67.0\%}) & \\
 \hline
\end{tabular}
}
\end{center}
\caption{{\small\textit{Sexual orientation of egos and alters for the edges in the whole graph. The figures presented here account for the direction of the edges: egos are detected first and alters are the partners they refer to during the contact-tracing interviews. Frequencies are given together with row and column proportions between brackets. The diagonal of the contingency table represents 58,46\% of the whole edges. The assortative mixing coefficient is $r=0.0512$. The independence between the sexual orientation of egos and alters is rejected by a $\chi^2$-test with a p-value smaller than $2.2\, 10^{-16}$. In theory, there should be no sexual contact between two heterosexual men or
between a heterosexual man and an MSM. The semantic of the database also
exclude sexual contact between women. However, those events actually occur in the
dataset.}}}
\label{table:transmission_whole}
\end{table}

A first class of partitions are constituted by nodes taking the same modalities of qualitative variables: area of residence, sexual orientation, age, detection mode...  Let us comment on the partition defined by the gender/sexual orientation variable (see Table \ref{table:transmission_whole}). As edges correspond to sexual contacts in the present graph, the gender/sexual orientation of adjacent vertices cannot be arbitrary of course. More than a half of the edges (56.47\%) link two MSM. Links between MSM and women make 1,208 edges (29.66\%) and there are 439 edges (10.78\%) between women and heterosexual men. Looking at the infection tree provided similar proportions: 1,202 edges (52.56\%), 667 edges (29.16\%) and 375 edges (16.40\%) respectively. Figures reveal an asymmetry in HIV infection: among (oriented) infection edges involving women, the latter are more often alters than egos (66.13\% of the edges shared with heterosexual men and 74.21\% of the edges shared with MSM). The declarative degree shows a smaller mean degree for heterosexual men and comparable degree distributions between women and MSM. MSM are expected to contribute most to the connectivity of the graph, especially bisexual men who act as contact points between women and MSM who declare only contacts with men.\\

Of course, a natural question is to see whether we can define other partitions that are more closely related to the modularity defined in \eqref{eq:modularity}. This is the topic of the next section, which is related with visual-mining and modularity clustering.

\section{Visual-mining}\label{subsec:visual-mining} 

Graph visualization techniques are used routinely to gain insights about
medium size graph structures, but their practical relevance is questionable
when the number of vertices and the density of the graph are high both for
computational issues (as many graph drawing algorithms have high complexities)
and for readability issues
\cite{DiBattistaEtAl1999GraphDrawingIII,HermanEtAl2000GraphIII}. We illustrate the clustering and visualization on the Cuba HIV data where the situation
is borderline as the giant component of the graph contains 2,386 vertices and
3,168 edges (respectively 44.28\% and 77.78\% of the global
quantities). As the graph is of medium size from a computational point of view
and has a low density, it is a reasonable candidate for state-of-the-art
global and detailed visualization techniques. We use the optimised force
directed placement algorithm proposed in \cite{tunkelang_PHD1999III}. It
recasts the classical force directed paradigm
\cite{fruchterman_reingold_SPE1991III} into a nonlinear optimization problem in
which the following energy is minimised over the vertex positions in the
euclidean plane, $(z_1,\ldots,z_n)$,
\begin{equation*}
\mathcal{E}(z_1,\ldots,z_n)=\sum_{1\leq i\neq j\leq n}\left(a_{i,j}\frac{1}{3\delta}\|z_i-z_j\|^3-\delta^2\ln\|z_i-z_j\|\right),
\end{equation*}
where, $\delta$ is a free parameter that is roughly proportional to the
expected average distance between vertices in the plane at the end of the
optimization process, $a_{i,j}$ are the terms of the adjacency matrix of the network
and $\|\cdot\|$ denotes the Euclidean distance in the plane.\\

However, the structure of the graph under study, in particular its uneven
density, has adverse effects on the readability of its global
representation. We rely therefore on the classical simplification approach
\cite{HermanEtAl2000GraphIII} that consists in building a clustering of the
vertices of the graph and in representing the simpler graph of the
clusters. More precisely, the general idea is to define a partition composed of groups with dense internal links but low inter-group connectivity. Each group can then be considered as a vertex of a new graph: two
such vertices are connected if there is at least one pair of original vertices
in each group that are connected in the original graph.

\begin{figure}[!ht]
\begin{center}
\begin{tabular}{cc}
\includegraphics[height=5cm]{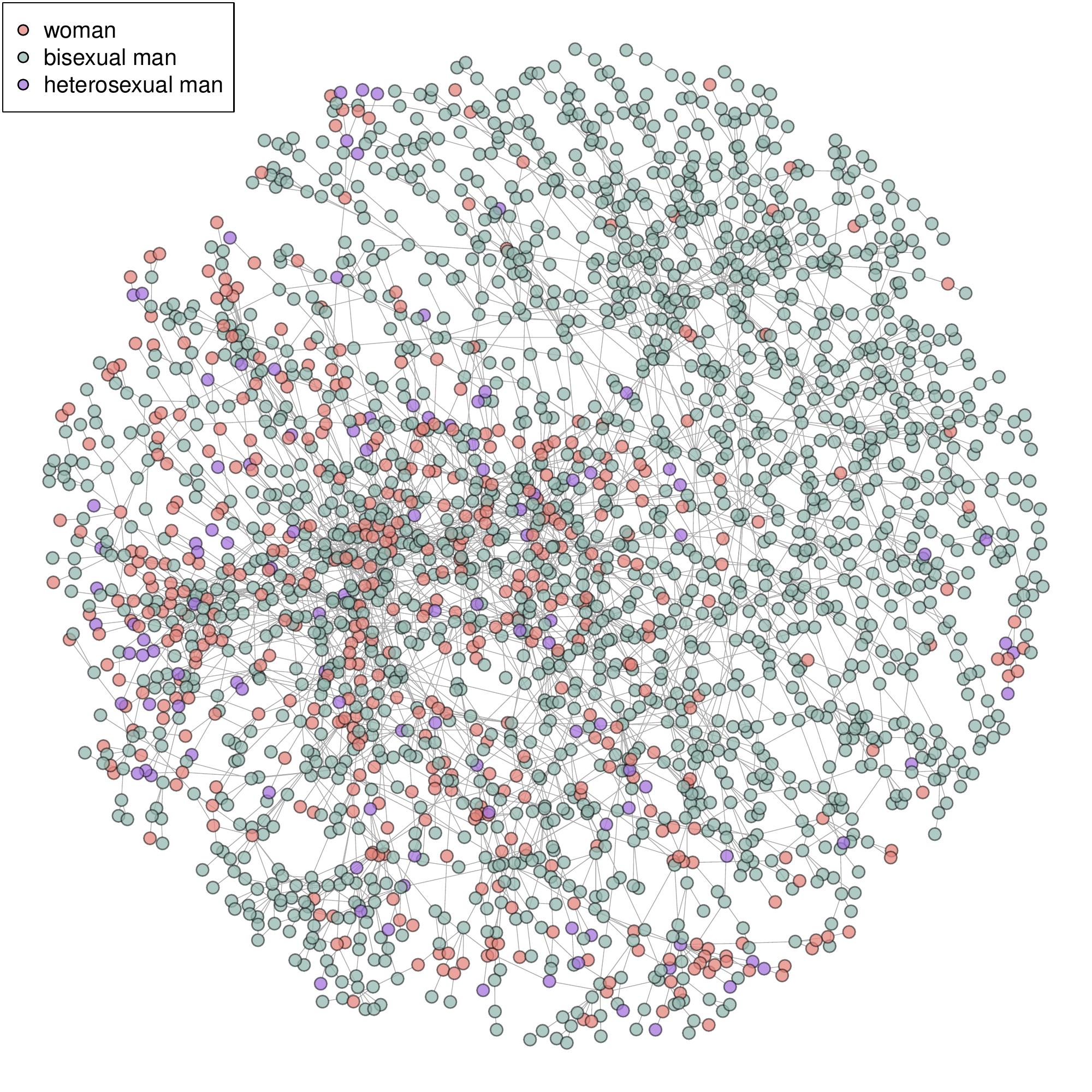} &
\includegraphics[width=5cm]{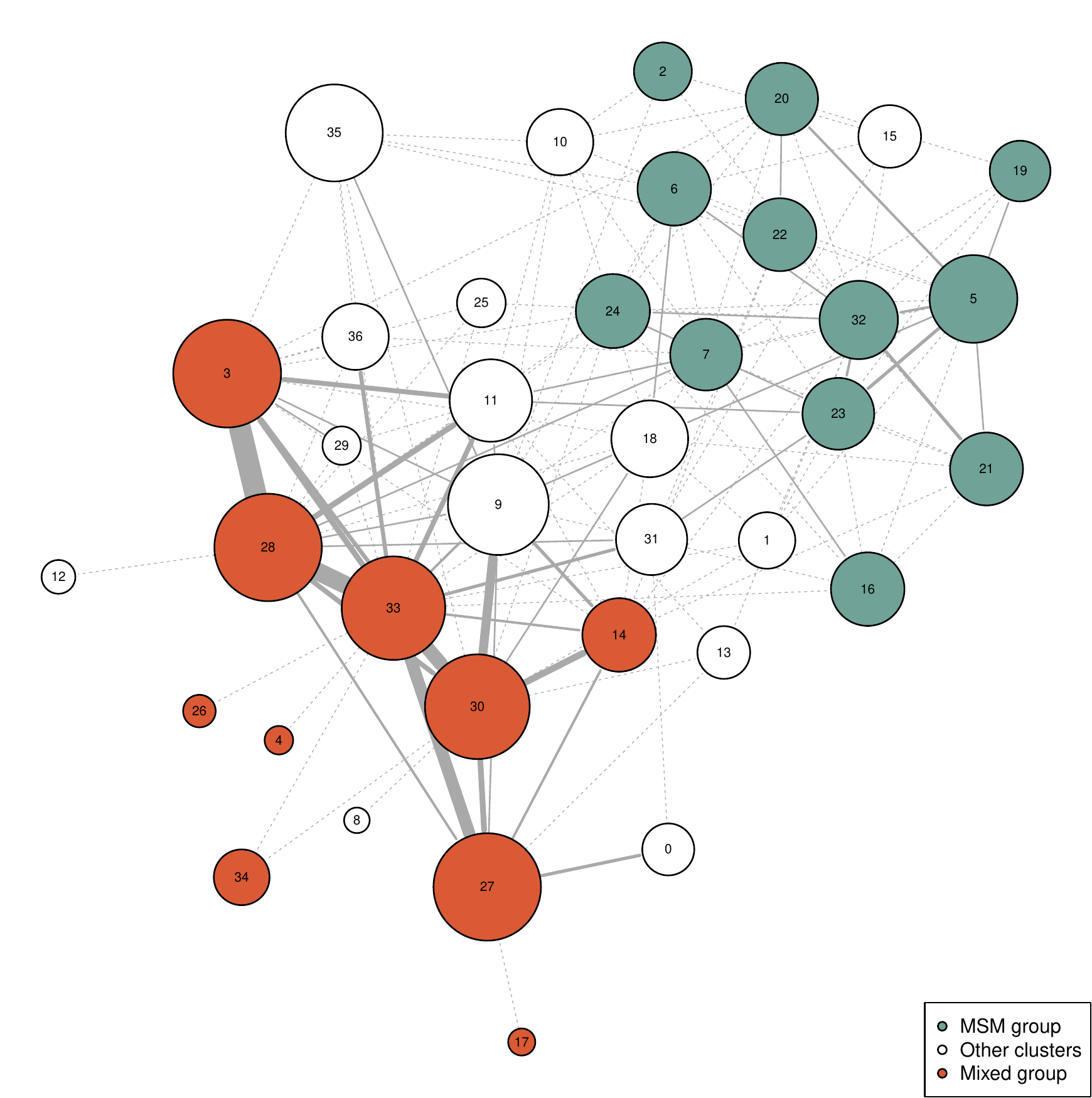}\\
(a) & (b)
\end{tabular}
\caption{{\small \textit{(a): Raw view of the giant component for the Cuban HIV epidemics. (b) Modularity clustering of the giant component in 37 classes.}}}
\label{fig:comp_largest}
\end{center}
\end{figure}

Following \cite{clemencon_etal_IWANN2011III,clemencon_etal_ESANN2011III,rossivilla-vialaneix2011societe-fran-caiseIII},
we compute a maximal modularity clustering \cite{newman_girvan_PRE2004III} as the
obtained clusters are well adapted to subsequent visual representation, as
shown in \cite{Noack2009III}.
Maximizing $Q_{\mathcal{P}}$ over all the partitions $\mathcal{P}$ provides an optimal $J$ classes partition. This is an NP-Hard and can only be solved via some heuristics. As in \cite{rossivilla-vialaneix2011societe-fran-caiseIII}, we use a
modified version of the multi-level greedy merging approach proposed in
\cite{NoackRotta2009MultiLevelModularityIII}: our modification guarantees that
the final clusters are connected. The optimization
process is carried out on the partitions for a given number of clusters $J$ but
also over the number of clusters $J$ itself which is then automatically selected. This makes the method essentially parameter free.\\

It should be noted however that one can find partitions with a rather high
modularity even in completely random graphs (configuration model graphs where vertices have different degrees but are paired independently) where no modular structure
actually exists (see \cite{ReichardtBornholdt2007ModularityDistributionIII} for
an estimation of the expected value of this spurious modularity in the limit
of large and dense graphs). To check that the modular structure found in a
network cannot be explained by this phenomenon, we use the simulation
approach proposed in
\cite{clemenconarazozarossitranIII,rossivilla-vialaneix2011societe-fran-caiseIII}. Using a Markov Chain Monte Carlo (MCMC) approach inspired by \cite{Roberts2000MCMCIII},
we generate configuration model graphs with exactly the same size and degree distribution as the epidemics graph. Using the above algorithm, we compute a
maximal modularity clustering on each of those graphs. The modularities of the
clustering provide an estimate of the distribution of the maximal
modularity in random graphs with our degree distribution. If a partition of this graph
exhibits a higher modularity, we conclude that it must be the result of
some actual modular structure rather than a random outcome.\\

The maximal modularity clustering is visualised using the force
directed placement algorithm described above. In addition to giving a general
idea of the global structure of the graph, the obtained visual representation
can be used to display distributions of covariates at the cluster level. Homogeneity tests are performed in order to
assess possible significant differences between these statistical
subpopulations.\\

However, as demonstrated in \cite{FortunatoBarthelemy2007III}, finding the
maximal modularity clustering can lead to ignoring small modular structures
that fall below the resolution limit of the modularity measure. It is then
recommended in \cite{FortunatoBarthelemy2007III} to recursively apply maximal
modularity clustering to the original clusters in order to investigate
potential smaller scale modules. We follow this strategy coupled with the
MCMC approach described above: each cluster is tested for substructure by
applying the maximal modularity clustering technique from
\cite{rossivilla-vialaneix2011societe-fran-caiseIII} and by assessing the actual
significance of a potential sub-structure via comparison with similar random
graphs.

To sum up, we recall the procedure that we recommend for clustering a large network:
\begin{itemize}
\item maximization of the modularity \eqref{eq:modularity} (see \cite{newman_girvan_PRE2004III}).
\begin{itemize}
\item this favours dense clusters and produces interesting partitions for visualization (Fortunato 2010)
\item the optimisation is an NP-hard problem but high quality sub-optimal solutions can be obtained by annealing (Rossi Villa-Vialaneix 2010) or other methods (Noak Rotta, 2009)
    \end{itemize}
\item Clustering significance:
\begin{itemize}
\item compute the modularity of the partition that is obtained,
\item test the significance of the obtained partition by simulating configuration models with same degree distribution and compute modularity.
    \end{itemize}
\item \textbf{Hierarchical clustering:} if the first clustering is relevant, and if the classes have large sizes, we can refine the partition.
\begin{itemize}
\item Reiterate the clustering for each element of the partition, without taking inter-cluster connections.
\item Test the significance of the cluster's partition
\item Test the significance of the global clustering of the graph.
\end{itemize}
\item \textbf{Coarsening:} merge clusters that induce the least reduction in modularity as long as we remain above the original graph.
\item \textbf{Visualization:} use the Fruchterman--Reingold algorithm to display the network of clusters
\end{itemize}

\begin{figure}[!ht]
\begin{center}
\includegraphics[width=6.5cm]{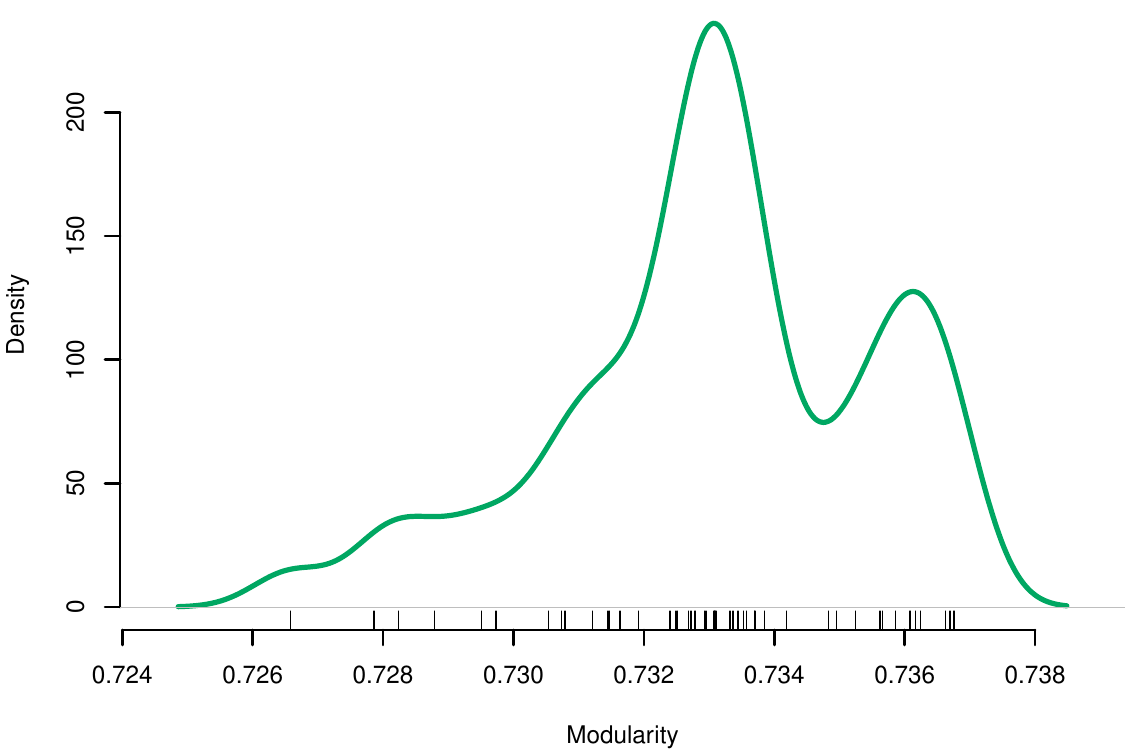}
\caption{\textit{{\small In Figure \ref{fig:comp_largest}, a modularity clustering is performed on the Cuban HIV data. The modularity of the partition obtained is $\simeq 0.85$. To test the significancy of this partition, 100 configuration model graphs with same size and same degree distribution as the observed one are simulated. The empirical distribution of the random modularity obtained by these simulations is depicted with small black bars on the abscissa axis and has a support bounded by $0.74$. This shows that the partition obtained by maximizing the modularity is significant (at level 95\% for instance). }}}
\end{center}
\end{figure}

\section{Analysis of the ``giant component''}\label{subsec:giant} 

The network density is globally low and very heterogeneous. But although the connectivity of the network seems fragile at first glance, density may be locally very high. The harmonic average of the geodesic path lengths equals 10.24 and 12.2 for the directed graph (taking into account the information of who mentions whom). Most of the graph connectivity is concentrated in the largest component (3,168 edges out of 4,073). The largest component has a diameter of 26 (36 when taking into account the direction of the infections) and the harmonic average of the geodesic path lengths are the same inside the largest component. These values are slightly higher than those of other real networks mentioned in \cite{newman_SIAMIII} but remain well below the number of vertices and compatible with the logarithmic scaling related to the so-termed small world effect.\\

Figure \ref{fig:comp_largest} (b) seems quite clear, with what appears to be two parts in the graph: the
lower part of the graph (on the figure) seems to be dominated by MSM while
the upper part gathers almost all persons from the giant component that have
only heterosexual contacts. However, the upper part is quite difficult to read
as it seems denser than the lower part. The
layout shows what might be interpreted as cycles and also a lot of small trees
connected to denser parts. The actual connection patterns between the upper
part and the lower part are also very unclear. Because of these crowding
effects, structural properties of the network from Figure
\ref{fig:comp_largest} appears quite difficult and probably
misleading. We rely therefore on the simplification technique outlined in
Section \ref{subsec:visual-mining} leveraging a clustering of the giant
component to get an insight into its general organization. 

\begin{figure}[!ht]
\begin{center}
\begin{tabular}{cc}
\includegraphics[width=5.8cm]{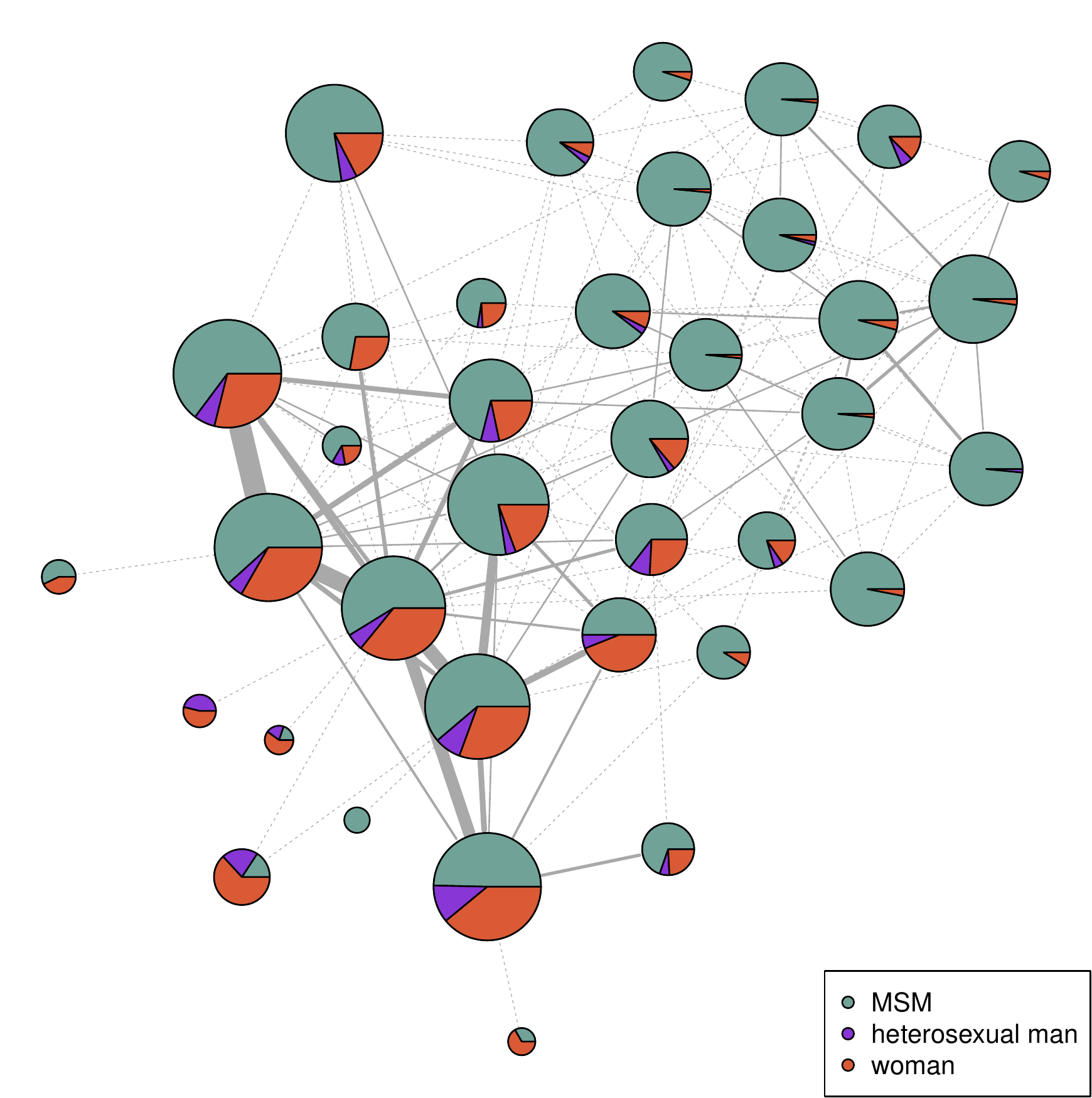} & \includegraphics[width=5.8cm]{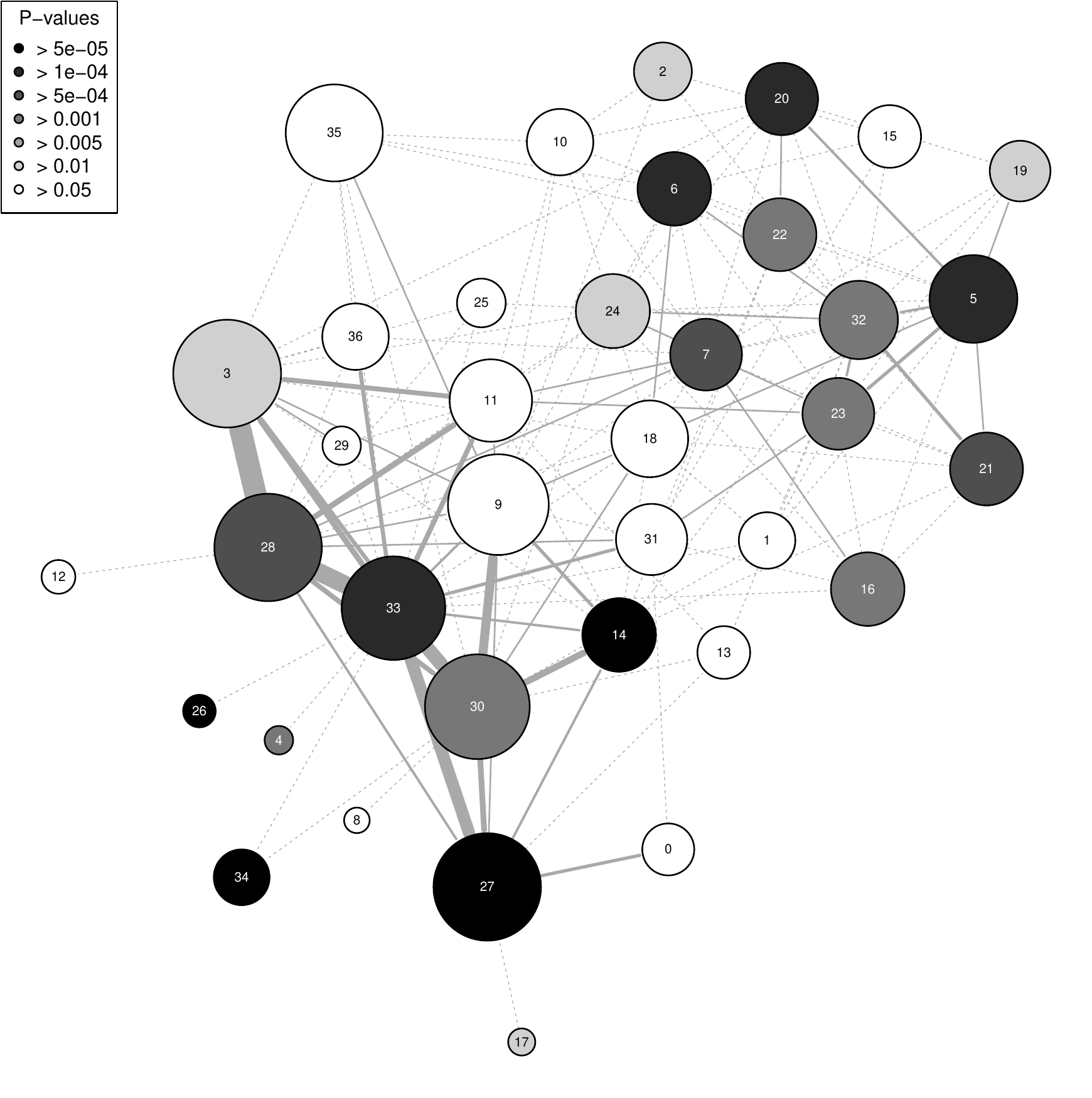}\\
(a) & (b)
\end{tabular}
     \caption{{\small\textit{The giant component divided into 37 clusters. (a) Each disk of
       representation corresponds to one cluster and has an area proportional
       to the number of persons (original vertices) gathered in the associated
       cluster. The pie chart of the disk displays the percentage of MSM (green),
       of heterosexual men (blue) and of women (red) in the cluster. Links between clusters
       summarise the connectivity pattern between members of the clusters. The
       thinnest edge width corresponds to only one connection between a member
       of one cluster and another person in the connected cluster (the
       corresponding edges are drawn using dashed segments). Thicker
       edges have a width proportional to the number of connected persons. (b) Disk areas and
       edges thicknesses are chosen as in (a). The grey
       level of a disk encode the $p$-value of a $\chi^2$ test of
       homogeneity in which the distribution of the sexual orientations in the
     associated cluster is compared to the distribution in the giant component.}}
     }
     \label{fig:cluster1}
\end{center}
\end{figure}

A graphical representation of the partition obtained by the
method from \cite{rossivilla-vialaneix2011societe-fran-caiseIII} is displayed in
Figure \ref{fig:cluster1} (a). The clustering thus produced exhibits a
modularity of 0.8522 and is made up of 37 clusters.
This modularity is
very high compared to the random level and strongly supports the hypothesis of
a specific (``non-random") underlying community structure. For comparison
purpose, the average maximal modularity attained by random graphs built from a
configuration model with the same size and degree distribution as those of the
giant component observed over a collection of 100 simulated replications
(using the same partitioning method) is of the order 0.74, with a maximum of
0.7435.\\

Considering that the modules are meaningful, the visual representation
provided by Figure \ref{fig:cluster1} (a) is more faithful to the underlying
graphical structure than the finer displays of Figure
\ref{fig:comp_largest} (b). That said, the two graphs tend to agree as the pie
charts of Figure \ref{fig:cluster1} clearly show two parts in the network: the
lower left part seems to gather most of the women and heterosexual
men (as the upper part of Figure \ref{fig:comp_largest} (b)), while the
upper right part contains clusters made almost entirely of MSM, as the lower
part of Figure \ref{fig:comp_largest} (b). While the display of Figure
\ref{fig:cluster1} (a) might seem cluttered, it is in fact very readable if one
considers that only 328 edges of the giant component connect persons from
different clusters while 2,840 connections happen inside clusters. Then most of
the edges on Figure \ref{fig:cluster1} (a) could be disregarded as they
corresponds to only one pair of connected persons (this is the case of 94 of
such edges out of 142 and the former are represented as dashed segments). Taken
this aspect into account, it appears that the MSM part of the giant component
(upper right part) is made of loosely connected clusters while the bulk of the
connectivity between clusters is gathered in the mixed part of the component,
in which most women and heterosexual men are gathered. The fact that the mixed part is more dense was already visible in Figure
\ref{fig:comp_largest} (b), but Figure \ref{fig:cluster1} (a) provides a much
stronger demonstration.\\

The pie chart based visualization of Figure \ref{fig:cluster1} (a) shows the sexual orientation distribution in the clusters and hence sheds light on its relationship with the graphical
structure. In Figure \ref{fig:cluster1} (b), a visual representation
of the corresponding $p$-values is given. The darker the node, the more statistically significant
the difference between the cluster distribution of sexual orientation and the
distribution of the giant component.

Combining Figures
\ref{fig:cluster1} (a) and (b) is very useful: Figure (b) highlights atypical clusters while Figure (a) identifies why they are atypical. It appears that among the 37 clusters, 22 exhibit a $\chi^2$ p-value below
5\%. They will be abusively referred to as ``atypical clusters" in the
following. The set of those clusters can be split into two subsets, depending
on the percentage of MSM in the cluster: above or below the global value of
76\% (the percentage in the giant component), as illustrated by Figure
\ref{fig:cluster1} (b). Almost two thirds (67\%) of the
individuals of the largest connected component lie in the atypical
clusters. Among the latter, 774 individuals belong to the 12 clusters which
display a large domination of MSM (denoted the MSM group of clusters in the
sequel) and 825 to the 10 clusters that contain an unexpectedly large number
of heterosexual persons (denoted the mixed group of clusters in the
sequel). \\

According to Figure \ref{fig:cluster1}, the two subsets of atypical
clusters seem to be almost disconnected. This is confirmed by a detailed
connectivity analysis. There are indeed 864 internal connections in the MSM group,
1,276 in the heterosexual group, and only 10 links between pairs of
individuals belonging to the two different groups. This asymmetry was expected, given the quality of the clustering with only 328 inter-cluster connections. Nevertheless, the number of
connections between the two groups of clusters is also small compared to
connections between the clusters of the groups: 129
connections between persons of distinct clusters in the group of mixed
clusters and 55 in the group of MSM clusters. Finally, there are 83
connections from persons in the group of mixed clusters to persons in non-atypical clusters, and 36 connections from persons in the group of MSM
clusters to persons in non-atypical clusters. Mean geodesic distances inside the MSM group are larger than in the mixed
group (respectively 9.95 and 7.28, computed without orientation). To conclude, the two groups are weakly connected to the outside, with a small number of direct connections, and rather internally more connected than expected.

\section{Descriptive statistics for epidemic on networks}

We now review some basic descriptive statistics for networks. Exhaustive statistical exploration of networks has been described by Newman \cite{newman_SIAMIII} for example.

\subsection{Estimating degree distributions}\label{subsec:degree_dist}

For the Cuban HIV data, we want to calculate for instance the degree distribution $(p_k:\; k\in \mathbb{N})$ using the number of declared sexual partners in the two years preceding detection, where $p_k$ is the proportion of vertices having declared $k$ sexual partners.\\ 

\begin{figure}[!ht]
\begin{center}
\begin{tabular}{cc}
(a) & \includegraphics[width=8cm]{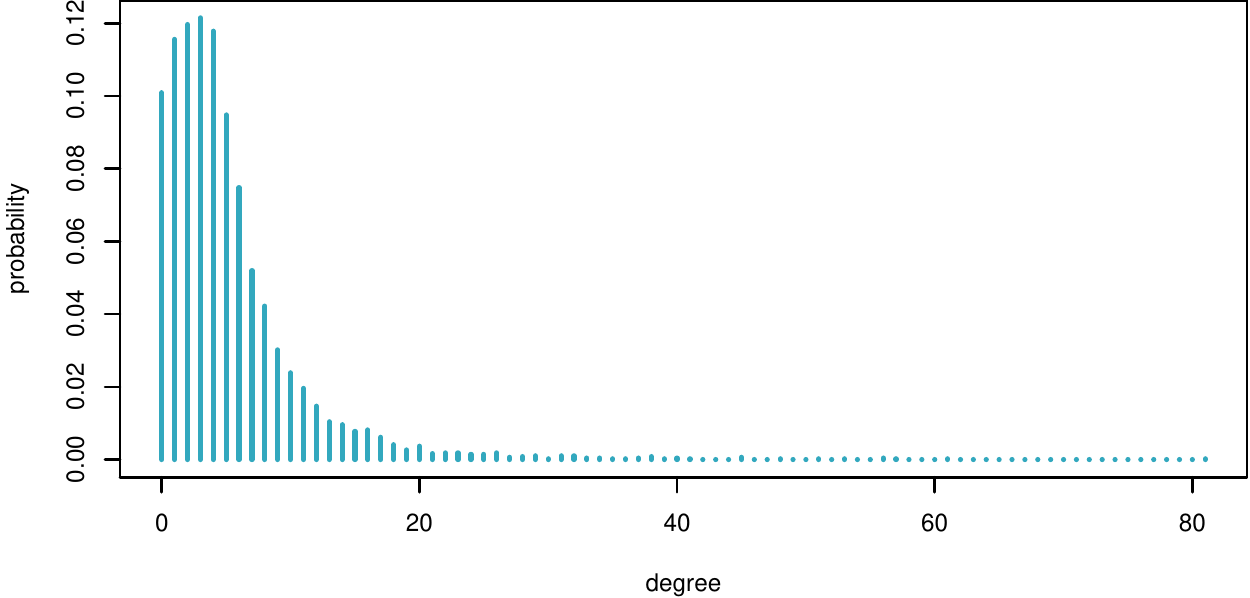} \\
(b) & \includegraphics[width=8cm]{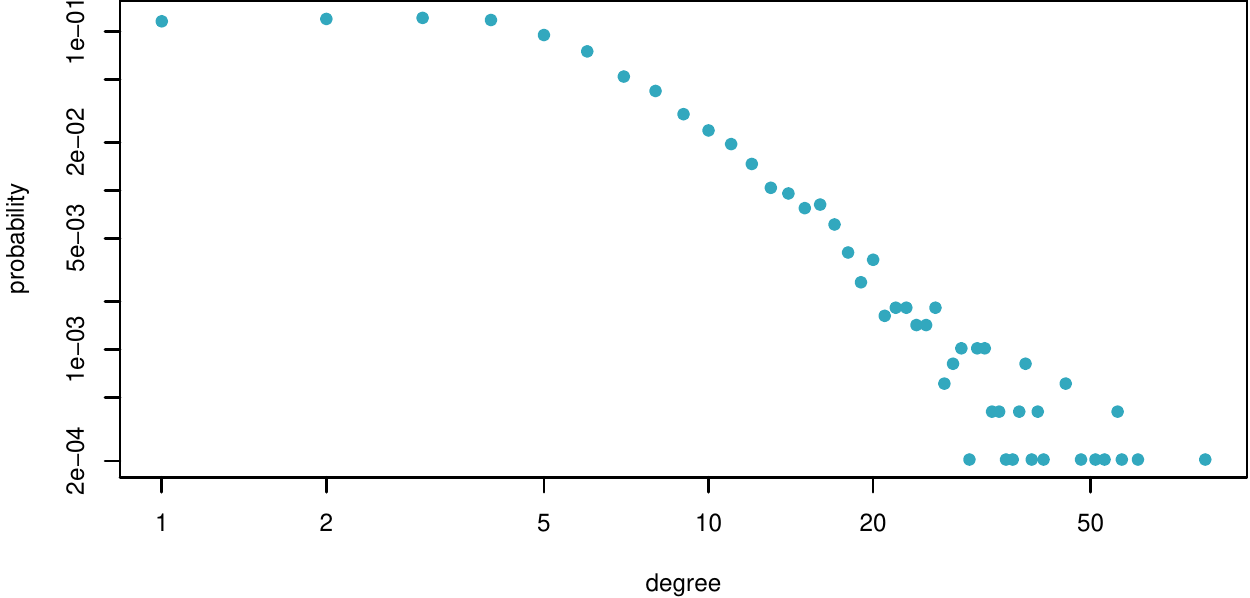}
\end{tabular}
\caption{{\small \textit{(a) Distribution of the declared number of sexual partners for the HIV+ individuals detected and present in the Cuban database. (b) Preceding degree distribution plotted in a log-log scale: the graph exhibits a power-law behaviour. }}}
\label{fig:deg}
\end{center}
\end{figure}


The degree distributions of most real-world networks, referred to as scale-free networks, often exhibit a power-law behaviour in their right tails (see \cite{durrettIII}), i.e.\  \begin{equation*}\label{power_law}
 p_k\sim k^{-\alpha}, \text{ as } k \text{ becomes large},
 \end{equation*}
 for some exponent $\alpha> 1$ (notice that $\sum_{k=1}^{\infty}1/k^{\alpha}<\infty$ in this case). Roughly speaking, this describes the situations where the majority of vertices have few connections, but a small fraction of the vertices are highly connected (e.g.\ Chapter 4 in \cite{NBW06III} for further details). We propose to fit a power-law exponent and consider two methods for this purpose, see also \cite{CSN09III}. First, we  minimize, over $\alpha>1$, the following measure of dissimilarity between the observed degree distribution and the power-law distribution with exponent $\alpha$ based on degree values larger than $k_0$
 \begin{equation}\label{dissimilarity}
\mathcal{K}_{k_0}(p, \alpha)=\sum_{k\geq k_0}\frac{p_k}{c_{p,k_0}}\log\left(\frac{C_{\alpha}\cdot p_k}{c_{p,k_0}\cdot k^{-\alpha}}\right),
\end{equation}
where $\log$ denotes the natural logarithm, $c_{p,k_0}=\sum_{k\geq k_0}p_k$ and $C_{\alpha}=\sum_{k\geq k_0}1/k^{\alpha}$.
Notice that, when $k_0$ is larger than the maximum observed degree distribution $k_{\max}$, we have $\mathcal{K}_{k_0}(p, \alpha)=0$ no matter the exponent $\alpha$. Also, the computation of \eqref{dissimilarity} involves summing a finite number of terms only, since the empirical frequency $p_k$ is equal to zero for any degree $k$ sufficiently large. The criterion $\mathcal{K}_{k_0}(p, \alpha)$ is known as the Kullback--Leibler divergence between the empirical and theoretical conditional distributions given that the degree is larger than $k_0$. Incidentally, we point out that other dissimilarity measures could be considered for the purpose of fitting a power-law, such as the $\chi^2$-distance for instance. For a fixed threshold $k_0\geq 1$, it is natural to select the value of the power-law exponent that provides the best fit, that is:
\begin{equation*}\label{min_contrast}
\widehat{\alpha}_{k_0}=\arg\min_{\alpha>1}\mathcal{K}_{k_0}(p, \alpha).
\end{equation*}
Choosing $k_0$ precisely being a challenging question to statisticians. Following in the footsteps of the heuristic selection procedures proposed in the context of heavy-tailed continuous distributions (see Chapter 4 in \cite{ResnickIII}), when possible, we suggest to choose $\widehat{\alpha}_{k_0}$ with $k_0$ in a region where the graph $\{(k,\; \widehat{\alpha}_{k}):\; k=1,\;\ldots, \; k_{\max}\}$ is becoming horizontal, or at least shows an inflexion point. For completeness, we also compute the Hill estimator:
\begin{equation*}
\widetilde{\alpha}_{m}=\left( \frac{1}{m}\sum_{j=1}^{m}\frac{k_{(j)}}{k_{(m)}} \right)^{-1},\label{hill}
\end{equation*}
where $n$ is the number of vertices of the graph under study, $1\leq m\leq n$ and $k_{(1)}=k_{\max},\; k_{(2)},\; \ldots,\; k_{(m)}$ denote the $m$ largest observed degrees sorted in decreasing order of their magnitude. The tuning parameter $m$ is selected graphically, by plotting the graph $\{(m,\widetilde{\alpha}_m):\; m=1,\; \ldots,\; n\}$. In the case when the degrees of the vertices of the graph are independent, as for the configuration model \cite{newmanstrogatzwattsIII}, this statistic can be viewed as a conditional maximum likelihood estimator and arguments based on asymptotic theory supports its pertinence in this situation, see \cite{HillIII}.\\

Let us consider the declared degree distribution in the Cuban database (see Fig. \ref{fig:deg}). Among the 5,389 individuals appearing in the database, 483 declared no sexual partners during this period. Degree distributions for the whole population exhibit a clear power-law behaviour. Power laws are fitted to the declared degree distributions, for the whole population and for the strata defined by the variable gender/sexual orientation respectively. Both methods present similar results. 
The resulting estimates (see Table \ref{tab:deg}) reveal the thickness of the upper tails: the smaller the tail exponent $\alpha$, the heavier the distribution tail. Women correspond to the heaviest tail, followed by MSM and heterosexual men. However, an ANOVA reveals no statistically significant impact of the covariates gender/sexual orientation. 
All the same, using the observed degree distribution, we obtain $(k_0,\; \alpha)=(3,\; 2.99)$ which is very close to the result when using the number of neighbours having been detected positive. \\
All the tail exponent estimates are below the critical value of $\alpha_c=3.4788$, below which a giant component exists in scale-free networks generated by means of the configuration model, and above the value $2$, below which the whole graph reduces to the giant component with probability one (see \cite{molloyreedIII,newman_SIAMIII}).\\

\begin{table}[!ht]
\begin{center}
\begin{tabular}{|c|cc|cccc|}
\hline
 & $\widehat{k}_0$  & $\widehat{\alpha}_{k_0}$  & Mean & Std dev. & Min & Max \\
 \hline
Whole population & 7 & 3.06  & 6.17 & 5.54 & 1 & 82 \\
\hline
Women & 6 & 2.71 &  5.88 &  5.03 &  1 &  39\\
Heterosexual men & 7 & 3.36 & 4.98 & 4.11 & 1 & 30\\
MSM & 7 & 3.02 & 6.43 & 5.84 & 1 & 82\\
\hline
\end{tabular}
\caption{Degree distribution for the Cuban HIV+ network, for the whole population and by sexual orientation.}
\label{tab:deg}
\end{center}
\end{table}

For completeness, we can also compare with the Hill estimator (\ref{hill}) to the estimator (\ref{min_contrast}) in each case, obtained by plotting the curves $(m,\widetilde{\alpha}_m)$ in Fig. \ref{fig:hill}: reassuringly, we found that both estimation methods yield similar results.

\begin{figure}[!ht]
     \centering
     \includegraphics[width=11cm, height=10cm]{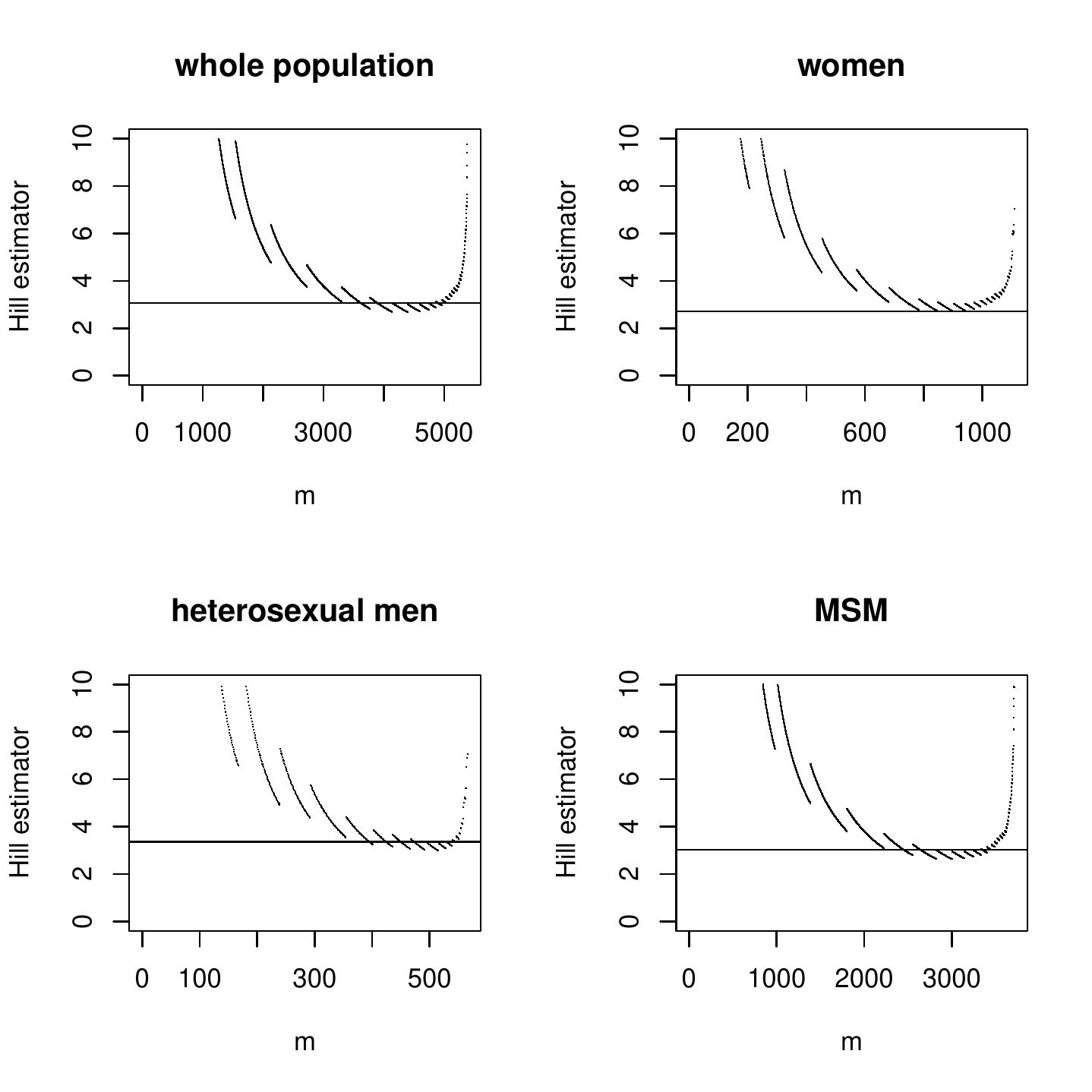}
     \caption{{\small\textit{Graph of $(m,\widetilde{\alpha}_m)$ for $m\in \{1,\dots,n\}$. This graph allows us to choose the Hill estimator. The horizontal line $y=\widehat{\alpha}_{k_0}$ permits to visualize the estimator $\widehat{\alpha}_{k_0}$ and compare it with the Hill estimator.}}}
     \label{fig:hill}
\end{figure}


\subsection{Joint degree distribution of sexual partners.}

The independence assumption between the degrees of adjacent vertices does not hold here, see Fig. \ref{fig:joint_distrib_degree}, in contrast to what is assumed for the vast majority of graph-based SIR models of epidemic disease, e.g.\ \cite{durrettIII,newman_SIAMIII}. Indeed, the linear correlation coefficient between the degree distributions of alters and egos is equal to $0.68$. Testing the significance of this coefficient, that describes the correlation of these degree distributions, allows us to test the independence of the latter. Independence between the degree distributions of alters and egos is rejected by a $\chi^2$-test with a p-value of $6.85\, 10^{-6}$. In particular, highly connected vertices tend to be connected to vertices with a high number of connections too. 
From the perspective of mathematical modeling, this suggests to consider graph models with a dependence structure between the degrees of adjacent nodes, in opposition to most percolation processes on (configuration model) networks used to describe the spread of epidemics \cite{molloyreedIII, ballnealIII, volzIII, decreusefonddhersinmoyaltranIII, grahamhouseIII}. However, it is worth noticing that, if we restrict our analysis to some specific, more homogeneous, subgroups, the independence assumption may be grounded in evidence. So if assumptions such that the network is generated by a configuration model do not hold globally, they may be valid for smaller clusters, which is another motivation for clustering.
\begin{figure}[!ht]
\begin{center}
      \includegraphics[width=6cm]{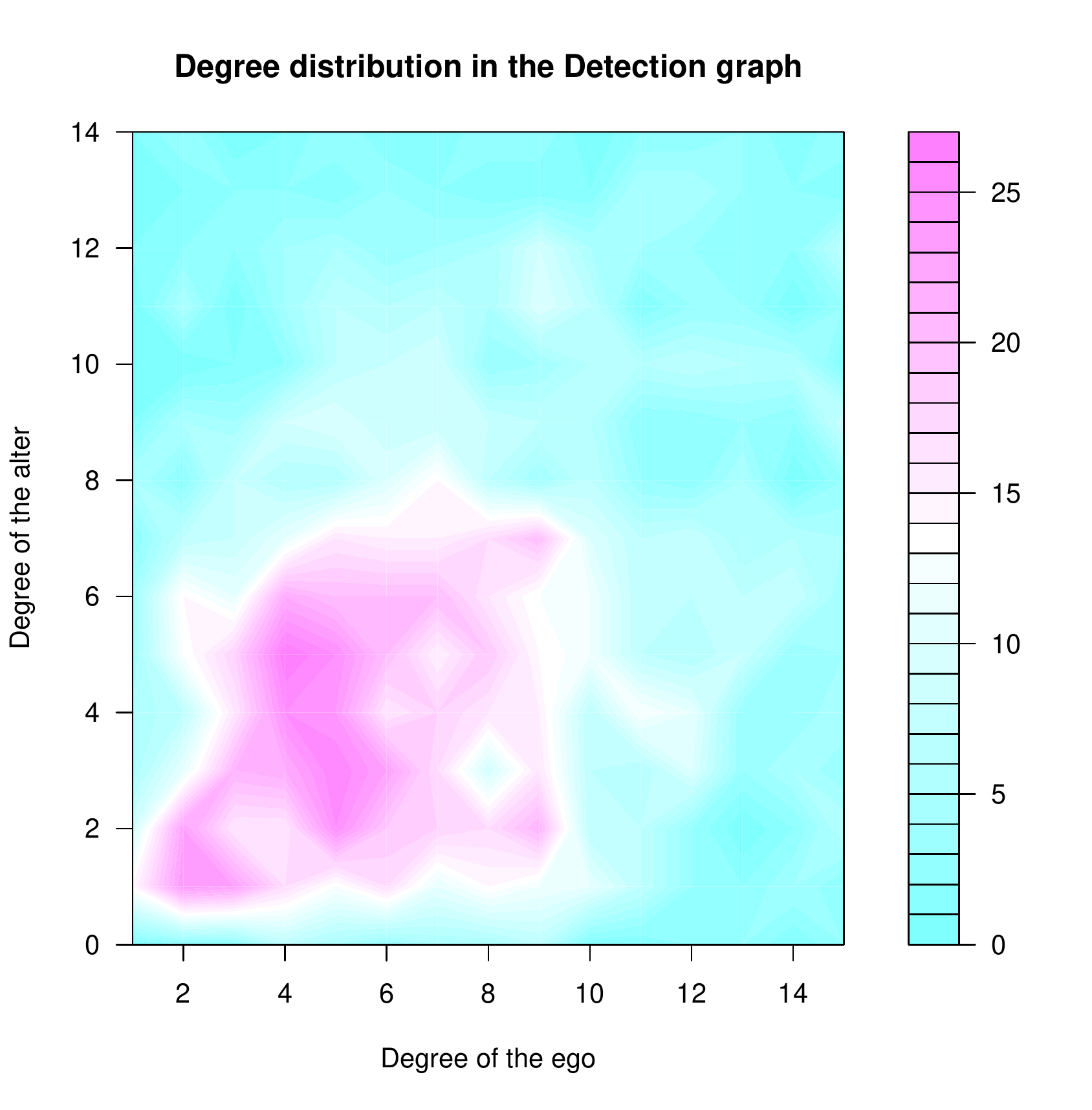}\\
     \caption{{\small\textit{Joint degree distribution of the number of contacts for connected vertices.}}}
     \label{fig:joint_distrib_degree}
\end{center}
\end{figure}

\subsection{Computation of geodesic distances and other connectivity properties} 

There is a large literature on describing the social networks on which epidemics might propagate (see Newman \cite{newman_SIAMIII} for a more exhaustive list of descriptive statistics, and \cite{clemenconarazozarossitranIII,clemenconarazozarossitranESMIII} for an application to the Cuban HIV epidemics). Here, we mention some of them, related to community and connexity. All results presented here are obtained with the R-package \verb"igraph" \cite{igraphIII}.\\

A set of connected vertices with the corresponding edges, constitutes a component of the graph. The collection of components forms a partition of the graph. We identify the components of the network and compute their respective sizes. When the size of the largest component is much larger than the size of the second largest component, see section IV A in \cite{newman_SIAMIII} and the references therein, one then refers to the notion of giant component.

A geodesic path between two connected vertices $x$ and $y$ is a path with shortest length that connects them, its length $d(x,y)$ being the geodesic distance between $x$ and $y$. One also defines the mean geodesic distance:
\begin{equation*}\label{eq:dist}
\mathcal{L}=\frac{1}{n(n+1)}\sum_{(x,y)\in \mathcal{V}^2}d(x,y),
\end{equation*}
where $\mathcal{V}$ denotes the set of all vertices of the connected graph and $n$ its size. For non-connected graphs, one usually computes a harmonic average. 
Mean geodesic distances measure ``how far" two randomly chosen vertices are, given the network structure. When $\mathcal{L}$ is much smaller than $n$, one says that a ``small-world effect" is observed. In this regard, the diameter of a connected graph, that is to say the length of the longest geodesic path, is also a quantity of major interest:
\begin{equation*} \label{eq:delta}
\delta=\max_{(x,y)\in \mathcal{V}^2}d(x,y).
\end{equation*}
Computations have been made for each component of the network of sexual contacts among individuals diagnosed as HIV positive before 2006 in Cuba, using the dedicated ``burning algorithm" for the mean geodesic distances, see \cite{AMO93III}.\\

Along these lines, we also investigate how the connectivity properties of the network evolve when removing various fractions of specific strata of the population: we studied the resilience to various strata (robustness of certain statistics such as mean geodesic distance or size of the largest component to deletion of points in these strata), the clustering coefficients (defined as the number of triangles over the number of connected triples of vertices) and the articulation points (points that disconnect the component they belong to into two components when removed; see Section 6 of \cite{clemenconarazozarossitranESMIII}). Indicators show an apparent weak resilience: 1,157 articulation points (out of 2,386 nodes), only 187 cliques (among them 177 triangles) and low assortative mixing coefficients. Global statistics thus indicate a low density of the graph (many articulation points, resilient structure, low clustering coefficients), the clustering emphasised the  important heterogeneity in the network, with some dense regions that are internally more connected than average and with few links to the outside. We found subgroups with atypical covariate distributions, each reflecting a different stage of the evolution of the epidemic. Clustering the graph also allows us to unfold the complex structure of the Cuban HIV contact-tracing network. As a byproduct, the clustering indicates sub-structures that may be considered as random graphs resulting from configuration models, bridging the gap between the modelling papers whose assumptions on network structures do not often match reality.

\setcounter{chapter}{1}
\setcounter{section}{0}
\renewcommand{\thechapter}{\Alph{chapter}}
\chapter*{Appendix: Finite Measures on $\Z_+$}
\addcontentsline{toc}{chapter}{Appendix: Finite Measures on $\Z_+$}

First, some notation is needed in order to clarify the way the atoms of a given element of $\M_F(\Z_+)$ are ranked.
For all $\mu\in\M_F(\Z_+)$, let $F_\mu$ be its cumulative distribution function and $F_\mu^{-1}$ be its right inverse defined as
\begin{equation}
\forall x\in \R_+,\, F^{-1}_{\mu}(x)=\inf\{i\in \Z_+,\, F_{\mu}(i)\geq x\}.\label{etape7}
\end{equation}
Let $\mu=\sum_{n\in\Z_+}a_n\delta_{n}$ be an integer-valued measure of $\M_F(\Z_+)$, i.e.\  such that the $a_n$'s are  themselves integers. Then, for each atom
$n\in\Z_+$ of $\mu$ such that $a_n>0$, we duplicate  the atom $n$ with
multiplicity $a_n$, and we rank the atoms of $\mu$ by increasing values, sorting arbitrarily the atoms having the same value. Then, we denote for any $i \le \cro{\mu,\mathbf 1}$,
\begin{equation}
\label{eq:numeromu}
\gamma_i(\mu)=F_{\mu}^{-1}(i),
\end{equation}
the level of the $i\th$ atom of the measure, when ranked as described
above. We refer to Example \ref{exempleFmu} for a simple illustration.


We now make precise a few topological properties of spaces of measures and measure-valued processes.
For $T>0$ and a Polish space $(E,d_E)$, we denote by $\D([0,T],E)$ the Skorokhod space of c\`{a}dl\`{a}g (right-continuous and left-limited) functions from $[0,T]$ into $E$ (e.g.\ \cite{billingsleyIII, joffemetivierIII}) equipped with the Skorokhod topology induced by the metric
\begin{equation}\label{distSkorokhod}
d_T(f,g):=\inf_{\alpha\in \Delta([0,T])} \left\{\sup_{\substack{(s,t)\in [0,T]^2,\\ s \ne t}} \left | \log\frac{\alpha(s)-\alpha(t)}{s-t}\right| + \sup_{t\le T} d_E\big( f(t),g(\alpha(t)) \big) \right\},
\end{equation}
where the infimum is taken over the set $\Delta([0,T])$ of continuous increasing functions $\alpha: [0,T] \to [0,T]$ such that $\alpha(0)=0$ and $\alpha(T)=T$.\\

Limit theorems are heavily dependent on the topologies considered.
 We introduce here several technical lemmas on the space of measures
 related to these questions. For any fixed $0 \le \varepsilon < A$, recall the definition of $\mathcal{M}_{\varepsilon,A}$ in \eqref{eq:defM}.
Note that for any $\nu\in\M_{\varepsilon,\, A}$, and $i\in \{0,\dots,5\}$,
$\cro{\nu,\chi^i}\leq A$ since the support of $\nu$ is included in
$\Z_+$.
\begin{lemma}
  \label{lem:uniform_integrability}
Let ${\mathfrak I}$ be an arbitrary set and consider a family $(\nu_\tau,\tau \in {\mathfrak
  I})$ of elements of $\M_{\varepsilon,\, A}$. Then, for any real-valued function $f$ on $\Z_+$ such that
$f(k)=o(k^5)$, we have that
\begin{equation*}
  \lim_{K\to \infty} \sup_{\tau\in{\mathfrak I}}|\langle \nu_\tau,\, f\ind_{[K,\infty)}\rangle|=0.
\end{equation*}
\end{lemma}
\begin{proof}
  By Markov inequality, for any $\tau\in{\mathfrak I},$ for any $K$,
  we have
  \begin{equation*}
    \sum_{k\ge K} |f(k)|\nu_\tau(k)\le A \sup_{k\ge K} \frac{|f(k)|}{k^5},
  \end{equation*}
hence
\begin{equation*}
  \lim_{K\to \infty} \sup_{\tau\in{\mathfrak I}}|\langle \nu_\tau,\,
  f\rangle|\le A \limsup_{k\to \infty} \frac{|f(k)|}{k^5}=0.
\end{equation*}
The proof is thus complete.
\end{proof}
\begin{lemma}
  \label{lemma:propM}
  For any $A>0$, the set $\M_{\varepsilon,\, A}$ is a closed subset of $\M_F(\Z_+)$
  embedded with the topology of weak convergence.
\end{lemma}
\begin{proof}
  Let $(\mu_n)_{n\in \N}$ be a sequence of $\M_{\varepsilon,\, A}$ converging to $\mu
  \in \M_F(\Z_+)$ for the weak topology, 
  Fatou's lemma for sequences of measures implies
  \begin{equation*}
    \cro{\mu,\chi^5} \leq 
    \lim\inf_{n\rightarrow\infty}\cro{\mu_n,\chi^5}.
  \end{equation*}
Since $\cro{\mu_n,\mathbf 1}$ tends to $\cro{\mu,\mathbf 1}$, we have that $\cro{\mu,\ind + \chi^5} \le A$.

Furthermore, by uniform integrability (Lemma \ref{lem:uniform_integrability}), it is also clear that
\begin{equation*}
  \varepsilon \le \lim_{n\to \infty} \cro{\mu_n,\, \chi}=\cro{\mu,\, \chi},
\end{equation*}
which shows that $\mu \in \M_{\varepsilon,\, A}$.
\end{proof}

\begin{lemma}
  \label{lemma:total_variation}
 The traces on $\M_{\varepsilon,\, A}$ of the total variation topology and of the weak
  topology coincide.
\end{lemma}
\begin{proof}
 It is well known that the total variation topology is coarser than
  the weak topology. In the reverse direction, assume that
  $(\mu_n)_{n\in \N }$ is a sequence of weakly converging measures all
  belonging to $\M_{\varepsilon,\, A}$. Since,
  \begin{equation*}
    d_{TV}(\mu_n,\, \mu)\le  \sum_{k\in \Z_+} |\mu_n(k)-\mu(k)|.
  \end{equation*}
according to Lemma \ref{lem:uniform_integrability}, it is then easily
deduced that the right-hand side converges to $0$ as $n$ goes to infinity.
\end{proof}

\begin{lemma}\label{lemme:convfnonbornee}If the sequence
  $(\mu_n)_{n\in \N}$ of $\M_{\varepsilon,\, A}^\N$ converges weakly to the measure
  $\mu\in \M_{\varepsilon,\, A}$, then $(\langle \mu_n,f\rangle)_{n\in \N}$ converges
  to $\langle \mu,f\rangle$ for all function $f$ such that $f(k)=o(k^5)$ for all large $k$.
\end{lemma}

\begin{proof} The triangle inequality implies that:
\begin{multline*}
|\langle \mu_n,f\rangle - \langle \mu,f\rangle|\leq |\langle \mu_n,f \ind_{[0,K]}\rangle - \langle \mu,f\ind_{[0,K]}\rangle|
+|\langle \mu,f\ind_{(K,+\infty)}\rangle|+|\langle \mu_n,f\ind_{(K,+\infty)}\rangle|.
\end{multline*}
We then conclude by uniform integrability and weak convergence.
\end{proof}
Recall that $\M_{\varepsilon,\, A}$ can be equipped with the total variation distance
topology, hence the topology on $\D([0,T],\M_{\varepsilon,\, A})$ is induced by the
distance
\begin{equation*}
  \rho_T(\mu_.,\, \nu_.)=\inf_{\alpha \in
    \Delta([0,T])}\biggl(\sup_{\substack{(s,t)\in [0,T]^2,\\ s \ne
      t}} \left | \log\frac{\alpha(s)-\alpha(t)}{s-t}\right|
  +\sup_{t\le T} d_{TV}(\mu_t,\, \nu_{\alpha(t)})\biggr).
\end{equation*}

\newpage



\end{document}